\input ./style/arxiv-general.cfg
\documentclass[aop,MSNbibl,seceqn,dvips]{arximspdf}
\makeatletter
   \@ifpackageloaded{graphicx}{}{\usepackage{graphicx}}
\makeatother

\doi{10.1214/14-AOP955} 
\volume{43}
\issue{6}
\pubyear{2015}
\firstpage{3052}
\lastpage{3132}
\docsubty{FLA}

\makeatletter
\def\mid{|}
\newcommand{\rrvert}{\vert}
\newcommand{\llvert}{\vert}
\newcommand{\new}[1]{#1}
\newcommand{\old}[2]{}
\newcommand{\eqref}[1]{(\ref{#1})}
\newtheorem{theorem}{Theorem}[section]
\newtheorem{lemma}[theorem]{Lemma}
\newproclaim{definition}[theorem]{Definition}
\newtheorem{proposition}[theorem]{Proposition}
\newtheorem{conjecture}[theorem]{Conjecture}
\newtheorem{corollary}[theorem]{Corollary}
\newproclaim{assumptions}[theorem]{Assumptions}

\newproclaim{remark}{Remark}
\newproclaim{example}{Example}
\newcommand{\F}{\mathcal{F}}
\newcommand{\A}{\mathcal{A}}
\newcommand{\dd}{\Delta_s}

\newcommand{\GUE}{\mathbb{GUE}}

\newcommand{\GTs}{{\widehat{\mathbb{GT}}}}

\newcommand{\GT}{{{\mathbb{GT}}}}

\newcommand{\eps}{\varepsilon}

\newcommand{\ii}{\mathbf{ i}}

\newcommand{\X}{\mathfrak{X}}
\newcommand{\Q}{\mathcal{Q}}

\newcommand{\one}{\mathbf{1}}
\newcommand{\zero}{\mathbf{0}}
\newcommand{\mone}{\mathbf{-1}}
\makeatother

\begin{document}
\begin{frontmatter}

\title{Asymptotics of symmetric polynomials with applications
to statistical mechanics and representation theory}
\runtitle{Asymptotics of symmetric polynomials}

\begin{aug}
\author[A]{\fnms{Vadim} \snm{Gorin}\ead[label=e1]{vadicgor@gmail.com}\thanksref{T1}}
\and
\author[B]{\fnms{Greta} \snm{Panova}\corref{}\ead[label=e2]{greta.panova@gmail.com}\thanksref{T2}}
\thankstext{T1}{Supported in part by the RFBR-CNRS Grants 10-01-93114,
11-01-93105 and by the NSF Grant DMS-14-07562.}
\thankstext{T2}{Supported in part by the Simons Postdoctoral Fellowship
at UCLA.}

\address[A]{Department of Mathematics\\
Massachusetts Institute of Technology\\
77 Massachusetts Ave.\\
Cambridge, Massachusetts 02139\\
USA\\
and\\
Institute for Information Transmission Problems\\
Bolshoy Karetny per. 19\\
Moscow 127994\\
Russia\\
\printead{e1}}
\affiliation{Massachusetts Institute of Technology and Institute for
Information Transmission Problems of Russian Academy of
Sciences, and University of Pennsylvania}

\address[B]{Mathematics Department\\
University of Pennsylvania\\
209 South 33rd St.\\
Philadelphia, PA 19104\\
USA\\
\printead{e2}}
\runauthor{V. Gorin and G. Panova}
\end{aug}

\received{\smonth{7} \syear{2013}}
\revised{\smonth{7} \syear{2014}}

%
\begin{abstract}
We develop a new method for studying the asymptotics of symmetric
polynomials of
representation-theoretic origin as the number of variables tends to
infinity. Several applications of our method
are presented: We prove a number of theorems concerning characters of
infinite-dimensional
unitary group and their $q$-deformations. We study the behavior of
uniformly random lozenge
tilings of large polygonal domains and find the GUE-eigenvalues
distribution in the limit. We
also investigate similar behavior for alternating sign matrices
(equivalently, six-vertex model
with domain wall boundary conditions). Finally, we compute the
asymptotic expansion of certain
observables in $O(n=1)$ dense loop model.
\end{abstract}

%
\begin{keyword}[class=AMS]
\kwd{60K35}
\kwd{22E99}
\kwd{05E05}
\kwd{60F99}
\end{keyword}
\begin{keyword}
\kwd{Symmetric polynomials}
\kwd{Schur function}
\kwd{lozenge tilings}
\kwd{GUE}
\kwd{ASM}
\kwd{6 vertex model}
\kwd{dense loop model}
\kwd{extreme characters of $U(\infty)$}
\end{keyword}
\end{frontmatter}

\setattribute{tocline}{skip}{\space}
\tableofcontents[alignleft,level=2]

\section{Introduction}\label{sec1}

\subsection{Overview}

In this article we study the asymptotic behavior of symmetric functions of
representation-theoretic origin, such as Schur rational functions or
characters of symplectic or
orthogonal groups, etcetera, as their number of variables tends to
infinity. In order to simplify the
exposition we stick to Schur functions in the \hyperref
[sec1]{Introduction} where it is
possible, but most of our
results hold in a greater generality.

The rational Schur function $s_\lambda(x_1,\ldots,x_n)$ is a symmetric
Laurent polynomial in
variables $x_1$, \ldots, $x_n$. They are parameterized by $N$-tuples of
integers $\lambda=(\lambda_1\ge
\lambda_2\ge\cdots\ge\lambda_N)$ (we call such $N$-tuples \emph
{signatures}, they form the set
$\mathbb{GT}_N$) and are given by Weyl's character formula as
\[
s_\lambda(x_1,\ldots,x_N) = \frac{\det [x_i^{\lambda
_j+N-j}
]_{i,j=1}^{N}}{\prod_{i<j}(x_i-x_j)}.
\]
Our aim is to study the asymptotic behavior of the \emph{normalized}
symmetric polynomials
%
\begin{equation}
\label{eq_Normalized_Schur} S_{\lambda}(x_1,\ldots,x_k;N,1)=
\frac{s_\lambda(x_1,\ldots
,x_k,\overbrace{1,\ldots,1}^{N-k})}{s_\lambda(\underbrace{1,\ldots,1}_N)}
\end{equation}
and also
%
\begin{equation}
\label{eq_Normalized_q_Schur} S_{\lambda}(x_1,\ldots,x_k;N,q)=
\frac{s_\lambda(x_1,\ldots
,x_k,1,q,q^2,\ldots,q^{N-k-1})}{s_\lambda(1,\ldots,q^{N-1})},
\end{equation}
for some $q>0$. Here $\lambda=\lambda(N)$ is allowed to vary with $N$,
$k$ is any fixed number
and $x_1,\ldots,x_k$ are complex numbers, which may or may not vary
together with $N$, depending on the
context. Note that there are explicit expressions (Weyl's dimension
formulas) for the denominators in formulas
\eqref{eq_Normalized_Schur} and \eqref{eq_Normalized_q_Schur}.
Therefore, their asymptotic
behavior is straightforward.

The asymptotic analysis of expressions \eqref{eq_Normalized_Schur},
\eqref{eq_Normalized_q_Schur}
is important because of the various applications in representation
theory, statistical mechanics
and probability, including:
\begin{itemize}
\item
For any $k$ and any \emph{fixed} $x_1,\ldots, x_k$, such that $|x_i|=1$,
the convergence of
$S_{\lambda}(x_1,\ldots,x_k;N,1)$ [from \eqref{eq_Normalized_Schur}] to
some limit and the
identification of this limit can be put in representation-theoretic
framework as the
approximation of indecomposable characters of the infinite-dimensional
unitary group $U(\infty)$
by normalized characters of the unitary groups $U(N)$; the latter
problem was first studied by
Vershik and Kerov \cite{VK}.

\item The convergence of $S_{\lambda}(x_1,\ldots,x_k;N,q)$ [from
\eqref
{eq_Normalized_q_Schur}] for any $k$ and any \emph{fixed} $x_1,\ldots
, x_k$
is similarly related to the \emph{quantization} of characters of
$U(\infty)$; see~\cite{G-A}.
\item The asymptotic behavior of \eqref{eq_Normalized_Schur} can be put
in the context of random matrix theory
as the study of the Harish-Chandra--Itzykson--Zuber integral
%
\begin{equation}
\label{eq_HC_intro1} \int_{U(N)} \exp\bigl(\operatorname{Trace}
\bigl(AUBU^{-1}\bigr)\bigr) \,dU,
\end{equation}
where $A$ is a fixed Hermitian matrix of finite rank, and $B=B(N)$ is
an $N\times N$ matrix changing in a
regular way as $N\to\infty$. In this formulation the problem was
thoroughly studied by Guionnet and
Ma\"\i da \cite{GM}.

\item A normalized Schur function \eqref{eq_Normalized_Schur} can be
interpreted as the expectation of a certain observable in
the probabilistic model of uniformly random lozenge tilings of planar
domains. The asymptotic
analysis of \eqref{eq_Normalized_Schur} as $N\to\infty$ with
$x_i=\exp
(y_i/\sqrt{N})$ and fixed
$y_i$s gives a way to prove the local convergence of random tiling to a
distribution of random
matrix origin, the GUE-corners process (the name \textit{GUE-minors
process} is also used). Informal
argument explaining that such convergence should hold was suggested
earlier by Okounkov and
Reshetikhin in \cite{OR-birth}.

\item When $\lambda$ is a \emph{staircase Young diagram} with $2N$ rows
of lengths ${N-1},{N-1},\break  {N-2},
{N-2}, \ldots, 1, 1, 0, 0$, \eqref{eq_Normalized_Schur} gives the
expectation of an
observable (closely related to the Fourier transform of the number of
vertices of type $a$ on a
given row) for the uniformly random configurations of the six-vertex
model with domain wall
boundary conditions (equivalently, alternating sign matrices).
Asymptotic behavior as $N\to\infty$
with $x_i=\exp(y_i/\sqrt{N})$ and fixed $y_i$ gives a way to study the
local limit of the
six-vertex model with domain wall boundary conditions near the boundary.

\item For the same staircase $\lambda$ the expression involving \eqref
{eq_Normalized_Schur} with
$k=4$ and Schur polynomials replaced by the characters of symplectic
group gives the mean of the
boundary-to-boundary current for the completely packed $O(n=1)$ dense
loop model; see \cite{GNP}. The
asymptotics (now with fixed $x_i$, not depending on $N$) gives the
limit behavior of this current,
significant for the understanding of this model.
\end{itemize}

In the present article we develop a new unified approach to study the
asymptotics of normalized
Schur functions \eqref{eq_Normalized_Schur}, \eqref
{eq_Normalized_q_Schur} (and also for more
general symmetric functions like symplectic characters and polynomials
corresponding to the root
system $BC_n$), which gives a way to answer all of the above limit
questions. There are 3 main
ingredients of our method:

\begin{longlist}[(1)]
\item[(1)] We find simple contour integral representations for the
normalized Schur polynomials
\eqref{eq_Normalized_Schur}, \eqref{eq_Normalized_q_Schur} with $k=1$,
that is, for
%
\begin{equation}
\label{eq_Schur_norm_1} \frac{s_\lambda(x,1,\ldots,1)}{s_\lambda(1,\ldots,1)}
\quad\mbox{and}\quad \frac{s_\lambda(x,1,q,\ldots,q^{N-2})}{s_\lambda(1,\ldots,q^{N-1})},
\end{equation}
and also for more general symmetric functions of
representation-theoretic origin.
\item[(2)] We study the asymptotics of the above contour integrals
using the \emph{steepest descent}
method.

\item[(3)] We find formulas expressing \eqref{eq_Normalized_Schur},
\eqref{eq_Normalized_q_Schur}
as $k\times k$ determinants of expressions involving \eqref
{eq_Schur_norm_1}, and combining the
asymptotics of these formulas with asymptotics of \eqref
{eq_Schur_norm_1} compute limits of
\eqref{eq_Normalized_Schur}, \eqref{eq_Normalized_q_Schur}.
\end{longlist}

In the rest of the \hyperref[sec1]{Introduction} we provide a more
detailed description
of our results. In Section~\ref{Section_intro_method} we briefly
explain our methods. In Sections~\ref{Section_intro_RT}--\ref{Section_intro_matrix}, we describe the
applications of our method in asymptotic
representation theory, probability and statistical mechanics. Finally,
in Section~\ref{Section_intro_compare}
we compare our approach for studying the asymptotics of symmetric
functions with other known methods.

In the next papers we also apply the techniques developed here to the
study of other classes of
lozenge tilings \cite{Pa} and to the investigation of the asymptotic
behavior of decompositions of
tensor products of representations of classical Lie groups into
irreducible components \cite{BuG}.

\subsection{Our method}
\label{Section_intro_method}

The main ingredient of our approach to the asymptotic analysis of
symmetric functions is the
following integral formula, which is proved in Theorem~\ref{Theorem_Integral_representation_Schur_1}. Let
$\lambda=(\lambda_1\ge\lambda_2\ge\cdots\ge\lambda_N)$, and let $x_1,
\ldots, x_k$ be complex
numbers. Denote
\[
S_\lambda(x_1,\ldots,x_k;N,1)=
\frac{s_\lambda(x_1,\ldots
,x_k,1,\ldots
,1)}{s_\lambda(1,\ldots,1)}
\]
with $N-k$ $1$s in the numerator and $N$ $1$s in the denominator.
%
\begin{theorem}[(Theorem~\ref{Theorem_Integral_representation_Schur_1})]
\label{Theorem_single_intro} For any complex number $x$ other than $0$
and~$1$,
we have
%
\begin{equation}
\label{eq_integral_formula_intro} S_\lambda(x;N,1) = \frac
{(N-1)!}{(x-1)^{N-1} }\frac{1}{2\pi\ii}
\oint_C \frac{x^z}{\prod_{i=1}^N(z-(\lambda_i+N-i))}\,dz,
\end{equation}
where the contour $C$ encloses all the singularities of the integrand.
\end{theorem}
We also prove various
generalizations of formula \eqref{eq_integral_formula_intro}: one can
replace $1$s by the geometric
series $1,q,q^2,\ldots$ (Theorem~\ref
{Theorem_Integral_representation_Schur_q}), Schur functions
can be replaced with characters of symplectic group (Theorems \ref
{Theorem_symplectic_integral_q}
and \ref{Theorem_symplectic_integral_1}) or, more, generally, with
multivariate Jacobi polynomials
(Theorem~\ref{Theorem_Jacobi_singlevar}). In all these cases a
normalized symmetric function is
expressed as a contour integral with integrand being the product of
elementary factors. The only
exception is the most general case of Jacobi polynomials, where we have
to use certain
hypergeometric series.

Recently (and independently of the present work) a formula similar to
\eqref{eq_integral_formula_intro} for the characters of orthogonal
groups $O(n)$ was found in
\cite{HJ} in the study of the mixing time of certain random walk on
$O(n)$. A close relative of our
formula \eqref{eq_integral_formula_intro} can be also found in
Section~3 of \cite{CPZ}.

Using formula \eqref{eq_integral_formula_intro} we apply tools from
complex analysis, mainly the
method of steepest descent, to compute the limit behavior of these
normalized symmetric functions.
Our main asymptotic results along these lines are summarized in Propositions
\ref{proposition_convergence_mildest}, \ref
{proposition_convergence_strongest},
\ref{Prop_convergence_GUE_case} for real $x$ and in Propositions
\ref{proposition_convergence_extended} and \ref
{Prop_convergence_GUE_extended} for complex $x$.

The next important step is the formula expressing $S_\lambda
(x_1,\ldots
,x_k;N,1)$ in terms of
$S_\lambda(x_i;N,1)$ which is proved in Theorem~\ref
{Theorem_multivariate_Schur_1}:
%
\begin{theorem}[(Theorem~\ref{Theorem_multivariate_Schur_1})] \label
{theorem_multi_intro}We have
%
\begin{eqnarray}
\label{eq_Multivariate_intro}&& S_{\lambda}(x_1,\ldots,x_k;N,1)
\nonumber
\\
&&\qquad=
\frac{1}{\prod_{i<j}(x_i-x_j)}\\
&&\qquad\quad{}\times\prod_{i=1}^k
\frac{(N-i)! }{
(x_i-1)^{N-k} }
\det \bigl[ D_{x_i}^{k-j} \bigr]_{i,j=1}^{k}
\Biggl(\prod_{j=1}^k S_{\lambda}(x_j;N,1)
\frac{(x_j-1)^{N-1}}{(N-1)!} \Biggr),\nonumber
\end{eqnarray}
where $D_{x}$ is the differential operator $x\frac{\partial}{\partial x}$.
\end{theorem}
Formula
\eqref{eq_Multivariate_intro} can again be generalized: $1$s can be
replaced with geometric series
$1,q,q^2,\ldots$ (Theorem~\ref{Theorem_multivariate_Schur_q}), Schur
functions can be replaced with
characters of the symplectic group (Theorems \ref{theorem_simplectic_multi_q},
\ref{Theorem_symp_multivar_1}) or, more, generally, with multivariate
Jacobi polynomials (Theorem~\ref{Theorem_Jacobi_multi}). Formulas similar to \eqref{eq_Multivariate_intro}
can be found in the literature; see, for example, \cite{GP}, Proposition~6.2, \cite{KS}.

The advantage of formula \eqref{eq_Multivariate_intro} is its
relatively simple form, but it is not
straightforward that this formula is suitable for the $N\to\infty$
limit. However, we are able to
rewrite this formula in a different form (see Proposition~\ref{Proposition_multivariate_expansion}), from which this limit
transition is immediate. Combining
the limit formula with the asymptotic results for $S_\lambda(x;N,1)$ we
get the full asymptotics
for $S_\lambda(x_1,\ldots,x_k;N,1)$. As a side remark, since we deal
with analytic functions and
convergence in our formulas is always (at least locally) uniform, the
differentiation in formula
\eqref{eq_Multivariate_intro} does not introduce any problems.

Theorems \ref{Theorem_single_intro} and \ref{theorem_multi_intro} allow
us to study the asymptotic behavior of normalized Schur functions in
various settings, which are motivated by the current applications:
\begin{itemize}
\item As $\lambda_i(N)/N \to f(i/N)$ in a sufficiently regular fashion
for a monotone piecewise continuous function $f$ on $[0,1]$ (used in
the statistical mechanics applications of Section~\ref{s:stat_mech}) or
as $\lambda(N)$ grows in certain sub-linear regimes (used in the
representation theoretic applications of Section~\ref{s:rep_theory}).
\item As the variables $x_1,\ldots,x_k$ are fixed, or as they depend on
$N$, for example, $x_i=e^{y_i/\sqrt{N}}$ for fixed $y_i$ (used in
Sections~\ref{Section_GUE} and~\ref{Section_ASM}).
\end{itemize}

We believe that the combination of Theorems \ref{Theorem_single_intro},
\ref{theorem_multi_intro}
with the well-developed steepest descent method for the analysis of
complex integral, paves the
way to study the delicate asymptotics of Schur polynomials (and more
general symmetric functions
of representation-theoretic origin) in numerous limit regimes which
might go well beyond the
applications presented in this paper.

\subsection{Application: Asymptotic representation theory}
\label{Section_intro_RT}

Let $U(N)$ denote the group of all $N\times N$ unitary matrices. Embed
$U(N)$ into $U(N+1)$ as a
subgroup acting on the space spanned by first $N$ coordinate vectors
and fixing $N+1$st vector,
and form the infinite-dimensional unitary group $U(\infty)$ as an
inductive limit
\[
U(\infty)=\bigcup_{N=1}^\infty U(N).
\]
Recall that a (normalized) \emph{character} of a group $G$ is a
continuous function $\chi(g)$, $g\in
G$ satisfying:
\begin{longlist}[(1)]
\item[(1)] $\chi$ is constant on conjugacy classes, that is, $\chi
(aba^{-1})=\chi(b)$;
\item[(2)] $\chi$ is positive definite, that is, the matrix $
[\chi
(g_i g_j^{-1}) ]_{i,j=1}^k$ is
Hermitian nonnegative definite,
for any $\{g_1,\ldots, g_k\}$;
\item[(3)] $\chi(e)=1$.
\end{longlist}
An \emph{extreme character} is an extreme point of the convex set of
all characters. If $G$ is a
compact group, then its extreme characters are normalized matrix traces
of irreducible
representations. It is a known fact (see, e.g., the classical
book of Weyl~\cite{W}) that irreducible representations of the unitary
group $U(N)$ are
parameterized by signatures, and the value of the trace of the
representation parameterized by
$\lambda$ on a unitary matrix with eigenvalues $u_1,\ldots, u_N$ is
$s_\lambda(u_1,\ldots,u_N)$. Using these facts and applying the result
above to $U(N)$, we conclude that the normalized characters of $U(N)$
are the functions
\[
\frac{s_\lambda(u_1,\ldots,u_N)}{s_\lambda(1,\ldots,1)}.
\]
For ``big'' groups such as $U(\infty)$, the situation is more delicate.
The study of characters of
this group was initiated by Voiculescu \cite{Vo} in 1976 in connection
with \emph{finite
factor representations} of $U(\infty)$. Voiculescu gave a list of
extreme characters, later
independently Boyer \cite{Bo} and Vershik--Kerov \cite{VK} discovered
that the classification
theorem for the characters of $U(\infty)$ follows from the result of
Edrei~\cite{Ed} on the
characterization of totally positive Toeplitz matrices. Nowadays,
several other proofs of
Voiculescu--Edrei classification theorem is known; see \cite{OkOlsh,BO-newA,Petrov-boundary}. The theorem itself reads:
%
\begin{theorem}
\label{Theorem_Voiculescu}
The extreme characters of $U(\infty)$ are parameterized by the points
$\omega$ of the
infinite-dimensional domain
\[
\Omega\subset{\mathbb R}^{4\infty+2}={\mathbb R}^\infty\times {
\mathbb R}^\infty\times{\mathbb R}^\infty\times{\mathbb
R}^\infty\times{\mathbb R} \times{\mathbb R},
\]
where $\Omega$ is the set of sextuples
\[
\omega=\bigl(\alpha^+,\alpha^-,\beta^+,\beta^-;\delta^+,\delta^-\bigr)
\]
such that
\begin{eqnarray*}
&\displaystyle\alpha^\pm=\bigl(\alpha_1^\pm\ge
\alpha_2^\pm\ge\cdots\ge0\bigr)\in {\mathbb
R}^\infty,\qquad \beta^\pm=\bigl(\beta_1^\pm
\ge\beta_2^\pm\ge\cdots\ge0\bigr)\in {\mathbb
R}^\infty,&
\\
&\displaystyle\sum_{i=1}^\infty\bigl(
\alpha_i^\pm+\beta_i^\pm\bigr)
\le\delta^\pm, \qquad\beta_1^+ +\beta_1^-\le1.&
\end{eqnarray*}
The corresponding extreme character is given by the formula
%
\begin{eqnarray}
\label{eq_characters_classical} \chi^{(\omega)}(U)&=&\prod_{u\in
\operatorname{Spectrum}(U)}
e^{\gamma^+(u-1)+\gamma^-(u^{-1}-1)}
\nonumber
\\[-8pt]
\\[-8pt]
\nonumber
&&{}\times\prod_{i=1}^{\infty}
\frac
{1+\beta
_i^+(u-1)}{1-\alpha^+_i(u-1)} \frac{1+\beta_i^-(u^{-1}-1)}{1-\alpha^-_i(u^{-1}-1)},
\end{eqnarray}
where
\[
\gamma^\pm=\delta^\pm-\sum_{i=1}^\infty
\bigl(\alpha_i^\pm+\beta _i^\pm
\bigr).
\]
\end{theorem}

Our interest in characters is based on the following fact.

\begin{proposition}
\label{prop_character_is_limit}
Every extreme normalized character $\chi$ of $U(\infty)$ is a uniform
limit of extreme
characters of $U(N)$. In other words, for every $\chi$ there exists a sequence
$\lambda(N)\in\mathbb{GT}_N$ such that for every $k$,
\[
\chi(u_1,\ldots,u_k,1,\ldots)=\lim_{N\to\infty}
S_\lambda(u_1,\ldots,u_k;N,1)
\]
uniformly on the torus $(S_1)^k$, where $S_1 =\{u\in\mathbb C\dvtx
|u|=1\}$.
\end{proposition}
In the context of representation theory of $U(\infty)$, this statement
was first observed by Vershik and Kerov \cite{VK}. However, this is just a particular case of a
very general convex
analysis theorem which was reproved many times in various contexts;
see, for example, \cite{V_ergodic,OkOlsh,DF-boundary}.

The above proposition raises the question which sequences of characters
of $U(N)$
approximate characters of $U(\infty)$. Solution to this problem was
given by Vershik and Kerov \cite{VK}.

Let $\mu$ be a Young diagram with row lengths $\mu_i$, column lengths
$\mu'_i$ and whose length of
main diagonal is $d$. Introduce \emph{modified Frobenius coordinates}
\[
p_i=\mu_i-i+1/2,\qquad q_i=
\mu'_i-i+1/2,\qquad i=1,\ldots,d.
\]
Note that $\sum_{i=1}^d p_i +q_i = |\mu|$.

Given a signature $\lambda\in\mathbb{GT}_N$, we associate two Young
diagrams $\lambda^+$ and
$\lambda^-$ to it: The row lengths of $\lambda^+$ are the positive
$\lambda_i$'s, while the row
lengths of $\lambda^-$ are minus the negative ones. In this way we get
two sets of modified
Frobenius coordinates: $p_i^+,q_i^+$, $i=1,\ldots,d^+$ and
$p_i^-,q_i^-$, $i=1,\ldots,d^-$.

\begin{theorem}[(\cite{VK,OkOlsh,BO-newA,Petrov-boundary})]
\label{Theorem_U_VK}
Let $\omega=(\alpha^\pm,\beta^\pm;\delta^\pm)$, and suppose that the
sequence $\lambda(N)\in\mathbb{GT}_N$ is such that
\begin{eqnarray*}
p_i^+(N)/N&\to&\alpha_i^+,\qquad p_i^-(N)/N\to
\alpha_i^-,\qquad q_i^+(N)/N\to\beta_i^+,\\
q_i^-(N)/N&\to&\beta_i^+,\qquad
\bigl|\lambda^+\bigr|/N\to\delta^+,\qquad \bigl|\lambda^-\bigr|/N\to\delta^-.
\end{eqnarray*}
Then for every $k$
\[
\chi^\omega(u_1,\ldots,u_k,1,\ldots)=\lim
_{N\to\infty} S_{\lambda(N)}(u_1,\ldots,u_k;N,1)
\]
uniformly on the torus $(S_1)^k$.
\end{theorem}

Theorem~\ref{Theorem_U_VK} is an immediate corollary of our results on
asymptotics of normalized
Schur polynomials, and a new short proof is given in Section~\ref
{Section_U_infty}.

Note the remarkable multiplicativity of Voiculescu--Edrei formula for
the characters of
$U(\infty)$: the value of a character on a given matrix [element of
$U(\infty)$] is expressed as a
product of the values of a single function at each of its eigenvalues.
There exists an independent
representation-theoretic argument explaining this multiplicativity.
Clearly, no such
multiplicativity exists for finite $N$, that is, for the characters of
$U(N)$. However, we claim that formula \eqref{eq_Multivariate_intro}
should be viewed as a manifestation of \emph{approximate
multiplicativity} for (normalized) characters of $U(N)$. To explain
this point of view we start
from $k=2$. In this case \eqref{eq_Multivariate_intro} simplifies to
\begin{eqnarray*}
&&S_\lambda(x,y;N,1)\\
&&\qquad= S_\lambda(x;N,1)S_\lambda(y;N,1)\\
&&\qquad\quad{}+
\frac
{(x-1)(y-1)}{N-1} \frac{ (x({\partial}/{(\partial
x)})-y({\partial}/{(\partial y)})) [S_\lambda(x;N,1)S_\lambda(y;N,1)]}{y-x}.
\end{eqnarray*}
More generally Proposition~\ref{Proposition_multivariate_expansion}
claims that for any $k$ formula,
\eqref{eq_Multivariate_intro} implies that, informally,
\[
S_\lambda(x_1,\ldots,x_k;N,1)=
S_\lambda(x_1;N,1)\cdots S_\lambda
(x_k;N,1)+O(1/N).
\]
Therefore, \eqref{eq_Multivariate_intro} states that normalized
characters of $U(N)$ are
approximately multiplicative, and they become multiplicative as $N\to
\infty$. This is somehow
similar to the work of Diaconis and Freedman \cite{DF-ex} on \emph
{finite exchangeable
sequences}. In particular, in the same way as results of \cite{DF-ex}
immediately imply de
Finetti's theorem (see, e.g., \cite{Aldous}), our results immediately
imply the multiplicativity of
characters of $U(\infty)$.

In \cite{G-A} a $q$-deformation of the notion of character of
$U(\infty
)$ was suggested.
Analogously to Proposition~\ref{prop_character_is_limit}, a
$q$-character is a limit of Schur
functions, but with different normalization. This time the sequence
$\lambda(N)$ should be such
that for every $k$,
%
\begin{equation}
\frac{s_{\lambda(N)}(x_1,\ldots,x_k,q^{-k},q^{-k-1},\ldots,q^{1-N})}{
s_{\lambda(N)}(1,q^{-1},\ldots,q^{1-N})} \label{eq_q_character_normalized}
\end{equation}
converges uniformly on the set $\{(x_1,\ldots,x_k)\in
\mathbb C^k \mid|x_i|=q^{1-i}\}$. An analogue of Theorem~\ref
{Theorem_U_VK} is the following one:

\begin{theorem}[(\cite{G-A})] \label{Theorem_U_q_approx}
Let $0<q<1$. Extreme $q$-characters of $U(\infty)$ are parameterized
by the points
of set $\mathcal N$ of all nondecreasing sequences of integers,
\[
\mathcal N=\{\nu_1\le\nu_2\le\nu_3\le\cdots
\}\subset\mathbb Z^{\infty}.
\]
Suppose that a sequence $\lambda(N)\in\mathbb{GT}_N$ is such that for
any $j>0$,
%
\begin{equation}
\label{eq_stabilization} \lim_{i\to\infty}\lambda_{N+1-j}(N)=
\nu_j,
\end{equation}
and then for every $k$,
%
\begin{equation}
\label{eq_x14} \frac{s_{\lambda(N)}(x_1,\ldots
,x_k,q^{-k},q^{-k-1},\ldots
,q^{1-N})}{
s_{\lambda(N)}(1,q^{-1},\ldots,q^{1-N})}
\end{equation}
converges uniformly on the set $\{(x_1,\ldots,x_k)\in
\mathbb C^k \mid|x_i|=q^{1-i}\}$, and these limits define the
$q$-character of $U(\infty)$.
\end{theorem}

Using the $q$-analogues of formulas \eqref{eq_integral_formula_intro} and
\eqref{eq_Multivariate_intro}, we give in Section~\ref
{Section_U_q_infty} a~short proof of the
second part of Theorem~\ref{Theorem_U_q_approx}; see Theorem~\ref
{theorem_q_limit_multivar}. This
should be compared with \cite{G-A}, where the proof of the same
statement was quite involved. We
go beyond the results of \cite{G-A}, give new formulas for the
$q$-characters and explain what
property replaces the multiplicativity of Voiculescu--Edrei characters
given in Theorem~\ref{Theorem_Voiculescu}.

\subsection{Application: Random lozenge tilings}
\label{Section_intro_lozenge}

Consider a tiling of a domain drawn on the regular triangular lattice
of the kind shown at Figure~\ref{Fig_polyg_domain} with rhombi of 3
types, where each rhombus is a
union of 2 elementary
triangles. Such rhombi are usually called \emph{lozenges} and they are
shown at Figure~\ref{Figure_loz}. The configuration of the domain is
encoded by the
number $N$ which is its width
and $N$ integers $\mu_1>\mu_2>\cdots>\mu_N$ which are the positions of
\emph{horizontal lozenges}
sticking out of the right boundary. If we write $\mu_i=\lambda_i+N-i$,
then $\lambda$ is a
signature of size $N$; see the left panel of Figure~\ref
{Fig_polyg_domain}. Due to combinatorial
constraints the tilings of such domain are in correspondence with
tilings of a certain polygonal
domain, as shown on the right panel of Figure~\ref{Fig_polyg_domain}.
Let $\Omega_\lambda$ denote
the domain encoded by a signature $\lambda$.

\begin{figure}

\includegraphics{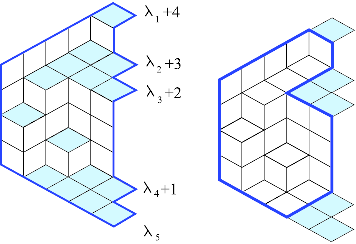}

\caption{Lozenge tiling of the domain encoded by signature $\lambda$
(left panel) and of corresponding polygonal domain (right panel).}
\label{Fig_polyg_domain}
\end{figure}

\begin{figure}[b]

\includegraphics{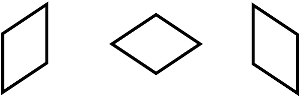}

\caption{The 3 types of lozenges; the middle one is called \textit
{horizontal}.}
\label{Figure_loz}
\end{figure}

It is well known that each lozenge tiling can be identified with
a stepped surface in $\mathbb R^3$ (the three types of lozenges
correspond to the three slopes of
this surface) and with a perfect matching of a subgraph of a hexagonal
lattice; see, for example,
\cite{Kenyon-Dimers}. Note that there are finitely many tilings of
$\Omega_\lambda$, and let
$\Upsilon_\lambda$ denote a uniformly random lozenge tiling of
$\Omega
_\lambda$. The interest in
lozenge tilings is caused by their remarkable asymptotic behavior. When
$N$ is large the rescaled
stepped surface corresponding to $\Upsilon_\lambda$ concentrates near a
deterministic limit shape.
In fact, this is true also for more general domains; see \cite{CKP}.
One feature of the limit shape
is the formation of so-called \emph{frozen regions}; in terms of
tilings, these are the regions
where asymptotically with high probability only single type of lozenges
is observed. This effect is
visualized in Figure~\ref{Figure_hex}, where a sample from the uniform
measure on tilings of the
simplest tilable domain, a hexagon, is shown. It is known that in this
case the boundary of the
frozen region is the inscribed ellipse; see~\cite{CLP}, and for more
general polygonal domains the
frozen boundary is an inscribed algebraic curve, see \cite{KO-Burgers}
and also
\cite{Petrov-curves}.

In this article we study the local behavior of lozenge tiling near a
\emph{turning point} of the
frozen boundary, which is the point where the boundary of the frozen
region touches (and is tangent
to) the boundary of the domain. Okounkov and Reshetikhin gave in \cite
{OR-birth} a nonrigorous
argument explaining that the scaling limit of a tiling in such
situation should be governed by the
\emph{GUE-corners} process (introduced and studied by Baryshnikov
\cite{Bar} and
Johansson--Nordenstam \cite{JN}), which is the joint distribution of
the eigenvalues of a Gaussian
Unitary ensemble (GUE-)random matrix (i.e., Hermitian matrix with
independent Gaussian entries)
and of its top-left corner square submatrices. In one model of tilings
of infinite polygonal
domains, the proof of the convergence can be based on the determinantal
structure of the
correlation functions of the model and on the double-integral
representation for the correlation
kernel, and it was given in \cite{OR-birth}. Another rigorous argument,
related to the asymptotics
of \emph{orthogonal polynomials} exists for the lozenge tilings of
hexagon (as in Figure~\ref{Figure_hex}); see \cite{JN,N}.

\begin{figure}

\includegraphics{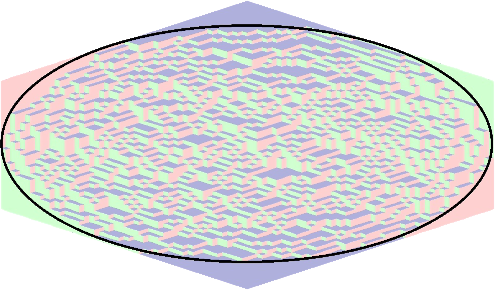}

\caption{A sample from uniform distribution on tilings of $40\times
50\times50$ hexagon and
corresponding theoretical frozen boundary. The three types of lozenges
are shown in three distinct colors.}\label{Figure_hex}
\end{figure}

Given $\Upsilon_\lambda$ let $\nu_1>\nu_2>\cdots>\nu_k$ be the
horizontal lozenges at the $k$th vertical
line from the left. (Horizontal lozenges are shown in blue in the left
panel of Figure~\ref{Fig_polyg_domain}.) We set $\nu_i=\kappa_i+k-i$
and denote the
resulting random signature
$\kappa$ of size $k$ as $\Upsilon_\lambda^k$. Further, let $\GUE_k$
denote the distribution of
$k$ (ordered) eigenvalues of a random Hermitian matrix from a Gaussian
unitary ensemble.

\begin{theorem}[(Theorem~\ref{Theorem_GUE})]
\label{Theorem_GUE_Intro}
Let $\lambda(N)\in\mathbb{GT}_N$, $N=1,2,\ldots$ be a sequence of
signatures. Suppose that there
exist a nonconstant piecewise-\break differentiable weakly decreasing
function $f(t)$ such that
\[
\sum_{i=1}^N\biggl\llvert
\frac{\lambda_i(N)}N - f(i/N)\biggr\rrvert = o(\sqrt{N})
\]
as $N\to\infty$ and also $\sup_{i,N} |\lambda_i(N)/N|<\infty$.
Then for
every $k$ as $N\to\infty$ we have
\[
\frac{\Upsilon_{\lambda(N)}^k-N E(f)}{\sqrt{NS(f)}} \to\GUE_k
\]
in the sense of weak convergence, where
\[
E(f)=\int_{0}^1 f(t) \,dt,\qquad  S(f)= \int
_0^1 f(t)^2 \,dt -E(f)^2
+ \int_0^1 f(t) (1-2t) \,dt.
\]
\end{theorem}
%
\begin{corollary}[(Corollary~\ref{Cor_minors})]
Under the same assumptions as in Theorem~\ref{Theorem_GUE_Intro} the
(rescaled) joint distribution of
$k(k+1)/2$ horizontal lozenges on the left $k$ lines weakly converges
to the joint distribution of the
eigenvalues of the $k$ top-left corners of a $k\times k$ matrix from a GUE.
\end{corollary}

Note that, in principle, our domains may approximate a nonpolygonal
limit domain as
$N\to\infty$. Thus the results of \cite{KO-Burgers} describing the
limit shape in terms of
algebraic curves are not applicable here, and not much is known about
the exact shape of the frozen
boundary. In particular, even the explicit expression for the
coordinate of the point where the
frozen boundary touches the left boundary (which we get as a side
result of Theorem~\ref{Theorem_GUE_Intro}) seems not to be present in the literature.

Our approach to the proof of Theorem~\ref{Theorem_GUE_Intro} is the
following: We express the expectations
of certain observables of uniformly random lozenge tilings through
normalized Schur polynomials $S_{\lambda}$ and
investigate the asymptotics of these polynomials. In this case we
prove and use the following asymptotic expansion
(given in Propositions \ref{Prop_convergence_GUE_case} and~\ref
{GUE_multivar_asymptotics}):
\begin{eqnarray*}
&&S_{\lambda}\bigl(e^{{h_1}/{\sqrt{N}}},\ldots,e^{
{h_k}/{\sqrt{N}}}; N,1\bigr) \\
&&\qquad=
\exp \bigl( \sqrt{N}E(f) (h_1+\cdots+h_k) +
\tfrac{1}2 S(f) \bigl(h_1^2 +\cdots+h_k^2
\bigr) + o(1) \bigr).
\end{eqnarray*}

We believe that our approach can be
extended to a natural $q$-deformation of uniform measure, which
assigns the weight $q^{\mathrm{vol}}$ to
lozenge tiling with volume $\mathrm{vol}$ below the corresponding stepped
surface, and also to lozenge
tilings with axial symmetry, as in \cite{FN,BK-O}. In the
latter case the Schur
polynomials are replaced with characters of orthogonal or symplectic
groups, and the limit object
also changes. We postpone the thorough study of these cases to a
future publication.

We note that there might be another approach to the proof of Theorem~\ref{Theorem_GUE_Intro}. Recently
there was progress in understanding random tilings of polygonal
domains. Petrov found double integral representations for the
correlation kernel describing the local
structure of tilings of a wide class of polygonal domains; see \cite
{Petrov-curves} and also~\cite{Metcalfe} for a similar result in context of random matrices.
Starting from these formulas,
one could try to prove the GUE-corners asymptotics along the lines of~\cite{OR-birth}.

\subsection{Application: Six-vertex model and random ASMs}
\label{Section_intro_ASM}

An \emph{alternating sign matrix} of size $N$ is a $N\times N$ matrix
whose entries are either
$\zero$,
$\one$ or $\mone$, such that the sum along every row and column is $1$
and, moreover, along
each row and each column the nonzero entries alternate in sign.
Alternating sign matrices are in
bijection with configurations of the six-vertex model with domain wall
boundary conditions as
shown at Figure~\ref{Fig_ASM}; more details on this bijection are given
in Section~\ref{Section_ASM}. A good review of the six-vertex model
can be found,
for example, in the book
\cite{Bax} by Baxter.

%
%
%
%

\begin{figure}

\includegraphics{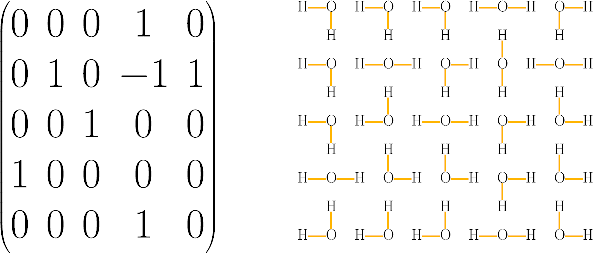}

\caption{An alternating sign matrix of size 5 and the corresponding
configuration of the
six-vertex model ({\emph{square ice}}) with domain wall boundary
condition. $\one$s in ASM correspond
to horizontal molecules H--O--H and $\mone$s to the vertical ones.}\label{Fig_ASM}
\end{figure}



Interest in ASMs from combinatorial perspective emerged since their
discovery in connection with
Dodgson condensation algorithm for determinant evaluations. Initially,
questions concerned
enumeration problems, for instance, finding the total number of ASMs of
given size $n$ (this was the
long-standing \emph{ASM conjecture} proved by Zeilberger \cite{Z} and
Kuperberg \cite{Ku}; the
full story can be found in the Bressoud's book \cite{Br}). Physicists'
interest stems from the fact that
ASMs are in one-to-one bijection with configurations of the
six-vertex model. Many questions on ASMs still remain open. Examples of
recent breakthroughs
include the Razumov--Stroganov \cite{RS} conjecture relating ASMs to
yet another model of
statistical mechanics [so-called O(1) \textit{loop model}], which was
finally proved very recently by
Cantini and Sportiello \cite{CS}, and the still open question on a
bijective proof of the fact
that totally symmetric self-complementary plane partitions and ASMs are
equinumerous. A brief
up-to-date introduction to the subject can be found, for example, in
\cite{BFZ}.

Our interest in ASMs and the six-vertex model is probabilistic. We
would like to know how a
\emph{uniformly random} ASM of size $n$ looks like when $n$ is large.
Conjecturally, the features
of this model should be similar to those of lozenge tilings: we expect
the formation of a limit
shape and various connections with random matrices. The properties of
the limit shape for ASMs
were addressed by Colomo and Pronko \cite{CP}; however, their arguments
are mostly not
mathematical, but physical.

In the present article we prove a partial result toward the following
conjecture.
%
\begin{conjecture}
\label{conjecture_ASM}
Fix any $k$. As $n\to\infty$ the probability that the number of
$\mone
$s in the first $k$ rows of a
uniformly random ASM of size $n$ is maximal (i.e., there is one $\mone
$ in second row, two $\mone$s in third row, etc.) tends to
$1$, and, thus $\mathbf{1}$s in first $k$ rows are interlacing. After proper
centering and rescaling, the distribution of the positions of $\one$s
tends to the
GUE-corners process as $n\to\infty$.
\end{conjecture}

Let $\Psi_k(n)$ denote the sum of coordinates of $\one$s minus the sum
of coordinates of $\mone$s
in the $k$th row of the uniformly random ASM of size $n$. We prove that
the centered and rescaled
random variables $\Psi_k(n)$ converge to the collection of i.i.d.
Gaussian random variables as
$n\to\infty$.

\begin{theorem}[(Theorem~\ref{Theorem_ASM})]
\label{Theorem_ASM_Intro}
For any fixed $k$ the random variable ${(\Psi_k(n)-n/2)}/{\sqrt{n}}$
weakly converges to the normal
random variable $N(0,\break  \sqrt{3/8})$. Moreover, the joint distribution
of any collection of such variables converges to the
distribution of independent normal random variables $N(0,\break  \sqrt{3/8})$.
\end{theorem}
\begin{remark*}
We also prove a bit stronger statement; see
Theorem~\ref{Theorem_ASM} for
the details.
\end{remark*}

Note that Theorem~\ref{Theorem_ASM_Intro} agrees with Conjecture \ref
{conjecture_ASM}. Indeed, if
the latter holds, then $\Psi_k(n)$ converges to the difference of \new
{the} sums of \new{the}
eigenvalues of
\new{a} $k\times k$ GUE-random matrix and of its $(k-1)\times(k-1)$
top left submatrix. But these
sums are the same as the traces of the corresponding matrices;
therefore, the difference of sums
equals the bottom right matrix element of \new{the} $k\times k$ matrix,
which is a Gaussian random
variable by the definition of GUE.

Our proof of Theorem~\ref{Theorem_ASM_Intro} has two components. First,
a result of Okada
\cite{Oka}, based on earlier work of Izergin and Korepin \cite{I,Kor}, shows that sums of
certain quantities over all ASMs can be expressed through Schur
polynomials (in an equivalent form
this was also shown by Stroganov \cite{St}). Second, our method gives
the asymptotic analysis of
these polynomials.

In fact, we claim that Theorem~\ref{Theorem_ASM_Intro} together with an
additional probabilistic
argument implies Conjecture \ref{conjecture_ASM}. However, this
argument is unrelated to the
asymptotics of symmetric polynomials and, thus is left out of the scope
of the (already long)
present paper; the proof of Conjecture \ref{conjecture_ASM} based on Theorem~\ref{Theorem_ASM_Intro} is presented by one of the authors in the later
article \cite{G-ASM}.

In the literature one can find another probability measure on ASMs
assigning the weight $2^{n_1}$
to the matrix with $n_1$ $\one$s. For this measure there are many
rigorous mathematical results,
due to the connection to the uniform measure on \emph{domino tilings of
the Aztec diamond}; see
\cite{EKLP,FS}. The latter measure can be viewed as a \emph
{determinantal point process},
which gives tools for its analysis. An analogue of Conjecture \ref
{conjecture_ASM} for the tilings
of Aztec diamond was proved by Johansson and Nordenstam \cite{JN}.

In regard to the combinatorial questions on ASMs, we note that there
has been interest in
\emph{refined} enumerations of alternating sign matrices, that is,
counting the number of ASMs with
fixed positions of $1$s along the boundary. In particular,
Colomo--Pronko \cite{CP-2point,CP-multi}, Behrend \cite{Be} and Ayyer--Romik \cite{AR} found
formulas relating $k$-refined
enumerations to $1$-refined enumerations for ASMs. Some of these
formulas are closely related to
particular cases of our multivariate formulas (Theorem~\ref
{Theorem_multivariate_Schur_1}) for
staircase Young diagrams.

\subsection{Application: $O(n=1)$-loop model}
\label{Section_intro_loop}

Recently found parafermionic observables in the so-called completely
packed $O(n=1)$ dense loop
model in a strip are also simply related to symmetric polynomials; see
\cite{GNP}. The $O(n=1)$
dense loop model is one of the representations of the percolation model
on the square lattice. For
the critical percolation models similar observables and their
asymptotic behavior were studied (see,
e.g., \cite{Smirnov}); however, the methods involved are usually
completely different from ours.

A configuration of the $O(n=1)$ loop model in a vertical strip consists
of two parts: a tiling of
the strip \new{on a square grid} of width $L$ and infinite height with
squares of two types shown
in Figure~\ref{fig:loop_model_squares} (left panel), and a choice of
one of the two types of
boundary conditions for each $1\times2$ segment along each of the
vertical boundaries of the
strip; \new{the} types appearing at the left boundary are shown in
Figure~\ref{fig:loop_model_squares} (right panel). Let $\daleth_L$
denote the
set of all configurations of
the model in the strip of width $L$. An element of $\daleth_6$ is shown
in Figure~\ref{fig:loop_model_example}. Note that the arcs drawn on
squares and
boundary segments form closed
loops and paths joining the boundaries. Therefore,
\new{the} elements of $\daleth_L$ have an interpretation as collections
of nonintersecting paths and closed
loops.

\begin{figure}

\includegraphics{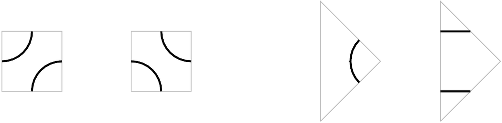}

\caption{Left panel: \new{the two} \old{Two}{} types of squares. Right
panel: \new{the two} \old{Two}{} types of boundary conditions.}
\label{fig:loop_model_squares}
\end{figure}

%
%

\begin{figure}[b]

\includegraphics{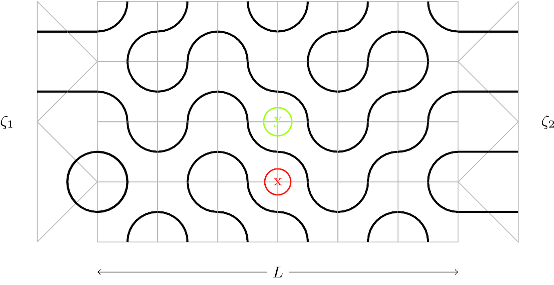}

\caption{A particular configuration of the dense loop model showing a
path passing between two
vertically adjacent points $x$ and $y$.}\label{fig:loop_model_example}
\end{figure}

%
%
%
%
%
%
%
%
%
%
%
%

In the simplest homogeneous case a probability distribution on $\daleth
_L$ is defined by declaring
the choice of one of the two types of squares to be an independent
Bernoulli random variable for
each \new{square}\old{point}{} of the strip and for each segment of the
boundary. That is, for each \new{square}\old{point}{} of the strip
we flip an unbiased coin to choose one of the two types of squares
(shown in Figure~\ref{fig:loop_model_squares}) and similarly for \new
{the} boundary
conditions. More generally, the type of
a square is chosen using a (possibly signed or even complex) weight
defined as a certain function
of its horizontal coordinate and depending on $L$ parameters
$z_1,\ldots
,z_L$; two other parameters
$\zeta_1$, $\zeta_2$ control the probabilities of the boundary
conditions \new{and, using a parameter $q$, the whole configuration is
further weighted by its number of closed loops}\old{}{It's important to
mention it, since the model is nonlocal and $q$ shows up}. We refer
the reader to~\cite{GNP} and references therein for the exact dependence of weights
on the parameters of the
model and for the explanation of the choices of parameters.

Fix two points $x$ and $y$, and consider a configuration $\omega\in
\daleth_L$. There are finitely
many paths passing between $x$ and $y$. For each such path $\tau$ we
define the current $c(\tau)$
as $0$ \new{if}\old{is}{} $\tau$ is a closed loop \new{or joins points
of the same boundary}; $1$ if $\tau$ joins the \new{two} boundaries and
$x$ lies above $\tau$;
$-1$ if $\tau$ joins the two boundaries and $x$ lies below $\tau$. The
total current $C^{x,y}(\omega)$
is the sum of $c(\tau)$ over all paths passing between $x$ and $y$. The
\emph{mean total current}
$F^{x,y}$ is defined as the expectation of~$C^{x,y}$. 

Two important properties of $F^{x,y}$ \new{are}\old{is}{} skew-symmetry
\[
F^{x,y}=-F^{y,x}
\]
and additivity
\[
F^{x_1,x_3}=F^{x_1,x_2} + F^{x_2,x_3}.
\]
These properties allow to express $F^{(x,y)}$ as a sum of several
instances of the mean total
current between two horizontally adjacent points
\[
F^{(i,j),(i,j+1)}
\]
and the mean total current between two vertically adjacent points
\[
F^{(j,i),(j+1,i)}.
\]

The authors of \cite{GNP} present a formula for $F^{(i,j),(i,j+1)}$ and
$F^{(j,i),(j+1,i)}$ which,
based on certain assumptions, expresses them through the symplectic characters
$\chi_{\lambda^L}(z_1^2,\ldots,z_L^2, \zeta_1^2,\zeta_2^2)$ where
$\lambda^L = (\lfloor
\frac{L-1}{2} \rfloor, \lfloor\frac{L-2}{2} \rfloor, \ldots, 1,0,0)$.
The precise relationship is
given in Section~\ref{section:dense_loop_model}. Our approach allows us
to compute the asymptotic
behavior of the formulas of \cite{GNP} as the lattice width $L
\rightarrow\infty$; see Theorem~\ref{theorem_dense_loop}. In particular, we prove that the leading term
in the asymptotic expansion
is independent of the boundary parameters $\zeta_1$, $\zeta_2$.

This problem was presented to the authors by de Gier \cite{G-MSRI,GP} during the program
``Random Spatial Processes'' at MSRI, Berkeley.

\subsection{Application: Matrix integrals}
\label{Section_intro_matrix}

Let $A$ and $B$ be two $N\times N$ Hermitian matrices with eigenvalues
$a_1,\ldots,a_N$ and
$b_1,\ldots,b_N$, respectively. The Harish-Chandra formula \cite{HC1,HC2} (sometimes known
also as Itzykson--Zuber \cite{IZ} formula in physics literature) is the
following evaluation of
the integral over the unitary group:
%
\begin{eqnarray}
\label{eq_HC_intro}&& \int_{U(N)} \exp\bigl(\operatorname{Trace}
\bigl(AUBU^{-1}\bigr)\bigr) \,dU
\nonumber
\\[-8pt]
\\[-8pt]
\nonumber
&&\qquad = \frac{\det
_{i,j=1,\ldots
, N}  (\exp(a_i b_j) )}{\prod_{i<j} (a_i-a_j) \prod_{i<j}
(b_i-b_j)} \prod
_{i<j} (j-i),
\end{eqnarray}
where the integration is with respect to the normalized Haar measure on
the unitary group $U(N)$.
Comparing \eqref{eq_HC_intro} with the definition of Schur polynomials
and using Weyl's dimension formula
\[
s_\lambda(1,\ldots,1)=\prod_{i<j}
\frac{(\lambda_i-i)-(\lambda_j-j)}{j-i},
\]
we observe that when $b_j=\lambda_j+N-j$ the above matrix integral is
the normalized Schur
polynomial times explicit product, that is,
\[
\frac{s_\lambda(e^{a_1},\ldots,e^{a_n})}{s_\lambda(1,\ldots,1)} \cdot \prod_{i<j}
\frac{e^{a_i}-e^{a_j}}{a_i-a_j}.
\]

Guionnet and Ma\"\i da studied (after some previous results
in the physics literature; see \cite{GM} and references therein) the
asymptotics of
the above integral as $N\to\infty$ when the rank of $A$
is finite and does not depend on $N$. This is precisely the asymptotics
of \eqref{eq_Normalized_Schur}.
Therefore,
our methods (in particular, Propositions \ref{proposition_convergence_mildest},
\ref{proposition_convergence_strongest}, \ref
{proposition_convergence_extended}) give a new proof
of some of the results of \cite{GM}. In the context of random matrices
the asymptotics of this
integral in the case when rank of $A$ grows as the size of $A$ grows
was also studied; see, for example,
\cite{GZ,CS}. However, currently we are unable to use our
methods for this case.

\subsection{Comparison with other approaches}
\label{Section_intro_compare}

Since asymptotics of symmetric polynomials as the number of variables
tends to infinity already
appeared in various contexts in the literature, it makes sense to
compare our approach to the
ones used before.

In the context of asymptotic representation theory the known approach
(see \cite{VK,OkOlsh,OkOlsh_BC,G-A})
is to use the so-called \emph{binomial formulas}. In the simplest case
of Schur polynomials such formulas read as
%
\begin{equation}
\label{eq_Binomial} S_\lambda(1+x_1,\ldots,1+x_k;
N,1) = \sum_\mu s_\mu(x_1,
\ldots ,x_k) c(\mu ,\lambda,N),
\end{equation}
where the sum is taken over all Young diagrams $\mu$ with at most $k$
rows, and $c(\mu,\lambda,N)$
are certain (explicit) coefficients. In the asymptotic regime of
Theorem~\ref{Theorem_U_VK} the
convergence of the left-hand side of \eqref{eq_Binomial} implies the
convergence of numbers
$c(\mu,\lambda,N)$ to finite limits as $N\to\infty$. Studying the
possible asymptotic behavior of
these numbers one proves the limit theorems for normalized Schur polynomials.

Another approach uses the decomposition
%
\begin{equation}
\label{eq_simple_decomp} S_\lambda(x_1,\ldots,x_k; N,1) =
\sum_\mu S_\mu(x_1,\ldots
,x_k;k,1) \,d(\mu ,\lambda,N),
\end{equation}
where the sum is taken over all signatures of length $k$. Recently in
\cite{BO-newA} and
\cite{Petrov-boundary} $k\times k$ determinantal formulas were found
for the coefficients
$d(\mu,\lambda,N)$. Again, these formulas allow the asymptotic analysis
which leads to the limit
theorems for normalized Schur polynomials.

The asymptotic regime of Theorem~\ref{Theorem_U_VK} is distinguished by
the fact that $\sum_i |\lambda_i(N)|/N$
is bounded as $N\to\infty$. This no longer holds when one studies
asymptotics of lozenge tilings, ASMs,
or $O(n=1)$ loop model. As far as the authors know, in the latter limit
regime neither formulas \eqref{eq_Binomial}
nor \eqref{eq_simple_decomp} gives simple ways to compute the
asymptotics. The reason for that is the fact that for any
fixed $\mu$ both $c(\mu,\lambda,N)$ and $d(\mu,\lambda,N)$ would
converge to zero as $N\to\infty$ and
more delicate analysis would be required to reconstruct the asymptotics
of normalized Schur polynomials.

Yet another, but similar approach to the proof of Theorem~\ref
{Theorem_U_VK} was used in \cite{Boyer2}
but, as far as authors know, it also does not extend to the regime we
need for other applications.

On the other hand the random-matrix asymptotic regime of \cite{GM} is
similar to the one we need for studying
lozenge tilings, ASMs, or $O(n=1)$ loop model. The approach of \cite
{GM} is based on the matrix model and
the proofs rely on \emph{large deviations} for Gaussian random
variables. However, it seems that the results
of \cite{GM} do not suffice to obtain our applications: for $k>1$ only
the first order asymptotics [which is the limit
of $\ln(S_\lambda(x_1,\ldots,x_k;N,1))/N$] was obtained in \cite{GM},
while our applications require more delicate analysis.
It also seems that the results of \cite{GM} (even for $k=1$) cannot be
applied in the framework of the
representation theoretic regime of Theorem~\ref{Theorem_U_VK}.

\section{Definitions and problem setup}\label{Section_definitions}

In this section we set up notation and introduce the symmetric
functions of our interest.

A \emph{partition} (or a \emph{Young diagram}) $\lambda$ is a
collection of nonnegative numbers
$\lambda_1\ge\lambda_2\ge\cdots,$ such that $\sum_i \lambda_i
<\infty$.
The numbers $\lambda_i$ are
\emph{row lengths} of $\lambda$, and the numbers $\lambda'_i=|\{
j\dvtx\lambda_j\ge i\}|$ are
\emph{column lengths} of $\lambda$.

More generally a \emph{signature} $\lambda$ of size $N$ is an
$N$-tuple of integers
$\lambda_1\ge\lambda_2\ge\cdots\ge\lambda_N$. The set of all signatures
of size $N$ is denoted
$\mathbb{GT}_N$. It is also convenient to introduce \emph{strict
signatures}, which are
$N$-tuples satisfying strict inequalities $\lambda_1>\lambda
_2>\cdots
>\lambda_N$; they from the
set $\GTs_N$. We are going to use the following identification between
elements of $\GT_N$ and
$\GTs_N$:
\[
\GT_N \ni\lambda\longleftrightarrow\lambda+\delta_N=\mu
\in\GTs _N,\qquad \mu_i=\lambda_i+N-i,
\]
where we set $\delta_N = (N-1,N-2,\ldots,1,0)$. The subset of $\GT_N$
($\GTs_N$) of all signatures
(strict signatures) with nonnegative coordinates is denoted $\GT_N^+$
($\GTs_N^+$).

One of the main objects of study in this paper are the rational Schur
functions, which originate
as the characters of the irreducible representations of the unitary
group $U(N)$ [equivalently, of
irreducible rational representations of the general linear group
$GL(N)$]. Irreducible representations
are parameterized by elements of $\GT_N$, which are identified with the
\emph{dominant weights}; see, for example, \cite{W} or \cite{Zh}. The
value of the character of the irreducible representation
$V_\lambda$ indexed by $\lambda\in\GT_N$, on a unitary matrix with
eigenvalues $u_1,\ldots,u_N$ is
given by the \emph{Schur function}
%
\begin{equation}
\label{eq_Schur_def} s_{\lambda}(u_1,\ldots,u_N) =
\frac{\det [
u_i^{\lambda_j+N-j} ]_{i,j=1}^N}{\prod_{i<j}(u_i-u_j)},
\end{equation}
which is a symmetric Laurent polynomial in $u_1,\ldots,u_N$. The
denominator in
\eqref{eq_Schur_def} is the Vandermonde determinant, and we denote it
through $\Delta$:
\[
\Delta(u_1,\ldots,u_N)=\det \bigl[u_i^{N-j}
\bigr]_{i,j=1}^N=\prod_{i<j}(u_i-u_j).
\]
When the numbers $u_i$ form a geometric progression, the determinant in
\eqref{eq_Schur_def} can
be evaluated explicitly as
%
\begin{equation}
\label{eq_Schur_at_q} s_{\lambda}\bigl(1,q,\ldots,q^{N-1}\bigr) = \prod
_{i<j}\frac{q^{\lambda_i+N-i}-q^{\lambda_j+N-j}}{q^{N-i}-q^{N-j}}.
\end{equation}
In particular, sending $q\to1$ we get
%
\begin{equation}
\label{eq_Weyl_dim} s_{\lambda}\bigl(1^N\bigr) = \prod
_{1\leq i<j\leq N} \frac{ (\lambda_i -i)
-(\lambda_j-j)}{j-i},
\end{equation}
where we used the notation
\[
1^N=(\underbrace{1,\ldots,1}_N).
\]
Identity \eqref{eq_Weyl_dim} gives the dimension of $V_\lambda$, and is
known as the Weyl's
dimension formula.

In what follows we intensively use the normalized versions of Schur functions:
\[
S_{\lambda}(x_1,\ldots,x_k;N,q) =
\frac{s_{\lambda}(x_1,\ldots
,x_k,1,q,q^2,\ldots,q^{N-1-k})}{s_{\lambda}(1,\ldots,q^{N-1})},
\]
in particular,
\[
S_{\lambda}(x_1,\ldots,x_k;N,1) =
\frac{s_{\lambda}(x_1,\ldots
,x_k,1^{N-k})}{s_{\lambda}(1^N)}.
\]

The Schur functions are characters of type $A$ (according to the
classification of root systems),
their analogues for other types are related to the \emph{multivariate
Jacobi polynomials}.

For $a,b>-1$ and $m=0,1,2,\ldots$ let $\mathfrak p_m(x;a,b)$ denote the
classical Jacobi polynomials
orthogonal with respect to the weight $(1-x)^a(1+x)^b$ on the interval
$[-1,1]$; see, for example,
\cite{BEE,Koekoek}. We use the normalization of \cite{BEE}, and
thus the polynomials can be
related to the Gauss hypergeometric function $_2F_1$,
\[
\mathfrak{p}_{m}(x;a,b) = \frac{\Gamma(m+a+1)}{\Gamma(m+1)\Gamma(a+1)} {}_2F_1
\biggl(-m,m+a+b+1,a+1;\frac{1-x}{2} \biggr).
\]
For any strict signature $\lambda\in\GTs_N^+$ set
\[
{\mathfrak P}_{\lambda}(x_1,\ldots,x_N;a,b) =
\frac{ \det[
\mathfrak
{p}_{\lambda_i}(x_j;a,b)]_{i,j=1}^N}{\Delta(x_1,\ldots,x_N)},
\]
and for any (nonstrict) $\lambda\in\GT_N^+$ define
%
\begin{equation}
\mathfrak{J}_{\lambda}(z_1,\ldots,z_N;a,b) =
c_\lambda \mathfrak{P}_{\lambda+\delta} \biggl(\frac{z_1+z_1^{-1}}{2},\ldots ,
\frac
{z_N+z_N^{-1}}{2};a,b \biggr),
\end{equation}
where $c_\lambda$ is a constant chosen so that the leading coefficient
of $\mathfrak{J}_\lambda$
is $1$. The polynomials $\mathfrak{J}_\lambda$ are (a particular case
of) $BC_N$ multivariate
Jacobi polynomials; see, for example, \cite{OkOlsh_BC} and also \cite
{HS,M2,Koornwinder}.
We also use their normalized versions
%
\begin{equation}
\label{eq_Jacobi_normalized} J_{\lambda}(z_1,\ldots,z_k;N,a,b) =
\frac{\mathfrak{J}_{\lambda
}(z_1,\ldots,z_k,1^{N-k};
a,b)}{\mathfrak{J}_{\lambda}(1^N;a,b)}.
\end{equation}
Again, there is an explicit formula for the denominator in \eqref
{eq_Jacobi_normalized} and also
for its $q$-version. For special values of parameters $a$ and $b$, the
functions $J_{\lambda}$
can be identified with spherical functions of classical Riemannian
symmetric spaces of compact
type, in particular, with normalized characters of orthogonal and
symplectic groups; see, for example,
\cite{OkOlsh_BC}, Section~6.

Let us give more details on the latter case of the symplectic group
$\operatorname{Sp}(2N)$, as we need it for
one of our applications. This case corresponds to $a=b=1/2$, and here
the formulas can be
simplified.

The value of character of irreducible representation of $\operatorname{Sp}(2N)$
parameterized by
$\lambda\in\GT_N^+$ on symplectic matrix with eigenvalues
$x_1,x_1^{-1},\ldots,x_N, x_N^{-1}$ is
given by (see, e.g., \cite{W,Zh})
\[
\chi_{\lambda}(x_1,\ldots,x_N) =
\frac{ \det [ x_i^{\lambda
_j+N+1-j} -
x_i^{-(\lambda_j+N+1-j)} ]_{i,j=1}^N} {\det [x_i^{N+1-j} -
x_i^{-N-1+j} ]_{i,j=1}^N}.\vadjust{\goodbreak}
\]
The denominator in the last formula can be expressed as a product formula,
and we denote it $\dd$
%
\begin{eqnarray}
\label{eq_alternative_form} \dd(x_1,\ldots,x_N)&=&\det \bigl[
x_i^{N-j+1} - x_i^{-N+j-1}
\bigr]_{i,j=1}^N\nonumber
\\
&= &\prod_i \bigl(x_i-x_i^{-1}
\bigr) \prod_{i<j}\bigl(x_i+x_i^{-1}
- \bigl(x_j+x_j^{-1}\bigr) \bigr)\\
& =&
\frac{\prod_{i<j}(x_i-x_j)(x_ix_j-1)\prod_i (x_i^2-1)}{
(x_1\cdots x_n)^n}.\nonumber
\end{eqnarray}
The normalized symplectic character is then defined as
\[
\X_{\lambda}(x_1,\ldots,x_k;N,q) =
\frac{\chi_{\lambda
}(x_1,\ldots
,x_k,q,\ldots,q^{N-k})}{\chi_{\lambda}(q,q^2,\ldots,q^{N})},
\]
in particular
\[
\X_{\lambda}(x_1,\ldots,x_k;N,1) =
\frac{\chi_{\lambda
}(x_1,\ldots
,x_k,1^{N-k})}{\chi_{\lambda}(1^N)},
\]
and both denominators again admit explicit formulas.

In most general terms, in the present article we study the symmetric
functions $S_\lambda$,
$J_\lambda$, $X_\lambda$, their asymptotics as $N\to\infty$ and its
applications.

\textit{Some further notation}.
We intensively use the $q$-algebra notation
\[
[m]_q = \frac{q^m-1}{q-1},\qquad [a]_q!=\prod
_{m=1}^a [m]_q,
\]
and $q$-Pochhammer symbol
\[
(a;q)_k = \prod_{i=0}^{k-1}
\bigl(1-aq^i\bigr).
\]

Since there are lots of summations and products in the text where $i$
plays the role of the index,
we write $\ii$ for the imaginary unit to avoid the confusion.

\section{Integral and multivariate formulas}

In this section we derive integral formulas for normalized characters
of one variable and also
express the multivariate normalized characters as determinants of
differential (or, sometimes,
difference) operators applied to the product of the single variable
normalized characters.

We first exhibit some general formulas, which we later specialize to
the cases of Schur functions,
symplectic characters and multivariate Jacobi polynomials.

\subsection{General approach}
\label{Section_general}

\begin{definition}
\label{Definition_class} For a given sequence of numbers $\theta
=(\theta
_1,\theta_2,\ldots)$, a
collection of functions $\{A_\mu(x_1,\ldots,x_N)\}$, $N=1,2,\ldots,$
$\mu
\in\GTs_N$ (or $\GTs_N^+$)
is called\vadjust{\goodbreak} a \emph{class of determinantal symmetric functions} with
parameter $\theta$, if there
exist functions $\alpha(u)$, $\beta(u)$, $g(u,v)$, numbers $c_N$ and
linear operator $T$ such
that for all $N$ and $\mu$ we have:
\begin{longlist}[(1)]
\item[(1)]
\[
\A_{\mu}(x_1,\ldots,x_N)=\frac{ \det[g(x_j;\mu
_i)]_{i,j=1}^N}{\Delta
(x_1,\ldots,x_N)},
\]
\item[(2)]
\[
\A_{\mu}(\theta_1,\ldots,\theta_N)
=c_N \prod_{i=1}^N \beta(\mu
_i) \prod_{i<j} \bigl(\alpha(
\mu_i)-\alpha(\mu_j)\bigr),
\]
%
%
\item[(3)] $g(x;m)$ ($m\in\mathbb{Z}$ for the case of $\GTs$ and
$m\in
\mathbb Z_{\ge0}$ for the case $\GTs^+$)
are eigenfunctions of $T$ acting on $x$ with eigenvalues $\alpha(m)$,
that is,
\[
T\bigl(g(x,m)\bigr)= \alpha(m)g(x,m),
\]
\item[(4)] $\alpha'(m)\ne0$ for all $m$ as above.
\end{longlist}
\end{definition}

\begin{proposition}\label{Proposition_general_multivar}
For $\A_\mu(x_1,\ldots,x_N)$, as in Definition~\ref
{Definition_class} we
have the following formula:
%
\begin{eqnarray}
\label{eqn:general_multivar}
&&\frac{\A_{\mu}(x_1,\ldots,x_k,\theta_1,\ldots,\theta_{N-k})}{\A
_{\mu
}(\theta_1,\ldots,
\theta_{N})}
\nonumber
\\
&&\qquad= \frac{ c_{N-k}}{c_N} \prod
_{i=1}^k \prod_{j=1}^{N-k}
\frac
{1}{x_i-\theta_j}
\frac{\det [ T_i^{j-1} ]_{i,j=1}^k}{\Delta(x_1,\ldots,x_k)}\\
&&\qquad\quad{}\times \prod_{i=1}^k \Biggl(
\frac{\A_{\mu}(x_i,\theta_1,\ldots,\theta_{N-1})}{\A_{\mu
}(\theta
_1,\ldots,\theta_N)} \prod_{j=1}^{N-1}(x_i-
\theta_j)\frac{ c_N}{c_{N-1}} \Biggr),\nonumber
\end{eqnarray}
where $T_i$ is operator $T$ acting on variable $x_i$.
\end{proposition}
\begin{remark*} Since operators $T_i$ commute, we have
\[
\det \bigl[ T_i^{j-1} \bigr]_{i,j=1}^k=
\prod_{i<j}(T_i-T_j).
\]
We also note that some of the denominators in \eqref
{eqn:general_multivar} can be grouped in the compact form
\[
\prod_{i=1}^k\prod
_{j=1}^{N-k} (x_i-\theta_j)
\Delta(x_1,\ldots ,x_k) = \frac{ \Delta(x_1,\ldots,x_k,\theta_1,\ldots,\theta
_{N-k})}{\Delta
(\theta_1,\ldots,\theta_{N-k})}.
\]
Moreover, in our applications, the coefficients $c_m$ will be inversely
proportional to $\Delta(\theta_1,\ldots,\theta_m)$, so we will be able
to write alternative formulas where the prefactors are simple ratios of
Vandermondes.
\end{remark*}

\begin{pf*}{Proof of Proposition~\ref{Proposition_general_multivar}}
We will compute the determinant from property (1) of $\A$ by summing
over all $k\times k$ minors
in the first $k$ columns, where we set the convention that $i$ is a row
index and $j$ is a column index. The rows in the corresponding minors
will be indexed by $I=\{i_1<i_2<\cdots<i_k\}$ and $\mu_I
=(\mu_{i_1},\ldots,\mu_{i_k})$. $I^c$ denotes the complement of $I$ in
$\{1,2,\ldots,n\}$. We have
%
\begin{eqnarray}
\label{eq_x15} &&\frac{\A_{\mu}(x_1,\ldots,x_k,\theta_1,\ldots,\theta_{N-k})}{\A
_{\mu
}(\theta_1,\ldots,\theta_N)}\nonumber\\
&&\qquad = \frac{1}{\prod_{i=1}^k\prod_{j=1}^{N-k}(x_i-\theta_j)}
\\
&&\qquad\quad{}\times\sum_{I=\{i_1<i_2<\cdots<i_k\}}(-1)^{\sum_{\ell\in I}(\ell
-1)} \A
_{\mu_I}(x_1,\ldots,x_k) \frac{\A_{\mu_{I^c}}(\theta_1,\ldots,\theta_{N-k})}{\A_{\mu
}(\theta
_1,\ldots,\theta_{N})}.
\nonumber
\end{eqnarray}
For each set $I$ we have
%
\begin{eqnarray}
\label{eq_x16} &&\frac{\A_{\mu_{I^c}}(\theta_1,\ldots,\theta_{N-k})}{\A_{\mu
}(\theta
_1,\ldots,\theta_{N})} \frac{c_{N}}{c_{N-k}}\nonumber\\
&&\qquad=\frac{\prod_{i \in I^c} \beta(\mu_i) \prod_{i<j; i,j \in I^c}
(\alpha
(\mu_i)-\alpha(\mu_j))}{\prod_{i=1}^N \beta(\mu_i)\prod_{1\leq
i<j \leq
N}(\alpha(\mu_i)-\alpha(\mu_j))}\nonumber
\\
&&\qquad= \Biggl[ \prod_{i\in I} \Biggl( \frac{1}{\beta(\mu_i)}
\prod_{j=i+1}^N\frac{1}{\alpha(\mu_i)-\alpha(\mu_j)} \Biggr)
\Biggr] \prod_{i \notin I, j\in I, i<j}\frac{1}{ \alpha(\mu_i)-\alpha
(\mu_j)}
\nonumber
\\[-8pt]
\\[-8pt]
\nonumber
&&\qquad= \Biggl[ \prod_{i\in I} \Biggl( \frac{1}{\beta(\mu_i)}
\prod_{j=i+1}^N\frac{1}{\alpha(\mu_i)-\alpha(\mu_j)} \Biggr)
\Biggr] \\
&&\qquad\quad{}\times\frac{\prod_{i<j, i,j \in I}(\alpha(\mu_i)-\alpha(\mu
_j))}{\prod_{r<s, r\in I, s\in[1,\ldots,N]} -(\alpha(\mu_s)-\alpha(\mu_r))}\nonumber
\\
&&\qquad=\prod_{i<j; i,j \in I}\bigl(\alpha(\mu_i)-
\alpha(\mu_j)\bigr) \cdot\prod_{r \in I}
\frac{(-1)^{r-1}}{\beta(\mu_r)\prod_{s \neq r} (\alpha(\mu
_r)-\alpha(\mu_s))}.
\nonumber
\end{eqnarray}
We also have that
%
\begin{eqnarray}
\label{eq_x17} &&\prod_{i<j; i,j \in I} \bigl(\alpha(
\mu_i)-\alpha(\mu_j)\bigr) A_{\mu_I}(x_1,
\ldots,x_k) \Delta(x_1,\ldots,x_k)\nonumber
\\
&&\qquad= \det \bigl[ \alpha(\mu_{i_\ell})^{j-1} \bigr]_{\ell,j=1}^k
\sum_{\sigma\in
S_k} (-1)^{\sigma} \prod
_{\ell=1}^k g(x_{\sigma_i};\mu_{i_\ell})\nonumber
\\
&&\qquad=\sum_{\sigma\in S_k} (-1)^{\sigma}\det \bigl[\alpha(
\mu_{i_\ell})^{j-1}g(x_{\sigma
_j};\mu _{i_\ell})
\bigr]_{\ell,j=1}^k\\
&&\qquad =\sum_{\sigma\in S_k}
(-1)^{\sigma}\det \bigl[T_{\sigma
_j}^{j-1}g(x_{\sigma_j};
\mu_{i_\ell}) \bigr]_{\ell,j=1}^k\nonumber
\\
&&\qquad=\det \bigl[ T_i^{j-1} \bigr]_{i,j=1}^k
\sum_{\sigma\in S_k}\prod_{\ell=1}^k
g(x_{\sigma_i};\mu_{i_\ell}).
\nonumber
\end{eqnarray}
Combining \eqref{eq_x15}, \eqref{eq_x16} and \eqref{eq_x17} we get
%
\begin{eqnarray}
\label{eq_x18}&& \frac{\A_{\mu}(x_1,\ldots,x_k,\theta_1,\ldots,\theta_{N-k})}{
\A_{\mu}(\theta_1,\ldots,\theta_N)} \prod_{i=1}^k
\prod_{j=1}^{N-k}(x_i-
\theta_j)\frac{c_N}{c_{N-k}}
\nonumber
\\
&&\qquad= \frac{\det [ T_i^{j-1} ]_{i,j=1}^k}{\Delta(x_1,\ldots,x_k)} \\
&&\qquad\quad{}\times\sum_{I=\{i_1<i_2<\cdots<i_k\}}\sum
_{\sigma\in S(k)}\prod_\ell
\frac{g(x_\ell;\mu_{i_{\sigma_\ell}}) }{\beta(\mu_{i_{\sigma
_\ell}})
\prod_{j \neq i_{\sigma_\ell}} (\alpha(\mu_{i_{\sigma_\ell
}})-\alpha
(\mu_j))}.\nonumber
\end{eqnarray}
Note that double summation in the last formula is a summation over all
(ordered) collections of
distinct numbers. We can also include into the sum the terms where some
indices $i_\ell$ coincide,
since application of the Vandermonde of linear operators annihilates
such terms. Therefore,
\eqref{eq_x18} equals
\[
\frac{\det [ T_i^{j-1} ]_{i,j=1}^k }{\Delta(x_1,\ldots
,x_k)}\prod_{\ell=1}^k \sum
_{i_\ell=1}^N \frac{g(x_\ell;\mu_{i_\ell}) }{\beta(\mu_{i_\ell})\prod_{j \neq
i_\ell}
(\alpha(\mu_{i_\ell})-\alpha(\mu_j))}.
\]
When $k=1$ the operators and the product over $\ell$ disappear, so we
see that the remaining sum
is exactly the univariate ratio $
\frac{\A_{\mu}(x_\ell,\theta_1,\ldots,\theta_{N-1})}{\A_{\mu
}(\theta
_1,\ldots,\theta_N)}\prod_{j=1}^{N-1}(x_\ell-\theta_j)
\frac{c_N}{c_{N-1}}$, and we obtain the desired formula.
\end{pf*}

\begin{proposition}\label{Proposition_integral_general}
Under the assumptions of Definition~\ref{Definition_class} we have the
following integral formula
for the normalized univariate $\A_{\mu}$:
\begin{eqnarray}
\label{eq_x19}&& \frac{\A_{\mu}(x,\theta_1,\ldots,\theta
_{N-1})}{\A_{\mu
}(\theta_1,\ldots,\theta_N)}
\nonumber
\\
&&\qquad= \Biggl(\frac{c_{N-1}}{c_N} \prod
_{i=1}^{N-1}\frac{1}{x-\theta
_i} \Biggr)\\
&&\qquad\quad{}\times\frac{1}{2\pi\ii} \oint_C \frac{g(x;z)\alpha'(z)}{\beta(z)\prod_{i=1}^N (\alpha(z)-\alpha
(\mu
_i))} \,dz.\nonumber
\end{eqnarray}
Here the contour $C$ includes only the poles of the integrand at $z=\mu
_i$, $i=1,\ldots,N$.
\end{proposition}
\begin{pf}
As a byproduct in the proof of Proposition~\ref
{Proposition_general_multivar} we obtained the
following formula:
%
\begin{eqnarray}
\label{eq_x20} &&\frac{\A_{\mu}(x,\theta_1,\ldots,\theta_{N-1})}{\A_{\mu}(\theta
_1,\ldots
,\theta_N)}\prod_{j=1}^{N-1}(x-
\theta_j) \frac{c_N}{c_{N-1}}
\nonumber
\\[-8pt]
\\[-8pt]
\nonumber
&&\qquad= \sum_{i=1}^N
\frac{g(x;\mu_{i}) }{\beta(\mu
_{i})\prod_{j \neq i}
(\alpha(\mu_{i})-\alpha(\mu_j))}.
\end{eqnarray}
Evaluating the integral in \eqref{eq_x19} as the sum of residues we
arrive at the right-hand side of~\eqref{eq_x20}.
\end{pf}

\subsection{Schur functions}
Here we specialize the formulas of Section~\ref{Section_general} to the
Schur functions.

\begin{proposition}
Rational Schur functions $s_\lambda(x_1,\ldots,x_N)$ (as above we
identify $\lambda\in\GT_N$ with
$\mu=\lambda+\delta\in\GTs_N$) are a class of determinantal
functions with
\begin{eqnarray*}
\theta_i&=&q^{i-1},\qquad g(x;m) =x^m,\qquad \alpha(x)=
\frac{q^x-1}{q-1},\qquad \beta(x)=1,
\\
c_N&=& \prod_{1\leq i<j \leq N}\frac{q-1}{q^{j-1}-q^{i-1}}=
\frac
{1}{q^{{N\choose3}}}\prod_{j=1}^{N-1}
\frac{1}{[j]_q!}, \qquad [Tf](x)=\frac{f(qx)-f(x)}{q-1}.
\end{eqnarray*}
\end{proposition}
\begin{pf}
This immediately follows from the definition of Schur functions \eqref
{eq_Schur_def} and
evaluation formula \eqref{eq_Schur_at_q}.
\end{pf}

Propositions \ref{Proposition_general_multivar} and \ref
{Proposition_integral_general} specialize
to the following theorems.

\begin{theorem} For any signature $\lambda\in\GT_N$ and any $k\le N$,
we have
\label{Theorem_multivariate_Schur_q}
\begin{eqnarray*}
S_{\lambda}(x_1,\ldots,x_k;N,q)&=&\frac{q^{{k+1\choose3} -(N-1)
{k\choose2}} \prod_{i=1}^k [N-i]_q!}{\prod_{i=1}^k\prod_{j=1}^{N-k}
(x_i-q^{j-1})}
\\
&&{}\times\frac{\det [ D_{i,q}^{j-1}
]_{i,j=1}^k}{\Delta(x_1,\ldots,x_k)} \prod_{i=1}^k
\frac{S_{\lambda}(x_i;N,q) \prod_{j=1}^{N-1}(x_i-q^{j-1})}{[N-1]_q!},
\end{eqnarray*}
where $D_{i,q}$ is the
difference operator acting on the function $f(x_i)$ by the formula
\[
[D_{i,q}f](x_i)=\frac{f(qx_i)-f(x_i)}{q-1}.
\]
\end{theorem}

\begin{theorem}
\label{Theorem_Integral_representation_Schur_q}
For any signature $\lambda\in\GT_N$ and any $x\in\mathbb C$ other than
$0$ or~$q^{i}$, $i\in\{0,\ldots,N-2\}$, we have
\[
S_{\lambda}(x;N,q) = \frac{[N-1]_q!q^{{N-1\choose2}}(q-1)^{N-1}}{
\prod_{i=1}^{N-1}(x-q^{i-1})}\cdot\frac{\ln(q)}{2\pi
\ii}
\oint_C \frac{x^zq^{z}}{\prod_{i=1}^N(q^z-q^{\lambda_i+N-i} )}\,dz,
\]
where the contour $C$ includes the poles at $\lambda_1+N-1,\ldots
,\lambda_N$, and no other poles of
the integrand.
\end{theorem}

\begin{remark*} There is an alternative derivation of Theorem~\ref{Theorem_Integral_representation_Schur_q} suggested by~Okounkov.
Let $x=q^k$ with $k>N$.
The definition of Schur polynomials implies the following symmetry for
any $\mu,\lambda\in\GT_N$:
%
\begin{equation}
\label{eq_Schur_symmetry} \frac{s_\lambda(q^{\mu_1+N-1},\ldots,q^{\mu_N})}{s_\lambda
(1,\ldots,q^{N-1})}= \frac{s_\mu(q^{\lambda_1+N-1},\ldots,q^{\lambda_N})}{s_\mu
(1,\ldots,q^{N-1})}.
\end{equation}
Using this symmetry,
\[
S_{\lambda}\bigl(q^{k};N,q\bigr) = \frac{h_{k+1-N}(q^{\lambda_1+N-1},\ldots,q^{\lambda
_N})}{h_{k+1-N}(1,\ldots,q^{N-1})},
\]
where $h_k=s_{(k,0,\ldots)}$ is the complete homogeneous symmetric
function. Integral
representation for $h_k$ can be obtained using their generating
function (see, e.g.,~\cite{M}, Chapter
I, Section~2)
\[
H(z)= \sum_{\ell=0}^{\infty} h_\ell(y_1,
\ldots,y_N) z^{\ell} =\prod_{i=1}^{N}
\frac{1}{1-z y_i}.
\]
Extracting $h_\ell$ as
\[
h_\ell=\frac{1}{2\pi\ii} \oint\frac{H(z)}{z^{\ell+1}} \,dz,
\]
we arrive at the integral representation equivalent to Theorem~\ref{Theorem_Integral_representation_Schur_q}. In fact symmetry \eqref
{eq_Schur_symmetry}
holds in a greater generality: namely, one can replace Schur functions
with \emph{Macdonald
polynomials}, which are their $(q,t)$-deformation; see \cite{M}, Chapter
VI. This means that, perhaps,
Theorem~\ref{Theorem_Integral_representation_Schur_q} can be extended
to the Macdonald
polynomials. On the other hand, we do not know whether a simple
analogue of Theorem~\ref{Theorem_multivariate_Schur_q} for Macdonald
polynomials exists.
\end{remark*}

Sending $q\to1$ in Theorems \ref{Theorem_multivariate_Schur_q},
\ref{Theorem_Integral_representation_Schur_q} we get the following:

\begin{theorem}\label{Theorem_multivariate_Schur_1} For any signature $\lambda\in\GT_N$ and any $k\le N$
we have
\begin{eqnarray*}
S_{\lambda}(x_1,\ldots,x_k;N,1)&=&\prod
_{i=1}^k\frac{(N-i)!
}{(N-1)!(x_i-1)^{N-k} }
\\
&&\hspace*{13pt}{}\times\frac{ \det [ D_{i,1}^{j-1} ]_{i,j=1}^{k} }{\Delta
(x_1,\ldots
,x_k)}\prod_{j=1}^k
S_{\lambda}(x_j;N,1) (x_j-1)^{N-1},
\end{eqnarray*}
where $D_{i,1}$ is the differential operator $x_i \frac{\partial
}{\partial x_i}$.
\end{theorem}

\begin{theorem}
\label{Theorem_Integral_representation_Schur_1}
For any signature $\lambda\in\GT_N$ and any $x\in\mathbb C$ other than
$0$ or $1$, we have
%
\begin{equation}
\label{eq_x34} S_\lambda(x;N,1) = \frac{(N-1)!}{(x-1)^{N-1} }\frac
{1}{2\pi\ii}
\oint_C \frac{x^z}{\prod_{i=1}^N(z-(\lambda_i+N-i))}\,dz,
\end{equation}
where the contour $C$ includes all the poles of the
integrand.
\end{theorem}

Note that this formula holds when $x\to1$. Clearly, $\lim_{x\to1}
S_\lambda(x;N,1)=1$. The
convergence of the integral in \eqref{eq_x34} to $1$ can be
independently seen, for example, by
application of L'Hospital's rule and evaluation of the resulting integral.

Let us state and prove several corollaries of Theorem~\ref
{Theorem_multivariate_Schur_1}.

For any integers $j,\ell,N$, such that $0\le\ell<j<N$, define the
polynomial $P_{j,\ell,N}(x)$ as
%
\begin{eqnarray}\qquad
\label{eq_poly_ptl}&& P_{j,\ell,N}(x)
\nonumber
\\[-8pt]
\\[-8pt]
\nonumber
&&\qquad= \pmatrix{j-1
\cr
\ell}\frac{N^{\ell} (N-j)!}{(N-1)!}
(x-1)^{j-\ell-N} \biggl[ \biggl(x\frac{\partial}{\partial x} \biggr)^{j-1-\ell}
(x-1)^{N-1} \biggr],
\end{eqnarray}
it is easy to see (e.g., by induction on $j-\ell$) that $P_{j,\ell,N}$
is a polynomial in $x$ of
degree $j-\ell-1$, and its coefficients are bounded as $N\rightarrow
\infty$. Also,
$P_{j,0,N}(x)=x^{j-1}+O(1/N)$.

\begin{proposition} \label{Proposition_multivariate_expansion}For any
signature $\lambda\in\GT_N$ and any $k\le N$, we have
%
\begin{eqnarray}
\label{eq_multivariate_expansion} &&S_{\lambda}(x_1,\ldots,x_k;N,1)\nonumber\\
&&\qquad =
\frac{1 }{\Delta(x_1,\ldots,x_k)}
\\
&&\qquad\quad{}\times\det \Biggl[\sum_{\ell=0}^{j-1}
\frac{D_{i,1}^{\ell}[S_{\lambda}(x_i;N,1)]}{N^\ell} P_{j,\ell,N}(x_i) (x_i-1)^{\ell+k-j}
\Biggr]_{i,j=1}^k.
\nonumber
\end{eqnarray}
\end{proposition}
\begin{pf}
We apply Theorem~\ref{Theorem_multivariate_Schur_1} and, noting that
\[
\biggl(x\frac\partial{\partial x} \biggr)^j\bigl[f(x)g(x)\bigr] =
\sum_{\ell=0}^j \pmatrix{j
\cr
\ell} \biggl(x
\frac\partial{\partial x} \biggr)^\ell \bigl[f(x)\bigr] \biggl(x\frac
\partial{\partial x} \biggr)^{j-\ell}\bigl[g(x)\bigr]
\]
for any $f$ and $g$, we obtain
%
\begin{eqnarray}
&&S_{\lambda}(x_1,\ldots,x_k;N,1)\nonumber
\\
&&\qquad=\frac{1}{\Delta(x_1,\ldots,x_k)} \det \biggl[ \frac{(N-j)!}{(N-1)!}\frac{D_i^{j-1}(S_{\lambda
}(x_i;N,1)(x_i-1)^{N-1})}{(x_i-1)^{N-k}}
\biggr]_{i,j=1}^k
\\
&&\qquad=\Biggl( \det \Biggl[ \sum_{\ell=0}^{j-1} D_{i,1}^{\ell}
\bigl[S_{\lambda
}(x_i;N,1)\bigr] \pmatrix{j-1\cr\ell}
\frac{(N-j)!}{(N-1)!}\nonumber\\
&&\hspace*{139pt}{}\times\frac{D_{i,1}^{j-\ell-1}(x_i-1)^{N-1}}{(x_i-1)^{N-k}
} \Biggr]_{i,j=1}^k\Biggr)\nonumber\\
&&\qquad\quad{}\Big/{\Delta(x_1,\ldots,x_k)}. 
\nonumber
\end{eqnarray}
\upqed\end{pf}

\begin{corollary}
\label{Corollary_multiplicativity_for_U} Suppose that the sequence
$\lambda(N)\in\mathbb{GT}_N$ is
such that
\[
\lim_{N\to\infty} S_{\lambda(N)}(x;N,1)= \Phi(x)
\]
uniformly on compact subsets of some region $M\subset\mathbb C$, then
for any $k$
\[
\lim_{N\to\infty} S_{\lambda(N)}(x_1,
\ldots,x_k;N,1)=\Phi (x_1)\cdots\Phi(x_k)
\]
uniformly on compact subsets of $M^k$.
\end{corollary}
\begin{pf}
Since $S_{\lambda(N)}(x;N,1)$ is a polynomial, it is an analytic
function. Therefore, the uniform
convergence implies that the limit $\Phi(x)$ is analytic and all
derivatives of $S_{\lambda(N)}(x)$ converge to
the derivatives of $\Phi(x)$.

Now suppose that all $x_i$ are distinct.
Then we can use Proposition~\ref{Proposition_multivariate_expansion},
and get as $N\to\infty$
\begin{eqnarray*}
&&S_{\lambda(N)}(x_1,\ldots,x_k;N,1) \\
&&\qquad=
\frac{ \det [ (x_i-1)^{k-j} {S_{\lambda(N)}(x_i;N,1)}
P_{j,0,N}(x_i) + O(1/N)
 ]_{i,j=1}^k}{\Delta(x_1,\ldots,x_k)}
\\
&&\qquad=\frac{ \det [ (x_i-1)^{k-j}
{S_{\lambda(N)}(x_i;N,1)} x_i^{j-1} + O(1/N)  ]_{i,j=1}^k}{\Delta
(x_1,\ldots,x_k)}
\\
&&\qquad=\prod_{i=1}^k S_{\lambda(N)}(x_i;N,1)
\frac{ \det [ (x_i-1)^{k-j}
x_i^{j-1}  ]_{i,j=1}^k +O(1/N)}{\Delta(x_1,\ldots,x_k)}
\\
&&\qquad=\prod_{i=1}^k S_{\lambda(N)}(x_i;N,1)
\biggl(1+ \frac
{O(1/N)}{\Delta
(x_1,\ldots,x_k)} \biggr),
\end{eqnarray*}
where $O(1/N)$ is uniform over compact subsets of $M^k$. We conclude that
%
\begin{equation}
\label{eq_x21} \lim_{N\to\infty} S_{\lambda(N)}(x_1,
\ldots,x_k;N,1)=\Phi (x_1)\cdots\Phi(x_k)
\end{equation}
uniformly on compact subsets of
\[
M^k{}\Big\backslash{}\bigcup_{i<j}
\{x_i=x_j\}.
\]
Since the left-hand side of \eqref{eq_x21} is analytic with only
possible singularities at $0$ for all $N$, the uniform
convergence in \eqref{eq_x21} also holds when some of $x_i$ coincide.
\end{pf}

\begin{corollary}\label{mult_logarithm}
Suppose that the sequence $\lambda(N)\in\mathbb{GT}_N$ is such that
\[
\lim_{N\to\infty} \frac{\ln (S_{\lambda(N)}(x;N,1)
)}{N}= \Psi(x)
\]
uniformly on compact subsets of some region $M\subset\mathbb C$. In
particular, there is a well-defined branch of logarithm in $M$ for
large enough $N$. Then for any $k$,
\[
\lim_{N\to\infty} \frac{\ln (S_{\lambda(N)}(x_1,\ldots
,x_k;N,1)
)}{N}=\Psi(x_1)+\cdots+
\Psi(x_k)
\]
uniformly on compact subsets of $M^k$.
\end{corollary}
\begin{pf}
As in the proof of Corollary~\ref{Corollary_multiplicativity_for_U} we
can first work with compact
subsets of $ M^k\setminus\bigcup_{i<j} \{x_i=x_j\}$ and then remove
the restriction $x_i\ne x_j$
using the analyticity. Notice that
\[
\frac{ ({\partial}/{(\partial x)} )^j S_{\lambda
}(x;N,1)}{S_{\lambda}(x;N,1)}\in \mathbb{Z} \biggl[\frac{\partial}{\partial x}\ln
\bigl(S_{\lambda
}(x;N,1)\bigr),\ldots,\frac{\partial^j
}{\partial x^j}\ln
\bigl(S_{\lambda}(x;N,1)\bigr) \biggr],
\]
that is, it is a polynomial in the derivatives of
$\ln(S_{\lambda}(x;N,1))$ of degree $j$ and so
\[
\frac{ (x({\partial}/{(\partial x)}) )^j S_{\lambda
}(x;N,1)}{S_{\lambda}(x;N,1)}\in \mathbb{Z} \biggl[x,\frac{\partial}{\partial x}\ln
\bigl(S_{\lambda
}(x;N,1)\bigr),\ldots,\frac{\partial^j
}{\partial x^j}\ln
\bigl(S_{\lambda}(x;N,1)\bigr) \biggr].
\]
Thus, when
\[
\lim_{N\to\infty} \frac{\ln (S_{\lambda(N)}(x;N,1) )}{N}
\]
exists, then
\[
\frac{ (x({\partial}/{(\partial x)}) )^j S_{\lambda
(N)}(x;N,1)}{N^j S_{\lambda(N)}(x;N,1)}
\]
converges and so does
\[
\frac{\det [\sum_{\ell=0}^{j-1} {D_{i,1}^{\ell
}[S_{\lambda
(N)}(x_i;N,1)]}/{N^\ell}
P_{j,\ell,N}(x_i) (x_i-1)^{\ell+k-j}  ]_{i,j=1}^k} {\prod_{i=1}^k
S_{\lambda(N)}(x_i;N,1)
\Delta(x_1,\ldots,x_k)}.
\]
Applying equation \eqref{eq_multivariate_expansion} to the last
expression, we get that
\[
\frac{ S_{\lambda(N)}(x_1,\ldots,x_k;N,1) }{\prod_{i=1}^k
S_{\lambda
(N)}(x_i;N,1) }
\]
converges to a bounded function and so does its logarithm
\[
\ln S_{\lambda(N)}(x_1,\ldots,x_k;N,1) - \sum
_{i=1}^k \ln S_{\lambda
(N)}(x_i;N,1).
\]
Dividing the last expression by $N$ and letting $N\to\infty$, we get
the statement.
\end{pf}

\begin{corollary}
\label{Corollary_multiplicativity_for_GUE} Suppose that for some
number $A$
\[
S_{\lambda(N)} \bigl(e^{{y}/{\sqrt{N}}};N,1 \bigr)e^{A\sqrt{N}y} \rightarrow
G(y)
\]
uniformly on compact subsets of domain
$\mathbb{D}\subset\mathbb{C}$ as $N\rightarrow\infty$. Then
%
\begin{eqnarray}
&&\lim_{N\to\infty}S_{\lambda(N)} \bigl(e^{{y_1}/{\sqrt{N}}},
\ldots,e^{{y_k}/{\sqrt{N}}};N,1 \bigr)\exp \bigl(A\sqrt {N}(y_1+\cdots
+y_k) \bigr)
\nonumber
\\[-8pt]
\\[-8pt]
\nonumber
&&\qquad=\prod_{i=1}^k
G(y_i)
\end{eqnarray}
uniformly on compact subsets of $\mathbb{D}^k$.
\end{corollary}
\begin{pf}
Let $S_{\lambda(N)}(e^{y/\sqrt{N}};N,1) e^{A\sqrt{N}y} = G_N(y)$. Since
$G_N(y)$ are entire
functions, $G(y)$ is analytic on $\mathbb{D}$. Notice that
\[
x\frac{\partial}{\partial x}f \bigl(\sqrt{N}\ln(x) \bigr) = \sqrt{N} f'
\bigl(\sqrt{N}\ln(x) \bigr),
\]
therefore
\begin{eqnarray*}
\biggl(x\frac{\partial}{\partial x} \biggr)^\ell S_{\lambda
(N)}(x;N,1) &=&
N^{\ell/2} \biggl[ \frac{\partial^\ell}{\partial y^\ell} \bigl(G_N(y)e^{-A\sqrt{N}y}
\bigr) \biggr]_{y=\sqrt{N}\ln x}
\\
&= &N^{\ell/2} \Biggl[\sum_{r=0}^\ell
\pmatrix{l
\cr
r}G_N^{(\ell
-r)}(y) (-A)^rN^{r/2}
e^{-A\sqrt{N}y} \Biggr]_{y=\sqrt{N}\ln x}
\\
&=&N^{\ell}(-A)^\ell \bigl[e^{-A\sqrt{N}y}G_N(y)
\bigl(1+O(1/\sqrt{N}) \bigr) \bigr]_{y=\sqrt
{N}\ln x},
\end{eqnarray*}
since the derivatives of $G_N(y)$ are uniformly bounded on compact
subsets of $\mathbb{D}$ as
$N\to\infty$. Further,
\[
(x-1)^\ell=N^{-\ell/2}y^\ell\bigl(1+O(1/\sqrt{N})
\bigr),\qquad x=e^{y/\sqrt{N}},
\]
and $P_{j,\ell,N} (e^{y/\sqrt{N}} )=1+O(1/\sqrt{N})$ with
$O(1/\sqrt{N})$ uniformly
bounded on compact sets. Thus, setting $x_i=e^{y_i/\sqrt{N}}$ in Proposition~\ref{Proposition_multivariate_expansion}, we get (for distinct $y_i$)
\begin{eqnarray*}
&&S_{\lambda(N)} \bigl(e^{y_1/\sqrt{N}},\ldots,e^{y_k/\sqrt
{N}};N,1
\bigr)e^{A\sqrt{N}(y_1+\cdots+y_k)}
\\
&&\qquad= \frac{1}{\Delta(x_1,\ldots,x_k)} \det \bigl[ (x_i-1)^{k-j}G_N(y_i)
\bigl(1+O(1/\sqrt{N}) \bigr) \bigr]_{i,j=1}^k
\\
&&\qquad=G_N(y_1)\cdots G_N(y_k)
\frac{\det [
(x_i-1)^{k-j} (1+O(1/\sqrt{N}) ) ]_{i,j=1}^k}{\Delta
(x_1,\ldots,x_k)}
\\
&&\qquad= G_N(y_1)\cdots G_N(y_k)
\bigl(1+O(1/\sqrt{N}) \bigr).
\end{eqnarray*}
Since the convergence is uniform, it also holds without the assumption
that $y_i$ are distinct.
\end{pf}

\subsection{Symplectic characters}
In this section we specialize the formulas of Section~\ref
{Section_general} to the characters
$\chi_{\lambda}$ of the symplectic group.

For $\mu\in\GTs_N^+$, let
\[
\A^s_{\mu}(x_1,\ldots,x_N) =
\frac{\det [ x_i^{\mu_j+1} -
x_i^{-\mu_j-1}
 ]_{i,j=1}^N}{\Delta(x_1,\ldots,x_N)}.
\]
Clearly, for $\lambda\in\GT_N^+$ we have
\[
\A^s_{\lambda+\delta}=\chi_{\lambda}(x_1,
\ldots,x_N)\frac{\prod_{i<j}(x_ix_j-1)\prod_i (x_i^2-1)}{
(x_1\cdots x_N)^N},
\]
where $\chi_\lambda$ is a character of the symplectic group $\operatorname{Sp}(2N)$.

\begin{proposition}
\label{Prop_Simplectic_satisfies}
Family $\A^s_{\mu}(x_1,\ldots,x_N)$
forms a class of determinantal functions with
\begin{eqnarray*}
\theta_i&=&q^{i}, \qquad g(x;m) =x^{m+1}-x^{-m-1},\qquad
\beta(x)=\frac{q^{x+1}-q^{-x-1}}{q-1},
\\
\alpha(x)&=&\frac{q^{x+1}+q^{-x-1}}{(q-1)^2},\qquad [Tf](x)=\frac{f(qx)+f(q^{-1}x)}{(q-1)^2},
\\
c_N&=&(q-1)^N \prod_{1\leq i<j \leq N}
\frac{(q-1)^2}{q^j-q^i} = \frac
{(q-1)^{N^2} }{(-1)^{{N\choose2}}\Delta(q,\ldots,q^N)}.
\end{eqnarray*}
\end{proposition}
\begin{pf}
Immediately following from the definitions and identity is the proof
\begin{eqnarray*}
&&\A^s_{\mu}\bigl(q,\ldots,q^n\bigr) \\
&&\qquad=
\biggl(\prod_i \bigl(q^{\mu_i+1}-q^{-\mu_i-1}\bigr)
\prod_{i<j}\bigl(q^{\mu_i+1}+q^{-\mu_i-1} - \bigl(q^{\mu_j+1}+q^{-\mu_j-1}\bigr)
\bigr)\biggr)\\
&&\qquad\quad{}\Big/\bigl((-1)^{{N\choose2}}\Delta\bigl(q,\ldots,q^n\bigr)\bigr).
\end{eqnarray*}
\upqed\end{pf}

Let us now specialize Proposition~\ref{Proposition_general_multivar}.

We have that
\begin{eqnarray*}
\X_{\lambda}(x_1,\ldots,x_k;N,q) &=&
\frac{\chi_{\lambda}(x_1,\ldots,x_k,q,\ldots,q^{N-k})}{\chi
_{\lambda
}(q,\ldots,q^{N})}
\\
&=& \frac{\dd(q,\ldots,q^N)
\Delta(x_1,\ldots,x_k,q,\ldots,q^{N-k})}{\dd(x_1,\ldots
,x_k,q,\ldots
,q^{N-k})\Delta(q,\ldots,q^N)}\\
&&{}\times \frac{\A_{\mu}^s(x_1,\ldots,x_k,q,\ldots,q^{N-k})}{\A_{\mu
}^s(q,\ldots,q^N)}.
\end{eqnarray*}

\begin{theorem}
\label{theorem_simplectic_multi_q} For any signature $\lambda\in\GT
_N^+$ and any $k\le N$, we have
%
\begin{eqnarray}
&&\X_{\lambda}(x_1,\ldots,x_k;N,q) \nonumber\\
&&\qquad=
\frac{\dd(q,\ldots,q^N)
(q-1)^{k^2-k} (-1)^{{k\choose2}}}{\dd(x_1,\ldots,x_k,q,\ldots,q^{N-k})}
\\
&&\qquad\quad{}\times
\det\bigl[ \bigl(D^s_{q,i}\bigr)^{j-1}
\bigr]_{i,j=1}^k\prod_{i=1}^k
\X_{\lambda}(x_i;N,q)\frac{\Delta_s(x_i,q,\ldots,q^{N-1})}{\Delta
_s(q,\ldots,q^N)},\nonumber 
\end{eqnarray}
where $D^s_{q,i}$ is the difference operator
\[
f(x)\to\frac{f(qx)+f(q^{-1}x)-2f(x)}{(q-1)^2}
\]
acting on variable $x_i$.
\end{theorem}
\begin{remark*} Note that in Proposition~\ref{Prop_Simplectic_satisfies}
the difference operator
differed by the shift $2/(q-1)^2$. This is still valid, since in either
case the operator is equal to $\prod_{i<j} (D^s_{q,i}-D^s_{q,j})$, and
the additional shifts cancel. However, the operator $D^s_{q,i}$ is well
defined when $q\to1$, which is used later.
\end{remark*}

Using Proposition~\ref{Proposition_integral_general} and computing the
coefficient in front of the
integral by straightforward algebraic manipulations we get the following.

\begin{theorem}\label{Theorem_symplectic_integral_q}
For any signature $\lambda\in\GT_N^+$ and any $q \neq1$ we have
%
\begin{eqnarray}
\label{formula:symplectic_integral_q} &&\X_{\lambda}(x;N,q)\nonumber\\
&&\qquad = \frac{(-1)^{N-1}\ln(q) (q-1)^{2N-1}[2N]_q!
}{(xq;q)_{N-1}
(x^{-1}q;q)_{N-1} (x-x^{-1})[N]_q}
\\
&&\qquad\quad{}\times\frac{1}{2\pi\ii} \oint \frac{
(x^{z+1}-x^{-z-1})}{\prod_{i=1}^N  (q^{z+1}+q^{-z-1} -
q^{-\lambda
_i+N-i-1} -q^{\lambda_i
+N-i+1}  )} \,dz
\nonumber
\end{eqnarray}
with contour $C$ enclosing the singularities of the integrand at
$z=\lambda_1+N-1,\ldots,\lambda_N$.
\end{theorem}

Theorem~\ref{Theorem_symplectic_integral_q} looks very similar to the
integral representation for
Schur polynomials, this is summarized in the following statement.

\begin{proposition}\label{symplectic_via_schur}
Let $\lambda\in\mathbb{GT}_N^+$. We have
\[
\X_{\lambda}(x;N,q) = \frac{ (1+q^N) }{x+1} S_{\nu}
\bigl(xq^{N-1};2N,q\bigr),
\]
where $\nu\in\GT_{2N}$ is a signature of size $2N$ given by $\nu_i =
\lambda_i+1$ for
$i=1,\ldots,N$ and $\nu_{i}=-\lambda_{2N-i+1}$ for $i=N+1,\ldots, 2N$.
\end{proposition}
\begin{pf}
First notice that for any $\mu\in\GTs^+_N$, we have
\begin{eqnarray*}
&&\oint_C \frac{(x^{z}-x^{-z})}{\prod_i (q^{z}+q^{-z}-q^{-\mu
_i-1}-q^{\mu
_i+1})}\,dz\\
&&\qquad= \oint_{C'}
\frac{x^{z}}{\prod_i (q^{z}+q^{-z}-q^{-\mu_i-1}-q^{\mu_i+1})}\,dz,
\end{eqnarray*}
where $C$ encloses the singularities of the integrand at $z=\lambda
_1+N-1,\ldots,\lambda_N$, and
$C'$ encloses all the singularities. Indeed, to prove this just write
both integrals as the sums
or residues. Further,
\[
q^z+q^{-z}-q^{-\mu_i-1}-q^{\mu_i+1} =
\bigl(q^z-q^{\mu_i+1}\bigr) \bigl(q^z-q^{-\mu
_i-1}
\bigr)q^{-z}.
\]
Therefore, the integrand in \eqref
{formula:symplectic_integral_q} transforms into
%
\begin{eqnarray}
\label{eq_x22} &&\frac{x^{z}q^{Nz}}{\prod_i
(q^{z}-q^{\lambda_i+N-i+1})(q^{z}-q^{-(\lambda_i+N-i)-1})}
\nonumber
\\[-8pt]
\\[-8pt]
\nonumber
&&\qquad =\frac{ (xq^{N-1} )^{z'}q^{z'}x^{-N}q^{N^2}}{\prod_i
(q^{z'}-q^{\lambda_i+1+2N-i})(q^{z'}-q^{-(\lambda_i+1-i)})},
\end{eqnarray}
where $z'=z+N$. The contour integral of \eqref{eq_x22} is readily
identified with that of Theorem~\ref{Theorem_Integral_representation_Schur_q} for $S_{\nu
}(xq^{N-1};2N,q)$. It remains only to
match the prefactors.
\end{pf}

Next, sending $q\to1$ we arrive at the following 3 statements:

Define
%
\begin{eqnarray}
\label{dd_q1} &&\Delta_s^1\bigl(x_1,
\ldots,x_k,1^{N-k}\bigr)\nonumber\\
 &&\qquad= \lim_{q \rightarrow1}
\frac
{\dd
(x_1,\ldots,x_k,q,\ldots,q^{N-k})}{(q-1)^{{N-k+1\choose2}}}
\\
&&\qquad= \dd(x_1,\ldots,x_k) \prod
_i\frac{ (x_i-1)^{2(N-k)}}{x_i^{N-k}} \prod_{1\leq i <j \leq
N-k}
\bigl(i^2-j^2\bigr) 2^{N-k}(N-k)!.\nonumber
\end{eqnarray}

\begin{theorem}\label{Theorem_symp_multivar_1}
For any signature $\lambda\in\GT_N^+$ and any $k\le N$, we have
%
\begin{eqnarray}
&&\X_{\lambda}(x_1,\ldots,x_k;N,1) \nonumber\\
&&\qquad=
\frac{\Delta_s^1(1^N)}{\Delta
_s^1(x_1,\ldots,x_k,1^{N-k})}
\nonumber
\\[-8pt]
\\[-8pt]
\nonumber
&&\qquad\quad{}\times
(-1)^{{k\choose2}} \det \biggl[ \biggl(x_i\frac{\partial}{\partial x_i}
\biggr)^{2(j-1)} \biggr]_{i,j=1}^k \\
&&\qquad\quad{}\times\prod
_{i=1}^k \X_{\lambda}(x_i;N,1)
\frac{(x_i-x_i^{-1})
(2-x_i-x_i^{-1})^{N-1} }{2(2N-1)!}.
\nonumber
\end{eqnarray}
\end{theorem}
\begin{remark*} The statement of Theorem~\ref{Theorem_symp_multivar_1}
was also proved by de Gier
and Ponsaing; see \cite{GP}.
\end{remark*}

\begin{theorem}\label{Theorem_symplectic_integral_1}
For any signature $\lambda\in\GT_N^+$ we have
\begin{eqnarray*}
\X_{\lambda}(x;N,1) &= &\frac{2(2N-1)!}{(x-x^{-1}) (x+x^{-1}-2)^{N-1}}
\\
&&{}\times\frac{1}{2\pi
\ii} \oint_C \frac{(x^z-x^{-z})}{\prod_{i=1}^N (z-
(\lambda_i+N-i+1))(z+\lambda_i+N-i+1)}\,dz,
\end{eqnarray*}
where the contour includes only the poles at $\lambda_i+N-i+1$ for
$i=1,\ldots,N$.
\end{theorem}

\begin{proposition}
\label{Proposition_Schur_Simplectic_1} For any signature $\lambda\in
\GT
_N^+$ we have
%
\begin{equation}
\X_{\lambda}(x;N,1) = \frac{2}{x+1}S_{\nu}(x;2N,1),
\end{equation}
where $\nu\in\GT_{2N}$ is a signature of size $2N$ given by $\nu_i =
\lambda_i+1$ for
$i=1,\ldots,N$ and $\nu_{i}=-\lambda_{2N-i+1}$ for $i=N+1,\ldots, 2N$.
\end{proposition}
\begin{remark*}We believe that the statement of Proposition~\ref
{Proposition_Schur_Simplectic_1}
should be known, but we are unable to locate it in the literature.
\end{remark*}

Analogously to the treatment of the multivariate Schur case, we can
also derive the same statements
as in Proposition~\ref{Proposition_multivariate_expansion} and Corollaries
\ref{Corollary_multiplicativity_for_U}, \ref{mult_logarithm},
\ref{Corollary_multiplicativity_for_GUE} for the multivariate
normalized symplectic characters.

\subsection{Jacobi polynomials}
Here we specialize the formulas of Section~\ref{Section_general} to the
multivariate Jacobi
polynomials. We do not present the formula for the $q$-version of
\eqref{eq_Jacobi_normalized},
although it can be obtained in a similar way.

Recall that for $\lambda\in\GT_N^+$,
\[
J_{\lambda}(z_1,\ldots,z_k;N,a,b) =
\frac{\mathfrak{J}_{\lambda
}(z_1,\ldots,z_k,1^{N-k};
a,b)}{\mathfrak{J}_{\lambda}(1^N;a,b)}.
\]

We produce the formulas in terms of polynomials $\mathfrak{P}_\mu$,
$\mu
\in\GTs_N^+$ and, thus,
introduce their normalizations as
%
\[
P_\mu(x_1,\ldots,x_k;N,a,b)=
\frac{\mathfrak{P}_{\mu}
(x_1,\ldots,x_k,
1^{N-k};a,b )}{\mathfrak{P}_{\mu}(1^N;a,b)}.
\]
These normalized polynomials are related to the normalized Jacobi via
\[
J_{\lambda}(z_1,\ldots,z_k;N,a,b) =
P_{\mu} \biggl(\frac
{z_1+z_1^{-1}}{2},\ldots,\frac{z_k+z_k^{-1}}{2};N,a,b
\biggr),
\]
where as usual $\lambda_i+N-i=\mu_i$ for $i=1,\ldots,N$.
%
\begin{proposition}
The polynomials $\mathfrak{P}_\mu(x_1,\ldots,x_N)$, $\mu\in\GTs_N^+$
are a class of determinantal functions with
\begin{eqnarray*}
\theta_i&=&1,\qquad g(x;m) ={\mathfrak p}_m(x;a,b),\qquad
\alpha(x)=x(x+a+b+1),\\
 \beta(x)&=&\frac{\Gamma(x+a+1)}{\Gamma(x+1)\Gamma(a)},
\\
c_N&=&\prod_{r=1}^N
\frac{\Gamma(r)\Gamma(a)}{\Gamma(r+a)} \prod_{1\leq
i<j\leq N}\frac{1}{(j-i)(2N-i-j+a+b+1)},
\\
T&=&\bigl(x^2-1\bigr) \frac{\partial^2}{\partial x^2} + \bigl((a+b+2)x +a-b\bigr)
\frac
{\partial}{\partial x}.
\end{eqnarray*}
\end{proposition}

\begin{pf}
We have (see, e.g., \cite{OkOlsh_BC}, Section~2C, and references therein)
%
\begin{eqnarray}
\mathfrak{P}_{\mu}\bigl(1^n;a,b\bigr)& =& \prod
_i \frac{\Gamma(\mu
_i+a+1)}{\Gamma
(\mu_i+1)} \times\prod
_{i<j} (\mu_i-\mu_j) (
\mu_i+\mu_j+a+b+1)
\nonumber
\\[-8pt]
\\[-8pt]
\nonumber
&&{}\times\prod_{r=1}^n \frac{\Gamma(r)}{\Gamma(r+a)}
\prod_{0\leq i<j
<n}\frac{1}{(j-i)(i+j+a+b+1)},
\nonumber
\end{eqnarray}
and also (see, e.g., \cite{Ed,Koekoek})
\[
m(m+2\sigma) p_m(x;a,b) = \biggl[\bigl(x^2-1\bigr)
\frac{\partial^2}{\partial x^2} + \bigl((a+b+2)x +a-b\bigr)\frac{\partial}{\partial x} \biggr]
p_m(x;a,b).
\]
Now the statement follows from the definition of polynomials $\mathfrak
{P}_\mu$.
\end{pf}

Specializing Proposition~\ref{Proposition_general_multivar}, using the
fact that for $x=\frac{z+z^{-1}}{2}$ we have $ \frac{\partial
}{\partial
x} = \frac{2}{1-z^{-2}}\frac{\partial}{\partial z}$ and $P_\mu
(x)=J_\lambda(z)$,
we obtain the following.
%
\begin{theorem}
\label{Theorem_Jacobi_multi} For any $\lambda\in\GT_N^+$ and any
$k\le
N$, we have
%
\begin{eqnarray}
&&J_{\lambda}(z_1,\ldots,z_k;N,a,b)\nonumber
\\
&&\qquad= \prod_{m=N-k+1}^N \frac{\Gamma(m+a)\Gamma(2m-1+a+b)}{\Gamma(m+a+b)} \cdot
\frac{1}{\prod_{i=1}^k (z_i+z_i^{-1}-2)^{N-k}}
\nonumber
\\[-8pt]
\\[-8pt]
\nonumber
&&\qquad\quad{}\times \frac{\det[\mathcal D_{i,a,b}^{j-1}]_{i,j=1}^k}{
2^{{k\choose2}} \Delta(z_1+z_1^{-1},\ldots,z_k+z_k^{-1})} \\
&&\qquad\quad{}\times\prod
_{i=1}^k J_{\lambda}(z_i;N,a,b)
\frac{(z_i+z_i^{-1}-2)^{N-1} \Gamma
(N+a+b)}{\Gamma(N+a)\Gamma(2N-1+a+b)},\nonumber
\end{eqnarray}
where $\mathcal D_{i,a,b}$ is the differential operator
\[
z_i^2
\frac{\partial^2}{\partial z_i^2}+ \frac
{((a+b+2)(z_i+z_i^{-1})+2a-2b-2z_i^{-1})}{1-z_i^{-2}} \frac{\partial
}{\partial z_i}.
\]
%
\end{theorem}

Next, we specialize Proposition~\ref{Proposition_integral_general} to
the case of multivariate Jacobi
polynomials. Note that thanks to the symmetry under $\zeta+(a+b+1)/2
\leftrightarrow-(\zeta+(a+b+1)/2)$
of the integrand we can extend the contour $C$ to include all the poles.
%
\begin{theorem}
\label{Theorem_Jacobi_singlevar} For any $\lambda\in\GT_N^+$ we have
%
\begin{eqnarray}
J_{\lambda}(z;N,a,b)&=&\frac{\mathfrak{J}_{\lambda
}(x,1^{N-1};a,b)}{\mathfrak{J}_{\lambda}(1^N;a,b)}\nonumber\\
&=& \frac{\Gamma(2N+a+b-1)}{\Gamma(n+a+b)\Gamma(a+1)}
\frac{1}{
(({(z+z^{-1})}/{2})-1 )^{N-1}}\nonumber\\
&&{}\times\frac{1}{2\pi\ii} \oint_C \biggl( {}_2F_1 \biggl(-\zeta,\zeta
+a+b+1;a+1;-\frac{(1-z)^2}{(4z)} \biggr)\\
&&\hspace*{142pt}{}\times\bigl(\zeta+(a+b+1)/2\bigr)\biggr)\nonumber\\
&&{}\Big/\biggl(\prod_i
(\zeta-\mu_i)(\zeta+\mu_i+a+b+1)\biggr)\,d\zeta,
\nonumber
\end{eqnarray}
where the contour includes the poles of the integrand at $\zeta
=-(a+b+1)/2 \pm(\mu_i+(a+b+1)/2)$ and $\mu_i=\lambda_i+N-i$ for
$i=1,\ldots,N$.
\end{theorem}

\section{General asymptotic analysis}

Here we derive the asymptotics for the single-variable normalized Schur
functions $S_{\lambda}(x;N,1)$.
In what follows $O$ and $o$ mean uniform estimates, not depending on
any parameters, and $\operatorname{const}$
stands for a positive constant which might be different from line to line.

\subsection{Steepest descent}\label{section:steepest_descent}

Suppose that we are given a sequence of signatures $\lambda(N)\in
\mathbb
{GT}_N$ [or, even, more
generally, $\lambda(N_k)\in\mathbb{GT}_{N_k}$ with
$N_1<N_2<N_3<\cdots
$]. We are going to study the
asymptotic behavior of $S_{\lambda(N)}(x;N,1)$ as $N\to\infty$ under
the assumption that there
exists a function $f(t)$ for which as $N\to\infty$, the vector
$(\lambda_1(N)/N,\ldots,\lambda_N(N)/N)$ converges to
$(f(1/N),\ldots
,f(N/N))$ in a certain sense
which is explained below.

Let $R_1,R_\infty$ denote the corresponding norms of the difference of
the vectors
$(\lambda_j(N)/{N})$ and $(f(j/N))$,
\[
R_1(\lambda,f)= \sum_{j=1}^N
\biggl\llvert \frac{\lambda
_j(N)}{N}-f(j/N)\biggr\rrvert , \qquad R_\infty(
\lambda,f)= \sup_{j=1,\ldots,N} \biggl\llvert \frac{\lambda
_j(N)}{N}-f(j/N)
\biggr\rrvert .
\]
In order to keep the computations compact we also introduce a modified
form $\widehat f(t)$ of the
function $f(t)$ via
\[
\widehat f(t)= f(t)+1-t.
\]
As in the previous sections, let $\mu(N) = \lambda(N)+ \delta_N$, so
$\widehat f$ is the limit of $\mu(N)/N$.
In order to state our results we introduce $w$, defined for any $y\in
\mathbb C$ by the equation
%
\begin{equation}
\label{eq_critical_point_equation_formulation} \int_0^1
\frac
{dt}{w-\widehat f(t)}=y.
\end{equation}
We remark that a solution to \eqref
{eq_critical_point_equation_formulation} can be interpreted as
an \emph{inverse Hilbert transform}. We also introduce the function
$\F(w;f)$
%
\begin{equation}
\label{eq_definition_F} \F(w;f)=\int_0^1 \ln\bigl(w-
\widehat f(t)\bigr)\,dt, \qquad w\in\mathbb C\setminus \bigl\{\widehat f(t)\mid t\in[0,1]
\bigr\}.
\end{equation}
Note that we need to specify which branch of the logarithm we choose in
\eqref{eq_definition_F}.
This choice is not very important at the moment, but it should be
consistent in all the formulas
which follow.

Observe that \eqref{eq_critical_point_equation_formulation} can be
rewritten as $\F'(w;f)=y$.

\begin{proposition}
\label{proposition_convergence_mildest} For $y \in\mathbb
{R}\setminus
\{0\}$, suppose that $f(t)$
is piecewise- continuous, $R_\infty(\lambda(N),f)$ is bounded,
$R_1(\lambda(N),f)/N$ tends to zero
as $N\to\infty$ and $w_0=w_0(y)$ is the (unique) real root of
\eqref{eq_critical_point_equation_formulation}.\vspace*{1pt} Further, let $y\in
\mathbb R\setminus\{ 0\}$ be
such that $w_0$ is outside the interval $[\frac{\lambda
_N(N)}{N},\frac
{\lambda_1(N)}{N}+1]$ for
all $N$ large enough. Then
%
\begin{equation}
\label{eq_x37} \lim_{N\to\infty}\frac{ \ln S_{\lambda(N)}(e^{y};N,1)}{N} = y
w_0-\F(w_0)-1-\ln\bigl(e^y-1\bigr).
\end{equation}
\end{proposition}
\begin{remark} When $y$ is positive, we can choose the branch of the
logarithm which has real
values at positive real points both in \eqref{eq_definition_F} and in
$\ln(e^{y}-1)$ inside
\eqref{eq_x37}. For negative $y$s we can choose the branch which has
the values with imaginary
part $\pi$.
\end{remark}

\begin{remark}
Note that piecewise-continuity of $f(t)$ is a
reasonable assumption since $f$ is
monotonic.
\end{remark}

\begin{remark} A somehow similar statement was proven by Guionnet and
Ma\"\i da; see
\cite{GM}, Theorem~1.2.
\end{remark}

When an accurate asymptotics of $\lambda(N)$ is known, Proposition~\ref{proposition_convergence_mildest} can be further refined. For
$w\in
\mathbb C$, denote [as
before $\mu_j(N)=\lambda_j(N)+N-j$]
%
\begin{equation}
\mathcal Q\bigl(w; \lambda(N),f\bigr)=\frac{\exp (N \F(w;f)  )}{
\prod_{j=1}^{N} (w -
{\mu_j(N)}/{N} )}.
\end{equation}


\begin{proposition}
\label{proposition_convergence_strongest} Let $y\in\mathbb R\setminus
\{ 0\}$ be such that
$w_0=w_0(y)$ [which is the (unique) real root of \eqref
{eq_critical_point_equation_formulation}] is
outside the interval $[\frac{\lambda_N(N)}{N},\frac{\lambda
_1(N)}{N}+1]$ for all large enough $N$.
Suppose that for a function $f(t)$
%
\begin{equation}
\label{eq_remainder_converging} \lim_{N\to\infty}\mathcal Q\bigl(w; \lambda(N),f
\bigr) =g(w)
\end{equation}
uniformly on an open $\mathcal M$ set in $\mathbb C$, containing $w_0$.
Assume also that $g(w_0)\ne0$ and $\F''(w_0;f)\ne0$. Then as $N\to
\infty$,
\[
S_{\lambda(N)}\bigl(e^{y};N,1\bigr) = \frac{g(w_0)}{\sqrt{-\F''(w_0;f)}} \cdot
\frac{\exp (N(yw_0-\F(w_0;f)) )} {e^N
(e^{y}-1
)^{N-1}}\cdot \bigl(1+o(1) \bigr).
\]
The remainder $o(1)$ is uniform over $y$ belonging to compact subsets
of $\mathbb R\setminus\{ 0\}$ and such that
$w_0=w_0(y)\in\mathcal M$.
\end{proposition}

\begin{remark*}If the complete asymptotic expansion of $\mathcal Q(w;
\lambda(N),f)$ as $N\to\infty$
is known, then, with some further work, we can obtain the expansion of
$S_{\lambda(N)}(e^{y};N,1)$
up to arbitrary precision. In such expansion, $o(1)$ in Proposition~\ref{proposition_convergence_strongest} is replaced by a power series
in $N^{-1/2}$ with
coefficients being the analytic functions of $y$. The general procedure
is as follows: we use the
expansion of $\mathcal Q(w; \lambda(N),f)$ (instead of only the first
term) everywhere in the below
proof and further obtain the asymptotic expansion for each term
independently through the steepest
descent method. This level of details is enough for our applications,
and we will not discuss it any
further; all the technical details can be found in any of the classical
treatments of the steepest
descent method; see, for example, \cite{Copson,Er}.
\end{remark*}
%
\begin{proposition}
\label{Prop_convergence_GUE_case} Suppose that $f(t)$ is
piecewise-differentiable,\break 
$R_\infty(\lambda(N), f)=O(1)$ (i.e., it is bounded) and $R_1(\lambda
(N),f)/\sqrt{N}$ goes to $0$
as $N\to\infty$. Then for any fixed $h\in\mathbb{R}$
\[
S_{\lambda(N)}\bigl(e^{h/\sqrt{N}};N,1\bigr)= \exp \bigl(\sqrt{N} E(f)h+
\tfrac{1}{2} S(f) h^2 + o(1) \bigr)
\]
as $N\to\infty$, where
\[
E(f)=\int_{0}^1 f(t) \,dt, \qquad S(f)= \int
_0^1 f(t)^2 \,dt -E(f)^2
+ \int_0^1 f(t) (1-2t) \,dt.
\]
Moreover, the remainder $o(1)$ is uniform over $h$ belonging to compact
subsets of $\mathbb
R\setminus0$.
\end{proposition}


We prove the above three propositions simultaneously.

We start investigating the asymptotic
behavior of the integral on the right-hand side of the integral
representation of Theorem~\ref{Theorem_Integral_representation_Schur_1},
%
\begin{equation}
\label{eq_integral_basic} S_\lambda\bigl(e^y;N,1\bigr) =
\frac
{(N-1)!}{(e^y-1)^{N-1}} \frac{1}{2\pi\ii} \oint_{C} \frac{ e^{yz}}{\prod_{j=1}^{N}(z - \mu_j(N))}
\,dz.
\end{equation}

Changing the variables $z=Nw$ transforms \eqref{eq_integral_basic} into
%
\begin{equation}
\label{eq_x2} \frac{(N-1)!}{(e^{y}-1)^{N-1}} N^{1-N} \frac{1}{2\pi
\ii}
\oint_{\mathcal C} \frac{
\exp(N yw)}{\prod_{j=1}^{N}(w - {\mu_j(N)}/{N})} \,dw.
\end{equation}

From now on we study the integral
%
\begin{eqnarray}
\label{eq_integral_transformed} &&\oint_{\mathcal C} \frac{ \exp(N
yw)}{\prod_{j=1}^{N} (w -
{\mu_j(N)}/{N} )} \,dw
\nonumber
\\[-8pt]
\\[-8pt]
\nonumber
&&\qquad=
\oint_{\mathcal C} \exp \bigl(N\bigl(yw-\F (w;f)\bigr) \bigr)\cdot \mathcal Q
\bigl(w; \lambda(N),f\bigr) \,dw,
\end{eqnarray}
where the contour $\mathcal C$ encloses all the poles of the integrand.

Note that $\operatorname{Re} ( \F(w;f) )$ is a continuous function in $w$,
while $\operatorname{Im} (
\F(w;f) )$ has discontinuities along the real axis (if we choose
the principal branch of
logarithm with a cut along the negative real axis), both these
functions are harmonic outside the
real axis.

In fact, the factor $\mathcal Q(w; \lambda(N),f)$ in \eqref
{eq_integral_transformed} has
subexponential growth. Indeed, under the assumptions of Proposition~\ref{proposition_convergence_strongest} this is automatically true,
while for other cases we use
the following two lemmas whose proofs are presented at the end of this section.

\begin{lemma} \label{Lemma_bound_logarithmic}
Let $A$ be the smallest interval in $\mathbb R$ containing all the
points $\{\widehat f(t)\mid0\le t \le1\}$ and $\{
\frac{\mu_j(N)}{N}\mid j=1,\ldots,N\}$. Under the assumptions of Proposition~\ref{proposition_convergence_mildest} as $N\to\infty$
\[
\ln\bigl|Q\bigl(w; \lambda(N),f\bigr)\bigr|\le o(N) \biggl(1+\sup_{a\in A}
\bigl|\ln(w-a)\bigr|+ \sup_{a\in A} \biggl\llvert \frac{1}{w-a
}\biggr
\rrvert \biggr),
\]
where $o(N)$ is uniform in $w$ outside $A$.
\end{lemma}
%
\begin{lemma} \label{Lemma_bound_GUE}
Let $A$ be the smallest interval in $\mathbb R$ containing all the
points $\{\widehat f(t)\mid0\le t \le1\}$ and $\{
\frac{\mu_j(N)}{N}\mid j=1,\ldots,N\}$. Under the assumptions of Proposition~\ref{Prop_convergence_GUE_case} as $N\to\infty$
\begin{eqnarray*}
\ln\bigl|Q\bigl(w; \lambda(N),f\bigr)\bigr|&\le& o(\sqrt{N}) \sup_{a\in A}
\biggl\llvert \frac{1}{w-a }\biggr\rrvert + O(1)\sup_{|t-s| \le1/N}
\biggl\llvert \ln \biggl(\frac{w-\widehat
f(t)}{w-\widehat f(s)} \biggr)\biggr\rrvert \\
&&{}+\sup
_{0\le t\le1}\biggl\llvert \frac{\widehat
f'(t)}{w-\widehat f(t)}\biggr\rrvert ,
\end{eqnarray*}
where $o(\sqrt{N})$ and $O(1)$ are uniform in $w$ outside $A$, and the
last $\sup$ is taken only over such
$t$ in which $\widehat f$ is differentiable.
\end{lemma}

The asymptotic analysis of the integrals of the kind \eqref
{eq_integral_transformed} is usually
performed using the so-called steepest descent method; see, for
example, \cite{Copson,Er}. We
will deform the contour to pass through the critical point of $yw-\F
(w;f)$. This point satisfies
the equation
%
\begin{equation}
\label{eq_critical_point_equation} 0=\bigl(yw-\F(w;f)\bigr)'= y- \int
_0^1 \frac{dt}{w-\widehat f(t)}.
\end{equation}
In general, equation \eqref{eq_critical_point_equation} [which is the
same as
\eqref{eq_critical_point_equation_formulation}] may have several roots,
and one has to be careful
to choose the needed one.

\begin{lemma} \label{lemma_critical_point_existence} Suppose that
$y\in
\mathbb R\setminus\{0\}$. If $y>0$, then
\eqref{eq_critical_point_equation} has a unique real root
$w_0(y)>\widehat f(0)$. If $y<0$, then
\eqref{eq_critical_point_equation} has a unique real root
$w_0(y)<\widehat f(1)$. Further,
$w_0(y)\to\infty$ as $y\to0$.
\end{lemma}
\begin{pf}
For $y>0$ the statement follows from the fact that the integral in
\eqref{eq_critical_point_equation} is a monotonic function of
$w>\widehat f(0)$
changing from $+\infty$ down to zero (when $w\to+\infty$). Similarly,
for $y<0$ we use the that the
integral in \eqref{eq_critical_point_equation} is a monotonic function
of $w<\widehat f(1)$
changing from zero (when $w\to-\infty$) down to $-\infty$.
\end{pf}
In what follows, without loss of generality, we assume that $y>0$, and
use $w_0=w_0(y)$ of Lemma~\ref{lemma_critical_point_existence}.

Next, we want to prove that one can deform the contour $\mathcal C$
into $\mathcal C'$ which
passes through $w_0$ in such a way that $\operatorname{Re} (yw-\F(w;f))$ has maximum
at $w_0$. The fact that $y$
is real simplifies the choice of the contour.

Let $\mathcal C'$ be the vertical line passing through $w_0$. We claim
that the contour $\mathcal
C$ in \eqref{eq_integral_transformed} can be deformed into $\mathcal
C'$ without changing the value
of integral. Indeed, observe that the integrand in \eqref
{eq_integral_transformed} decays like
$|w|^{-N}$ as $|w|\to\infty$ in such way that $\operatorname{Re}(w)$ stays bounded
from above. Therefore, for
$N\ge2$ we can deform the contour as desired.

We will now study the integral over $w\in\mathcal C'$. The definitions
immediately imply that
\[
\operatorname{Re} \bigl(yw-\F(w;f)\bigr) < \operatorname{Re}\bigl(yw_0-\F(w_0;f)
\bigr),\qquad w\in\mathcal C', w\ne w_0.
\]

Now the integrand is exponentially small in $N$ (compared to its value
at $w_0$) everywhere on the
contour $\mathcal C'$ outside arbitrary neighborhood of $w_0$. Inside a
small $\eps$-neighborhood
of $w_0$ we can do the Taylor expansion for $yw-\F(w;f)$,
\[
yw-\F(w;f)=yw_0-\F(w_0;f)-\frac{(w-w_0)^2}{2}\cdot
\F''(w_0;f) + (w-w_0)^3
\cdot\delta,
\]
where the absolute value of the remainder $\delta$ is bounded by the
maximum of $|\F'''(w;f)|$ in
the $\eps$-neighborhood.

Note that $\F''(w_0;f)<0$, and denote $u=-\ii\sqrt{-\F''(w_0;f)}$.
Setting $w=w_0+ s/(u\sqrt{N})$,
and choosing a small $\eps>0$, whose exact value will be specified later,
\eqref{eq_integral_transformed} is approximated by
%
\begin{eqnarray}
\label{eq_x3} &&\exp \bigl(N\bigl(yw_0-\F(w_0;f)
\bigr) \bigr)\nonumber
\\
&&\quad{}\times\int_{w_0-i\varepsilon
}^{w_0+i\varepsilon} \exp \bigl(-N
\F''(w_0;f) (w-w_0)^2/2
+ N\delta (w-w_0)^3 \bigr)\nonumber\\
&&\hspace*{25pt}\qquad\quad\times{}\mathcal Q\bigl(w;
\lambda(N),f\bigr) \,dw
\nonumber
\\[-8pt]
\\[-8pt]
\nonumber
&&\qquad= \frac{\exp (N(yw_0-\F(w_0;f)) )}{u\sqrt{N}}
\\
&&\qquad\quad{}\times \int_{-\sqrt{N}\varepsilon|u|}^{+\sqrt{N}\varepsilon|u|} \exp \bigl(-s^2/2+
s^3\tilde\delta/\sqrt{N} \bigr) \mathcal Q \bigl(w_0+s
/(u\sqrt{N}); \lambda(N),f \bigr) \,ds\nonumber
\\
&&\qquad\approx\sqrt{2\pi} \frac{1}{u\sqrt{N}} \mathcal Q\bigl(w_0;
\lambda (N),f\bigr)\exp \bigl(N\bigl(yw_0-\F(w_0;f)
\bigr) \bigr),
\nonumber
\end{eqnarray}
where
\[
|\tilde\delta|\le|u|^{-3} \sup_{w\in[w_0-i\varepsilon
,w_0+i\varepsilon
]}
\F'''(w;f).
\]

When we approximate the integral over vertical line by the integral
over the
$\varepsilon$-neighborhood [reduction of \eqref
{eq_integral_transformed} to the first line in
\eqref{eq_x3}] the relative error can be bounded as
%
\begin{eqnarray}
\label{eq_x4}&& \operatorname{const}
\times\exp\bigl(N \operatorname{Re} \bigl(\F(w_0+i\varepsilon;f)-
\F(w_0;f)\bigr)\bigr)
\nonumber
\\[-8pt]
\\[-8pt]
\nonumber
&&\qquad \approx \operatorname{const}\times\exp\bigl(-N
\varepsilon^2\bigl|\F''(w_0;f)\bigr|/2
\bigr).
\end{eqnarray}
Next, we estimate the relative error in the approximation in \eqref
{eq_x3} [i.e., the sign
$\approx$ in \eqref{eq_x3}]. Suppose that $\varepsilon<|u/\delta|/2$,
and divide the integration
segment into a smaller subsegment $|s|<N^{-1/10}\sqrt[3]{\sqrt
{N}/|\tilde\delta|}$ and its
complement. When we omit the $s^3$ term in the exponent, we get the
relative error at most
$\operatorname{const}\times N^{-3/10}$ when integrating over the smaller subsegment
[which comes from the factor
$\exp (s^3\tilde\delta/\sqrt{N} )$ itself] and
$\operatorname{const}\times
\exp (-N^{-2/15}|\tilde\delta|^{-2/3}/4 )$ when integrating
over its complement (which
comes from the estimate of the integral on this complement).

When we replace the integral over $[-\sqrt{N}\varepsilon|u|,+\sqrt
{N}\varepsilon|u|]$ by the
integral over $(-\infty,+\infty)$ in \eqref{eq_x3}, we get the error
\[
\operatorname{const}\exp\bigl(-N\varepsilon^2 \bigl|u^2\bigr|/2\bigr).
\]
Finally, there is an error of
\[
\operatorname{const}\sup_{w\in[w_0-i\varepsilon,w_0+i\varepsilon]}\bigl|\mathcal Q \bigl(w; \lambda(N),f \bigr)-
\mathcal Q \bigl(w_0; \lambda(N),f \bigr)\bigr|
\]
coming from the factor $\mathcal Q (w_0+s
/(u\sqrt{N}); \lambda(N),f )$. Summing up, the total relative
error in the approximation in
\eqref{eq_x3} is at most constant times
%
\begin{eqnarray}
\label{eq_x5} &&N^{-3/10}+\exp \bigl(-N^{-2/15}|\tilde\delta
|^{-2/3}/4 \bigr) + \exp\bigl(-N\varepsilon^2
\bigl|u^2\bigr|/2\bigr)
\nonumber
\\[-8pt]
\\[-8pt]
\nonumber
&&\qquad{}+ \sup_{w\in[w_0-i\varepsilon,w_0+i\varepsilon]}\bigl|\mathcal Q \bigl(w; \lambda(N),f \bigr)-\mathcal
Q \bigl(w_0; \lambda(N),f \bigr)\bigr|.
\nonumber
\end{eqnarray}

Combining \eqref{eq_x2} and \eqref{eq_x3} we get
\begin{eqnarray*}
&&\frac{s_{\lambda}(e^y,1^{N-1})}{s_{\lambda}(1^N)}
\\
&&\qquad\approx \frac{1}{\sqrt{2\pi}}\frac{(N-1)!}{(e^{y}-1)^{N-1}}
\\
&&\qquad\quad{}\times N^{1-N}\frac{1}{\sqrt
{-\F''(w_0;f)}\sqrt{N}} \mathcal Q \bigl(w_0;
\lambda(N),f \bigr)\exp\bigl(N\bigl(yw_0-\F(w_0;f)
\bigr)\bigr).
\end{eqnarray*}
Using Stirling's formula we arrive at
%
\begin{eqnarray}
\label{eq_final_approximation} &&\frac{s_{\lambda
}(e^y,1^{N-1})}{s_{\lambda}(1^N)}
\nonumber
\\[-8pt]
\\[-8pt]
\nonumber
&&\qquad\approx \frac{1}{e^N(e^{y}-1)^{N-1}}\frac{\mathcal Q (w_0; \lambda
(N),f )}{\sqrt{-\F''(w_0;f)}}
\exp\bigl(N\bigl(yw_0-\F(w_0;f)\bigr)\bigr),
\end{eqnarray}
with the relative error in \eqref{eq_final_approximation} being the sum
of \eqref{eq_x4},
\eqref{eq_x5} and $O(1/N)$ coming from Stirling's approximation, and
$\varepsilon$ satisfying
$\varepsilon<|u/\delta|/2$.

Now we are ready to prove the three statements describing the
asymptotic behavior of normalized Schur
polynomials.

\begin{pf*}{Proof of Proposition~\ref{proposition_convergence_mildest}}
Use \eqref{eq_final_approximation} and Lemma~\ref
{Lemma_bound_logarithmic}, and note that after taking logarithms and
dividing by $N$ the relative error in
\eqref{eq_final_approximation} vanishes.
\end{pf*}

\begin{pf*}{Proof of Proposition~\ref{proposition_convergence_strongest}}
Again this follows from \eqref{eq_final_approximation}. It remains to check
that the error term in \eqref{eq_final_approximation} is negligible.
Indeed, all the derivatives
of $\F$, as well as $|u|$, $|\delta|$, $|\tilde\delta|$ are bounded in
this limit regime. Thus,
choosing $\varepsilon=N^{-1/10}$ we conclude that all the error terms vanish.
\end{pf*}

\begin{pf*}{Proof of Proposition~\ref{Prop_convergence_GUE_case}}
Equation \eqref{eq_critical_point_equation} for $w_0$ reads
\[
h/\sqrt{N}- \int_0^1 \frac{dt}{w_0-\widehat f(t)}=0.
\]
Clearly, as $N\to\infty$ we have $w_0\approx\sqrt{N}/h\to\infty$. Thus
we can write
\[
\int_0^1 \frac{dt}{w_0-\widehat f(t)}
\\
=\frac{1}{w_0} \int_0^1 \biggl(1+
\frac{\widehat f(t)}{w_0} + \biggl(\frac{\widehat f(t)}{w_0} \biggr)^2 + O \biggl(
\frac
{1}{w_0^3} \biggr) \biggr) \,dt.
\]
Denote
\[
A=\int_0^1 \widehat f(t)\,dt,\qquad  B= \int
_0^1 \bigl(\widehat f(t) \bigr)^2
\,dt,
\]
and rewrite \eqref{eq_critical_point_equation} as
\[
w_0^2-w_0 \frac{\sqrt{N}}{h} -
\frac{A \sqrt{N}}{h} =O(1).
\]
If follows that as
$N\to\infty$, we have
%
\begin{eqnarray}
\label{eq_w_as_expansion} %
w_0&= &\frac{\sqrt{N}}{2h} +
\frac{1}2\sqrt{ \frac{N}{h^2}+4 \frac
{A\sqrt
{N}}{h}} + O(1/
\sqrt{N})\nonumber\\
&=&\frac{\sqrt{N}}{h} + A + O (1/\sqrt {N} )\quad\mbox{and alternatively }
\\
 \frac{1}{w_0} &=& \frac{h}{\sqrt{N}} -
\frac{Ah^2}{N} +O \bigl(N^{-3/2} \bigr).\nonumber
\end{eqnarray}

Next, let us show that the error in \eqref{eq_final_approximation} is
negligible. For this, choose
$\varepsilon$ in \eqref{eq_x3} to be $N^{1/10}$. Note that $|\F
''(w_0;f)|$ is of order $N^{-1}$,
and $|\F'''(w;f)|$ (and, thus, also $|\delta|$) is of order $N^{-3/2}$
on the integration contour
and $|u|$ is of order $N^{-1/2}$. The inequality $\varepsilon
<|u/\delta
|/2$ is satisfied.
The term coming from \eqref{eq_x4} is bounded by $\exp(-\operatorname{const}\times
N^{1/5})$ and is negligible.
As for \eqref{eq_x5} the first term in it is negligible, the second
one is bounded by\break 
$\exp(-\operatorname{const}\times N^{2/15})$ and negligible, the third one is bounded
by\break  $\exp(-\operatorname{const}\times
N^{1/5})$ which is again negligible. Turning to the fourth term, Lemma~\ref{Lemma_bound_GUE} and
asymptotic expansion \eqref{eq_w_as_expansion} imply that both
$\mathcal Q(w;\lambda(N),f)$ and
$\mathcal Q(w_0;\lambda(N),f)$ can be approximated as $1+o(1)$ as
$N\to
\infty$, and we are done.

Note that
\[
\frac{e^{h/\sqrt N}-1}{\sqrt{-\F''(w_0;f)}} = 1+o(1)
\]
as $N\to\infty$. Now
\eqref{eq_final_approximation} yields that
%
\begin{eqnarray}\qquad
\label{eq_x6}&& \frac{s_{\lambda}(e^{h/\sqrt{N}},1^{N-1})}{s_\lambda
(1^N)}
\nonumber
\\[-8pt]
\\[-8pt]
\nonumber
&&\qquad= \exp \bigl(N \bigl(-1-\ln
\bigl(e^{h/\sqrt N} -1\bigr) + hw_0/\sqrt N -\F
(w_0;f) \bigr) \bigr) \bigl(1+o(1)\bigr).
\end{eqnarray}
As $N\to\infty$ using the Taylor expansion of the logarithm, we have
\begin{eqnarray*}
\F(w_0;f)&=&\int_0^1\ln
\bigl(w_0-\widehat f(t)\bigr)\,dt\\
&=&\ln(w_0)+\int
_0^1 \biggl(-\frac{\widehat f(t)}{w_0}-
\frac{(\widehat f(t))^2}{2w_0^2} + O \biggl(\frac{1}{w_0^3} \biggr) \biggr) \,dt
\\
&=&\ln(w_0)-\frac
{A}{w_0}-\frac{B}{2w_0^2}+O \biggl(
\frac{1}{w_0^3} \biggr),
\end{eqnarray*}
and using \eqref{eq_critical_point_equation} together with \eqref
{eq_w_as_expansion},
\[
\frac{hw_0}{\sqrt N} =1+ \frac{A}{w_0} + \frac{B}{w_0^2}+ O \biggl(
\frac
{1}{w_0^3} \biggr) = 1 + \frac{Ah}{\sqrt{N}}-\frac{A^2h^2}{N}+
\frac
{Bh^2}{N}+O\bigl(N^{-3/2}\bigr).
\]
Thus
\begin{eqnarray*}
&&-1-\ln\bigl(e^{h/\sqrt N} -1\bigr) + hw_0/\sqrt N -
\F(w_0;f)
\\
&&\qquad= -\ln \bigl(w_0\bigl(e^{h/\sqrt{N}} -1\bigr)\bigr)+
\frac{A}{w_0} + \frac{B}{w_0^2}+\frac{A}{w_0}+
\frac{B}{2w_0^2}+O\bigl(N^{-3/2}\bigr)
\\
&&\qquad=-\ln \biggl(\frac{w_0h}{\sqrt{N}} \biggl(1+\frac{h}{2\sqrt
{N}}+
\frac
{h^2}{6N}+O\bigl(N^{-3/2}\bigr) \biggr) \biggr)+
\frac{2A}{w_0}+\frac
{3B}{2w_0^2}+O\bigl(N^{-3/2}\bigr)
\\
&&\qquad=-\ln \biggl( 1 + \frac{Ah}{\sqrt{N}}+\frac{(B-A^2)h^2}{N} \biggr) - \ln
\biggl(1+\frac{h}{2\sqrt{N}}+\frac{h^2}{6N} \biggr)\\
&&\qquad\quad{}+\frac
{2Ah}{\sqrt
{N}}+
\frac{(({3}/2)B-2A^2)h^2}{N} +O\bigl(N^{-3/2}\bigr)
\\
&&\qquad=
\frac{A h}{\sqrt{N}} + \frac{B - A^2}{2}\cdot\frac{h^2}{N} -
\frac
{h}{2\sqrt{N}} -\frac{h^2}{24N} +O\bigl(N^{-3/2}\bigr).
\end{eqnarray*}
To finish the proof observe that
\[
A=E(f)+1/2,\qquad B= \int_0^1 f^2(t)
\,dt + 2\int_0^1f(t) (1-t)\,dt+ 1/3;
\]
thus \eqref{eq_x6} transforms into
\[
\exp\bigl( E(f)h\sqrt{N} + S(f)h^2/2 \bigr) \bigl(1+o(1)
\bigr).
\]
\upqed\end{pf*}

Now we prove Lemmas \ref{Lemma_bound_logarithmic} and \ref{Lemma_bound_GUE}.
\begin{pf*}{Proof of Lemma~\ref{Lemma_bound_logarithmic}}
We take the logarithm of $\mathcal Q(w;\lambda(N),f)$ and aim to prove
that the result is small.
For that observe the following estimate:
%
\begin{eqnarray}\qquad
\label{eq_x31}&& \Biggl\llvert \sum_{j=1}^N
\ln \biggl(w- \frac{\mu_j(N)}{N} \biggr)-\sum_{j=1}^N
\ln \bigl(w- \widehat f(j/N)\bigr)\Biggr\rrvert
\nonumber
\\[-8pt]
\\[-8pt]
\nonumber
&&\qquad \le
\sum_{j=1}^N \biggl\llvert \int
_{{\mu_j(N)}/{N}}^{\widehat
f(j/N)} \frac{dx}{w-x}\biggr\rrvert \le
\frac{1}{N}\cdot\sup_{a\in A} \biggl\llvert
\frac{1}{w-a }\biggr\rrvert \cdot \sum_{j=1}^N
\bigl|\lambda_j(N)-f(j/N)\bigr|.
\nonumber
\end{eqnarray}
Further, using a usual second-order approximation of the integral
(trapezoid formula) we can write
%
\begin{eqnarray}
\label{eq_x32} &&\sum_{j=1}^{N} \ln
\bigl(w- \widehat f(j/N)\bigr) \nonumber\\
&&\qquad= N \Biggl( \sum_{j=1}^{N}
\frac
{\ln
(w- \widehat f(j/N))}{N} \Biggr)\nonumber
\\
&&\qquad= N \int_0^1 \ln\bigl(w-\widehat f(t)\bigr)\,dt +
\frac{\ln(w-\widehat f(1))-\ln(w-\widehat f(0))}{2}
\\
&&\qquad\quad{}+ T(w,f,N) \nonumber\\
&&\qquad= N \F(w;f) + \frac{\ln(w-\widehat f(1))-\ln(w-\widehat f(0))}{2} + T(w,f,N).
\nonumber
\end{eqnarray}
Under the conditions of Proposition~\ref
{proposition_convergence_mildest}, the function $\widehat f(t)$ is
piecewise-continuous, and the remainder $T(w,f,N)$ can be bounded via
%
\begin{eqnarray}
\bigl|T(w,f,N)\bigr|&\le& N \sum_{j=1}^N
\sup_{{(j-1)}/N\le t,s \le j/N} \frac
{\llvert \ln(w-\widehat f(t))-\ln(w-\widehat
f(s))\rrvert }{N}
\nonumber
\\[-8pt]
\\[-8pt]
\nonumber
&\le& o(N) \Bigl(1+\sup_{a\in A} \bigl|\ln(w-a)\bigr|\Bigr).
\end{eqnarray}
On the other hand, the right-hand side of \eqref{eq_x31} is bounded
from above by
$o(N) \sup_{a\in A} \llvert \frac{1}{w-a }\rrvert $. Combining these two
bounds we arrive at the desired estimate for
$\mathcal Q(w;\lambda(N),f)$.
\end{pf*}
\begin{pf*}{Proof of Lemma~\ref{Lemma_bound_GUE}}
We proceed in the same way as in the proof of Lemma~\ref
{Lemma_bound_logarithmic}.\vadjust{\goodbreak} This time, the right-hand side of
\eqref{eq_x31} is bounded from above by
$o(\sqrt{N}) \sup_{a\in A} \llvert \frac{1}{w-a }\rrvert $. We also have
%
\begin{eqnarray}
\bigl|T(w,f,N)\bigr|&\le& N \sum_{j=1}^N
\sup_{{(j-1)}/N\le t,s \le j/N} \frac
{\llvert \ln(w-\widehat f(t))-\ln(w-\widehat
f(s))\rrvert }{N}
\nonumber
\\[-8pt]
\\[-8pt]
\nonumber
&\le& O(1)\sup_{|t-s| \le1/N} \biggl\llvert \ln \biggl(
\frac{w-\widehat
f(t)}{w-\widehat f(s)} \biggr)\biggr\rrvert + \sup_{0\le t\le1}\biggl
\llvert \frac{\widehat f'(t)}{w-\widehat f(t)}\biggr\rrvert ,
\nonumber
\end{eqnarray}
where the last $\sup$ is taken only over those points where $\widehat
f$ is differentiable and the term with prefactor
$O(1)$ arises because of the possible discontinuities of $\widehat f$.
\end{pf*}
\begin{remark*} Note that the restriction of Propositions \ref
{proposition_convergence_mildest} and
\ref{Prop_convergence_GUE_case} that $f(t)$ should have finitely many
points of discontinuity is used only in the
proofs of the above two lemmas. It is very plausible that this
restriction can be removed if one uses more delicate
estimates in these proofs.
\end{remark*}

\subsection{Values at complex points}
\label{subsection_complex_points}

The propositions of the previous section deal with $S_\lambda(e^y;N,1)$
when $y$ is \emph{real}. In this
section we show that under mild assumptions the results extend to
complex $y$s.

In the notation of the previous section, suppose that we are given a
weakly-decreasing
nonnegative function $f(t)$, the complex function $\F(w;f)$ is defined through
\eqref{eq_definition_F}, $y$ is an arbitrary complex number and $w_0$
is a critical point of
$yw-\F(w;f)$, that is, a solution of equation \eqref
{eq_critical_point_equation}.

We call a simple piecewise-smooth contour $\gamma(s)$ in $\mathbb C$ a
\emph{steepest descent
contour} for the above data if the following conditions are satisfied:
\begin{longlist}[(1)]
\item[(1)] $\gamma(0)=w_0$;
\item[(2)] the vector $ (\F''(w_0;f) )^{-1/2}$ is tangent to
$\gamma$ at point $0$;
\item[(3)] $\operatorname{Re}(y\gamma(s)-\F(\gamma(s);f))$ has a global maximum at $s=0$;
\item[(4)] the following integral is finite:
\[
\int_{-\infty}^{\infty} \exp \bigl(\operatorname{Re}\bigl(y\gamma(s)-\F\bigl(
\gamma (s);f\bigr) \bigr)\bigr)\bigl|\gamma'(t)\bigr| \,dt <\infty.
\]
\end{longlist}

\begin{remark*} Often the steepest descent contour can be found as a
level line $\operatorname{Im}
(yw-\F(w;f))=\operatorname{Im}(yw_0-\F(w_0;f))$.
\end{remark*}

\begin{example}\label{ex1}
Suppose that $f(t)=0$. Then
\[
\F(w;f)=\int_{0}^1\ln(w-1+t)\,dt =w\ln(w)-(w-1)
\ln(w-1)-1
\]
and
%
\begin{equation}
\label{eq_x23} \F'(w;0)=\ln(w)-\ln(w-1)=-\ln(1-1/w).
\end{equation}
And for any $y$ such that $e^{y}\ne1$, the critical point is
\[
w_0=w_0(y)= \frac{1}{1-e^{-y}}.
\]

Let us assume that $e^{-y}$ is not a negative real number. This implies
that $w_0$ does not belong
to the segment $[0,1]$.

Figure~\ref{Figure_contours} sketches the level lines $\operatorname{Re}(yw-\F
(w;0))=\operatorname{Re}(yw_0-\F(w_0;0))$ for one
particular value of $y$. Let es explain the qualitative features of
these level lines.

Taylor expanding $yw-\F(w;0)$ near $w_0$ we observe that there are 4
level lines going out of
$w_0$. Note that the level lines cannot cross. Indeed, any
intersection of the level lines is a
critical point of $yw-\F(w;0)$, but the only critical point is at
$w_0$. When $|w|\gg1$, we have
$\operatorname{Re}(yw-\F(w;0))\approx \operatorname{Re}(yw)-\ln|w|$. Therefore level lines intersect
a circle of big radius $R\gg
1$ in 2 points, and the level lines' picture should have two infinite
branches which are close to the
rays of the line $\operatorname{Re}(yw)=\operatorname{const}$ and one loop. We claim that this loop
should enclose some points of
the segment $[0,1]$. Indeed, due to the maximum principle a
nonconstant harmonic function cannot
have closed level line; on the other hand, the only points where
$\operatorname{Re}(yw-\F(w;0))$ is not harmonic
lie in the segment $[0,1]$.

\begin{figure}

\includegraphics{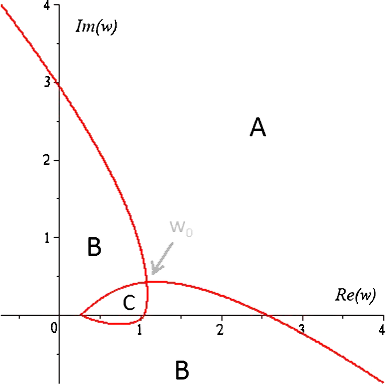}

\caption{Sketch of the level lines $\operatorname{Re}(yw-\F(w;0))=\operatorname{Re}(yw_0-\F(w_0;0))$
for $y=1-\ii$.}
\label{Figure_contours}
\end{figure}

Now the plane is divided into three regions $A,B$ and $C$ as shown in
Figure~\ref{Figure_contours}.
$\operatorname{Re}(yw-\F(w;0))>\operatorname{Re}(yw_0-\F(w_0;0))$ in $A$, $C$, and $\operatorname{Re}(yw-\F
(w;0))<\operatorname{Re}(yw_0-\F(w_0;0))$ in $B$. One
way to see this fact is by analyzing $\operatorname{Re}(yw-\F(w;0))$ for very large $|w|$.

There are two smooth curves
$\operatorname{Im}(yw-\F(w;0))=\operatorname{Im}(yw_0-\F(w_0;0))$ passing
through $w_0$. Taylor
expanding $yw-\F(w;0)$ near $w_0$ we observe that one of them has a
tangent vector parallel to
$\sqrt{\F''(w_0;0)}$, and another one has a tangent vector parallel to
$\ii\sqrt{\F''(w_0;0)}$. We
conclude that the former one lies inside the region $B$. In the
neighborhood of $w_0$ this curve is
our steepest descent contour. The only property which still might not
hold is property number
$4$. But in this case, we can modify the contour outside a small
neighborhood of $w_0$, so that
$\operatorname{Re}(yw-\F(w;0))$ rapidly decays along it. This is always possible
because for $|w|\gg1$, we have
$\operatorname{Re}(yw-\F(w;0))\approx \operatorname{Re}(yw)-\ln|w|$.
\end{example}
\begin{example}\label{ex2}
More generally let $f(t)=\alpha(1-t)$, then
\begin{eqnarray*}
\F\bigl(w;\alpha(1-t)\bigr)&=&\int_{0}^1\ln
\bigl(w+(\alpha+1) (t-1)\bigr)\,dt \\
&=&\frac{w\ln(w)-(w-(\alpha+1))\ln
(w-(\alpha+1))}{\alpha+1}-1
\end{eqnarray*}
and
\[
\F'\bigl(w;\alpha(1-t)\bigr)=\frac{\ln(1-(\alpha+1)/w)}{\alpha+1}.
\]
For any $y$ such that $e^{y}\ne1$, the critical point is
\[
w_0=w_0(y)= (\alpha+1)/\bigl(1-e^{-y(\alpha+1)}
\bigr).
\]
Note that if we set $w=u(\alpha+1)$, then
\[
\F\bigl(w;\alpha(1-t)\bigr)=u\ln(u)-(u-1)\ln(u-1)+\ln(\alpha+1)-1,
\]
which is a constant plus $\F(u;0)$ from Example~\ref{ex1}. Therefore, the
linear transformation of the
steepest descent contour of Example~\ref{ex2} gives a steepest descent contour
for Example~\ref{ex2}.
\end{example}

\begin{proposition}
\label{proposition_convergence_extended}
Suppose that $f(t)$, $y$ and $w_0$ are such that there exists a
steepest descent contour $\gamma$,
and moreover, the contour of integration in \eqref{eq_integral_basic}
can be deformed to
$\gamma$ without changing the value of the integral. Then Propositions
\ref{proposition_convergence_mildest} and \ref
{proposition_convergence_strongest} hold for this
$f(t)$, $y$ and $w_0$.
\end{proposition}
\begin{pf}
The proof of Propositions \ref{proposition_convergence_mildest} and
\ref{proposition_convergence_strongest} remains almost the same. The
only changes are in formula
\eqref{eq_x3} and subsequent estimates of errors. Note that condition
$4$ in the definition of steepest
descent contour guarantees that the integral over $\gamma$ outside
arbitrary neighborhood of
$w_0$ is still negligible as $N\to\infty$.

Observe that the integration in
\eqref{eq_x3} now goes not over the segment $[w_0-i\varepsilon
,w_0+i\varepsilon]$ but over the
neighborhood of $w_0$ on the curve $\gamma_0$. This means that in the
relative error calculation, a new
term appears, which is a difference of the integral
\[
\int e^{-s^2/2} \,ds
\]
over the interval $[-\sqrt{N}\varepsilon|u|, \sqrt{N}\varepsilon|u|]$
of real line and over the part of
rescaled curve $\frac{\gamma(t)-\gamma(0)}{\sqrt{N} u}$ inside circle
of radius
$\sqrt{N}\varepsilon
|u|$ around the origin. The difference of the two integrals equals to the
integral of $\exp(-s^2/2)$ over the lines connecting their endpoints.
But since
$1/u=-(\F''(w_0;f))^{-1/2}$ is tangent to $\gamma$ at $0$, it follows
that for small $\varepsilon$ the
error is the integral of $\exp(-s^2/2)$ over segment joining $\sqrt
{N}\varepsilon|u|$ and $\sqrt{N}\varepsilon|u| +
Q_1$ plus the integral of $\exp(-s^2/2)$ joining $-\sqrt
{N}\varepsilon
|u|$ and $-\sqrt{N}\varepsilon|u|
+ Q_2$ with $|Q_1|< (\sqrt{N}\varepsilon|u|)/100$ and similarly for
$Q_2$. Clearly, these
integrals exponentially decay as $N\to\infty$, and we are done.
\end{pf}

It turns out that in the context of Proposition~\ref
{Prop_convergence_GUE_case} the required
contour always exists.
%
\begin{proposition}
\label{Prop_convergence_GUE_extended}
Proposition~\ref{Prop_convergence_GUE_case} is valid for any $h\in
\mathbb C$.
\end{proposition}
\begin{pf}
Recall that in the context of Proposition~\ref
{Prop_convergence_GUE_case} $y=h/\sqrt{N}$ and goes
to~$0$ as $N\to\infty$, while $w_0\approx1/y$ goes to infinity. In
what follows without loss of generality we assume that
$h$ is not an element of $\mathbb R_{\le0}$ and choose in all
arguments the principal branch of logarithms with cut along negative
real axis.
(In order to work with $h\in\mathbb R_{<0}$, we should choose other branches.)

Let us construct the right steepest descent contour passing through
the point $w_0$.
Choose positive number $r$ such that $r>|\widehat f(t)|$ for all $0\le
t\le1$.
Set $\Psi$ to be the minimal strip (which is a region between two
parallel lines)
in complex plane parallel to the vector $\ii/h$ and containing the
disk of radius $r$ around the origin.

Since $w_0$ is a saddle point of $yw-\F(w;f)$,
in the neighborhood of $w_0$ there are two smooth curves $\operatorname{Im}(yw-\F
(w;f))=\operatorname{Im}(yw_0-\F(w_0;f))$
intersecting at $w_0$. Along one of them $\operatorname{Re}(yw-\F(w;f))$ has maximum
at~$w_0$, along another one
it has minimum; we need the former one. Define the contour $\gamma$ to
be the smooth curve
$\operatorname{Im}(yw-\F(w;f))=\operatorname{Im}(yw_0\F(w_0;f))$
until it leaves $\Psi$ and the
curve (straight line) $\operatorname{Re}(yw)=\operatorname{const}$ outside $\Psi$.

Let us prove that $\operatorname{Re}(yw-\F(w;f))$ has no local extremum on $\gamma$
except for~$w_0$, which would imply that $w_0$ is its
global maximum on $\gamma$. First note that outside $\Psi$ we have
\[
\operatorname{Re}\bigl(yw-\F(w;f)\bigr) =\operatorname{Re}(yw)-\int_{0}^1
\ln\bigl|w-\widehat f(t)\bigr| \,dt,
\]
with the first term here being a constant, while the second being
monotone along the contour.
Therefore, outside $\Psi$ we cannot have local extremum. Next,
straightforward computation shows that
if $N$ is large enough, then one can always choose two independent of
$N$ constants $1/2>G_1>0$ and $G_2>0$ such that
$\operatorname{Re}(yw-\F(w;f))>\operatorname{Re}(yw_0-\F(w_0;f))$ for $w$ in $\Psi$ satisfying
$|w|=G_1|w_0|$ or $|w|=G_2|w_0|$. It
follows, that if $\operatorname{Re}(yw-\F(w;f))$ had a local extremum, then such
extremum would exist at some
point $w_1\in\Psi$ satisfying $G_1|w_0|<|w_1|<G_2|w_0|$. But since
$\operatorname{Im}(yw-\F(w;f))$ is constant on
the contour inside $\Psi$, we conclude that $w_1$ is also a critical
point of $yw-\F(w;f)$.
However, there are no critical points other than $w_0$ in this region.

Now we use the contour $\gamma$ and repeat the argument of Proposition~\ref{Prop_convergence_GUE_case} using it. Note that the deformation of
the original contour of
\eqref{eq_integral_basic} into $\gamma$ does not change the value of
the integral. The only part
of proof of Proposition~\ref{Prop_convergence_GUE_case} which we
should modify is the estimate for
the relative error in \eqref{eq_final_approximation}. Here we closely
follow the argument of
Proposition~\ref{proposition_convergence_extended}. The only change is
that the bound on $Q_1$
and $Q_2$ is now based on the following observation: The straight line
defined by $\operatorname{Re}(yw)=
\operatorname{Re} (y w_0)$ (which is the main part of the contour $\gamma$) is
parallel to the vector $\ii/y$. On the other hand,
\[
\sqrt{F''(w_0)} = \ii/y\bigl (1+O(1/
\sqrt{N}) \bigr)\approx\ii/y.
\]
\upqed\end{pf}
\begin{remark*} In the proof of Proposition~\ref
{Prop_convergence_GUE_extended} we have shown, in
particular, that the steepest descent contour exists, and thus
asymptotic theorem is valid for all
complex $y$, which are close enough to $1$. This is somehow similar to
the results of Guionnet and
Ma\"\i da; cf. \cite{GM}, Theorem~1.4.
\end{remark*}

\section{Statistical mechanics applications}\label{s:stat_mech}

\subsection{GUE in random tilings models}
\label{Section_GUE}

Consider a tiling of a domain drawn on the regular triangular lattice
of the kind shown at Figure~\ref{Fig_polyg_domain} with rhombi of 3
types which are usually called
\emph{lozenges}. The
configuration of the domain is encoded by the number $N$ which is its
width and $N$ integers
$\mu_1>\mu_2>\cdots>\mu_N$ which are the positions of \emph{horizontal
lozenges} sticking out of the
right boundary. If we write $\mu_i=\lambda_i+N-i$, then $\lambda$ is a
signature of size $N$; see
left panel of Figure~\ref{Fig_polyg_domain}. Due to combinatorial
constraints the tilings of such
domain are in correspondence with tilings of a certain polygonal
domain, as shown on the right
panel of Figure~\ref{Fig_polyg_domain}.

Let $\Omega_\lambda$ denote the domain encoded by $\lambda\in\GT_N$,
and define $\Upsilon_\lambda$
to be a \emph{uniformly random} lozenge tiling of $\Omega_\lambda$. We
are interested in the
asymptotic properties of $\Upsilon_\lambda$ as $N\to\infty$ and
$\lambda
$ changes in a certain
regular way.

Given $\Upsilon_\lambda$ let $\nu_1>\nu_2>\cdots>\nu_k$ be
positions of
the horizontal lozenges at the $k$th vertical
line from the left. (Horizontal lozenges are shown in blue in the left
panel of Figure~\ref{Fig_polyg_domain}.) We again set $\nu_i=\kappa
_i+k-i$ and denote
the resulting random
signature $\kappa$ of size $k$ by $\Upsilon_\lambda^k$.

Recall that the Gaussian unitary ensemble is a probability measure on
the set of $k\times k$
Hermitian random matrices with density proportional to $\exp
(-\operatorname{Trace}(X^2)/ 2)$. Let $\GUE_k$ denote the
distribution of $k$ (ordered) eigenvalues of such random matrices.

In this section we prove the following theorem.

\begin{theorem}
\label{Theorem_GUE}
Let $\lambda(N)\in\mathbb{GT}_N$, $N=1,2,\ldots$ be a sequence of
signatures. Suppose that there
exists a nonconstant piecewise-differentiable weakly decreasing
function $f(t)$ such that
\[
\sum_{i=1}^N\biggl\llvert
\frac{\lambda_i(N)}N - f(i/N)\biggr\rrvert = o(\sqrt{N})
\]
as $N\to\infty$ and also $\sup_{i,N} |\lambda_i(N)/N|<\infty$.
Then for
every $k$ as $N\to\infty$, we have
\[
\frac{\Upsilon_{\lambda(N)}^k-N E(f)}{\sqrt{NS(f)}} \to\GUE_k
\]
in the sense of weak convergence, where
\[
E(f)=\int_{0}^1 f(t) \,dt,\qquad S(f)= \int
_0^1 f(t)^2 \,dt -E(f)^2
+ \int_0^1 f(t) (1-2t) \,dt.
\]
\end{theorem}
\begin{remark*}
For any nonconstant weakly decreasing $f(t)$,
we have $S(f)>0$.
\end{remark*}

\begin{corollary}
\label{Cor_minors}
Under the same assumptions as in Theorem~\ref{Theorem_GUE} the
(rescaled) joint distribution of
$k(k+1)/2$ horizontal lozenges on the left $k$ lines weakly converges
to the joint distribution of
the eigenvalues of the $k$ top-left corners of a $k\times k$ matrix
from GUE.
\end{corollary}
\begin{pf}
Indeed, conditionally on $\Upsilon_{\lambda}^k$ the distribution of
the remaining $k(k-1)/2$ lozenges is
uniform subject to interlacing conditions and the same property holds
for the eigenvalues of the
corners of GUE random matrix; see \cite{Bar} for more details.
\end{pf}

Let us start the proof of Theorem~\ref{Theorem_GUE}.

\begin{proposition}
\label{prop_distribution_of_lozenges}
The distribution of $\Upsilon_\lambda^k$ is given by
\[
\operatorname{Prob}\bigl\{\Upsilon_\lambda^k=\eta\bigr\} =
\frac{s_\eta(1^k)
s_{\lambda/\eta
}(1^{N-k})}{s_\lambda(1^N)},
\]
where $s_{\lambda/\eta}$ is the skew Schur polynomial.
\end{proposition}
\begin{pf}
Let $\kappa\in\GT_M$ and $\mu\in\GT_{M-1}$. We say that $\kappa$ and
$\mu$ interlace and write
$\mu\prec\kappa$, if
\[
\kappa_1\ge\mu_2\ge\kappa_2\ge\cdots\ge
\mu_{M-1}\ge\kappa_M.
\]
We also agree that $\GT_0$ consists of a single point, \emph{empty}
signature $\varnothing$
and $\varnothing\prec\kappa$ for all $\kappa\in\GT_1$.

For $\kappa\in\GT_K$ and $\mu\in\GT_L$ with $K>L$, let $\operatorname{Dim}(\mu
,\kappa)$ denote the number
of sequences $\zeta^L\prec\zeta^{L+1}\prec\cdots\prec\zeta^K$
such that
$\zeta^i\in\GT_i$,
$\zeta^L=\kappa$ and $\zeta^K=\mu$. Note that through the
identification of each $\zeta^i$ with configuration
of horizontal lozenges on a vertical line, each such sequence
corresponds to a lozenge tiling of a certain domain
encoded by $\kappa$ and $\mu$, so
that, in particular the tiling on the left panel of Figure~\ref
{Fig_polyg_domain} corresponds to
the sequence
\[
\varnothing\prec(2)\prec(3,0)\prec(3,1,0)\prec(3, 3,0,0)\prec(4,3,3,0,0).
\]
It follows that
\[
\operatorname{Prob}\bigl\{\Upsilon_\lambda^k=\eta\bigr\} =
\frac{\operatorname{Dim}(\varnothing
,\eta)
\operatorname{Dim}(\eta,\lambda)}{\operatorname{Dim}(\varnothing,\lambda)}.
\]
On the other hand the \emph{combinatorial formula} for (skew) Schur
polynomials (see, e.g., \cite{M}, Chapter I, Section~5)
yields that for $\kappa\in\GT_K$ and $\mu\in\GT_L$ with $K>L$, we have
\[
\operatorname{Dim}(\mu,\kappa) = s_{\kappa/\mu}\bigl(1^{K-L}\bigr), \qquad
\operatorname{Dim}(\varnothing,\mu) = s_{\mu}\bigl(1^{L}\bigr).
\]
\upqed\end{pf}

Introduce the multivariate normalized Bessel function $B_k(x;y)$,
$x=\break (x_1,\ldots, x_k)$,
$y=(y_1,\ldots,y_k)$ through
\[
B_k(x;y) = \frac{\det_{i,j=1,\ldots, k}  (\exp(x_i y_j)
)}{\prod_{i<j} (x_i-x_j) \prod_{i<j}
(y_i-y_j)} \prod_{i<j}
(j-i).
\]

The functions $B_k(x;y)$ appear naturally as a result of computation of
Harish-Chandra--Itzykson--Zuber matrix integral \eqref{eq_HC_intro}.
Their relation to Schur
polynomials is explained in the following statement.

\begin{proposition}
\label{prop_Bessel_vs_Schur}
For $\lambda=(\lambda_1,\lambda_2,\ldots,\lambda_k)\in\GT_k$, we have
\begin{eqnarray*}
&&\frac{s_\lambda(e^{x_1},\ldots,e^{x_k})}{s_\lambda(1^k)} \prod_{i<j} \frac{e^{x_i}-e^{x_j}}{x_i-x_j}
\\
&&\qquad= B_k(x_1,\ldots,x_k; \lambda_1+k-1,
\lambda_2+k-2,\ldots,\lambda_k).
\end{eqnarray*}
\end{proposition}
\begin{pf}
The proof immediately follows from the definition of Schur polynomials
and the evaluation of
$s_\lambda(1^k)$ given in \eqref{eq_Weyl_dim}.
\end{pf}

We study $\Upsilon_\lambda^k$ for $\lambda\in\GT_N$ through its
moment generating functions $\mathbb E B_k(x;
\Upsilon_\lambda^k+\delta_k)$, where $x=(x_1,\ldots,x_k)$, $\delta
_k=(k-1,k-2,\ldots,0)$ as above,
and $\mathbb E$ stands for the expectation. Note that for $k=1$, the
function $\mathbb E B_k(x;
\Upsilon_\lambda^k+\delta_k)$ is nothing but usual one-dimensional
moment generating function\break 
$\mathbb E \exp(x \Upsilon_\lambda^1)$.

\begin{proposition} We have
\[
\mathbb E B_k\bigl(x; \Upsilon_\lambda^k+
\delta_k\bigr) = \frac{s_\lambda(e^{x_1},\ldots,e^{x_k},1^{N-k})}{s_\lambda(1^N)} \prod
_{1\le i<j\le
k}\frac{e^{x_i}-e^{x_j}}{x_i-x_j}.
\]
\label{Prop_tiling_gen_function}
\end{proposition}
\begin{pf} Let $Z=(z_1,\ldots,z_m)$ and $Y=(y_1,\ldots,y_n)$, and let
$\mu\in\GT_{m+n}$, then
(see, e.g., \cite{M}, Chapter I, Section~5)
\[
\sum_{\kappa\in\GT_{m}} s_\kappa(Z)s_{\mu/\kappa}(Y)=s_\mu(Z,Y).
\]
Therefore, Propositions \ref{prop_distribution_of_lozenges} and \ref
{prop_Bessel_vs_Schur} yield
\begin{eqnarray*}
\bigl( \mathbb E B_k\bigl(x; \Upsilon_\lambda^k+
\delta_k\bigr) \bigr) \prod_{i<j}
\frac{x_i-x_j}{e^{x_i}-e^{x_j}}&=&\sum_{\eta\in\GT_k} \frac{s_\eta(e^{x_1},\ldots,e^{x_k})}{s_\eta(1^k)}
\cdot\frac
{s_\eta(1^k)
s_{\lambda/\eta}(1^{N-k})}{s_\lambda(1^N)}
\\
&=&\frac{ \sum_{\eta
\in\GT_k}
s_\eta(e^{x_1},\ldots,e^{x_k}) s_{\lambda/\eta}(1^{N-k})}{s_\lambda
(1^N)}\\
&=&\frac{
s_\lambda(e^{x_1},\ldots,e^{x_k},1^{N-k})}{s_\lambda(1^N)}.
\end{eqnarray*}
\upqed\end{pf}

The counterpart of Proposition~\ref{Prop_tiling_gen_function} for
$\GUE
_k$ distribution is the
following.
%
\begin{proposition} We have
\label{prop_GUE_gen_function}
%
\begin{equation}
\label{eq_GUE_moment_gen_function} \mathbb E B_k(x;\GUE_k) = \exp \bigl(
\tfrac{1}{2}\bigl(x_1^2+\cdots +x_k^2
\bigr) \bigr).
\end{equation}
\end{proposition}
\begin{pf}
Let $X$ be a (fixed) diagonal $k\times k$ matrix with eigenvalues
$x_1,\ldots,x_k$, and let $A$ be random $k\times k$ Hermitian
matrix from $\GUE$. Let us compute
%
\begin{equation}
\label{eq_expectation_of_trace} \mathbb E \exp \bigl(\operatorname{Trace}(XA) \bigr).
\end{equation}
From one hand, standard integral evaluation shows that \eqref
{eq_expectation_of_trace} is equal
to the right-hand side of \eqref{eq_GUE_moment_gen_function}. On the
other hand, we can rewrite
\eqref{eq_expectation_of_trace} as
%
\begin{equation}
\label{eq_Harish_Chandra} \int_{y_1\ge y_2\ge\cdots\ge y_k} P_{\GUE_k}(dy) \int
_{u\in U(k)} P_{\mathrm{Haar}}(du) \exp \bigl(\operatorname{Trace}
\bigl(YuXu^{-1}\bigr) \bigr),
\end{equation}
where $P_{\GUE_k}$ is probability distribution of $\GUE_k$, $P_{\mathrm{Haar}}$
is normalized Haar measure
on the unitary group $U(k)$ and $Y$ is Hermitian matrix (e.g.,
diagonal) with eigenvalues
$y_1,\ldots,y_k$. The evaluation of the integral over unitary group in
\eqref{eq_Harish_Chandra}
is well-known (see \cite{HC1,HC2,IZ,OV}),
and the
answer is precisely $B_k(y_1,\ldots,y_k; x_1,\ldots,x_k)$. Thus
\eqref
{eq_Harish_Chandra}
transforms into the left-hand side of \eqref{eq_GUE_moment_gen_function}.
\end{pf}

In what follows we need the following technical proposition.
%
\begin{proposition}
\label{Prop_convergence_of_gen_func}
Let $\phi^N=(\phi^N_1\ge\phi^N_2\ge\cdots\ge\phi^N_k)$,
$N=1,2,\ldots
$ be a sequence of
$k$-dimensional random variables.\vadjust{\goodbreak} Suppose that there exists a random
variable $\phi^\infty$
such that for every $x=(x_1,\ldots,x_k)$ in a neighborhood of
$(0,\ldots,0)$, we have
\[
\lim_{N\to\infty} \mathbb E B_k\bigl(x;
\phi^N\bigr) = \mathbb E B_k\bigl(x;\phi ^{\infty}
\bigr).
\]
Then $\phi^N\to\phi^\infty$ in the sense of weak convergence of random
variables.
\end{proposition}
\begin{pf}
For $k=1$ this is a classical statement; see, for example,
\cite{Bi}, Section~30. For general
$k$ this statement is, perhaps, less known, but it can be proven by
the same standard techniques as
for $k=1$.
\end{pf}

Next, note that the definition implies the following property for the
moment generating function
of $k$-dimensional random variable $\phi$:
\[
\mathbb E B_k(x_1,\ldots,x_k;a\phi+b) = \exp
\bigl(b(x_1+\cdots+x_k)\bigr) \mathbb E B_k(a
x_1,\ldots,a x_k; \phi).
\]
Also observe that for any nonconstant weakly decreasing $f(t)$, we
have $S(f)>0$. The following
statement, together with Proposition~\ref{Prop_tiling_gen_function},
gives the moment generating
function for the shifted and normalized $\Upsilon^k_{\lambda(N)}$ as
$N\to\infty$.

\begin{proposition}\label{GUE_multivar_asymptotics}
In the assumptions of Theorem~\ref{Theorem_GUE} for any $k$ reals
$h_1,\ldots,h_k$, we have
\begin{eqnarray*}&&
\lim_{N\to\infty}\frac{s_{\lambda(N)} ( e^{{h_1}/{\sqrt
{N S(f)
}}},\ldots, e^{{h_k}/{\sqrt{N
S(f)}
}},1^{N-k} )}{s_{\lambda(N)}(1^N)} \\
&&\quad{}\times\exp \biggl(-\sqrt{N}
\frac
{E(f)}{\sqrt{S(f)}}(h_1+\cdots+h_k) \biggr)
\\
&&\qquad= \exp \biggl(\frac{1}{2}\bigl(h_1^2+
\cdots+h_k^2\bigr) \biggr).
\end{eqnarray*}
\end{proposition}
\begin{pf}
For $k=1$ this is precisely the statement of Proposition~\ref
{Prop_convergence_GUE_case}. For
general $k$ we combine Proposition~\ref{Prop_convergence_GUE_extended}
and Corollary~\ref{Corollary_multiplicativity_for_GUE}.
\end{pf}
\begin{pf*}{Proof of Theorem~\ref{Theorem_GUE}}
Propositions \ref{GUE_multivar_asymptotics} and \ref
{Prop_tiling_gen_function}, and the
observation that $(e^{x_i}-e^{x_j})/(x_i-x_j)$ tends to $1$ when
$x_i,x_j\to0$ show that as
$N\to\infty$ the moment generating function for the shifted and normalized
$\Upsilon^k_{\lambda(N)}$ converges to the corresponding moment
generating function for the
$\GUE_k$ as given in Proposition~\ref{prop_GUE_gen_function}. Now Proposition~\ref{Prop_convergence_of_gen_func} implies the weak convergence
\[
\bigl(\Upsilon^k_{\lambda(N)} - NE(f)\bigr)/\sqrt{NS(f)} \to
\GUE_k,
\]
and
Theorem~\ref{Theorem_GUE} then follows.
\end{pf*}

\subsection{Asymptotics of the six vertex model}
\label{Section_ASM}

Recall that an \emph{alternating sign matrix} of size $N$ is a
$N\times
N$ matrix filled with
$\zero$s
$\one$s and $\mone$s in such a way that the sum along every row and
column is $\one$, and moreover, along
each row and each column $\one$s and $\mone$s are alternating, possibly
separated by an arbitrary
number of~$\zero$s. Alternating sign matrices are in bijection with
configurations of the
six-vertex (``square ice'') model with domain wall boundary conditions.
The configurations of the
$6$-vertex model are assignments of one of 6 types of $\mathrm{H_2O}$
molecules
shown in Figure~\ref{Figure_six} to the vertices of $N\times N$ square
grid in such a
way that the $\mathrm{O}$ atoms are
at the vertices of the grid. To each $\mathrm{O}$ atom there are
two $\mathrm{H}$ atoms
attached, so that they are
at angles $90^\circ$ or $180^\circ$ to each other, along the grid
lines, and between any two
adjacent $\mathrm{O}$ atoms there is exactly one $\mathrm{H}$. We also impose the
so-called \textit{domain wall boundary
conditions} as shown in Figure~\ref{Fig_ASM} in the \hyperref
[sec1]{Introduction}. In
order to get an ASM we replace
the vertex of each type with $\zero$, $\one$ or $\mone$, as shown in
Figure~\ref{Figure_six}; see, for example, \cite{Ku} and references
therein for more details. Figure~\ref{Fig_ASM} gives one example of
ASM and corresponding configuration of the $6$-vertex model.

\begin{figure}

\includegraphics{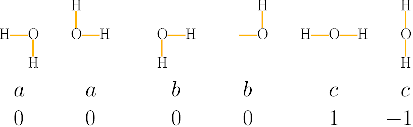}

\caption{Types of vertices in the
six vertex model divided by groups and their correspondence to numbers
in ASM.}\label{Figure_six}\vspace*{-4pt}
\end{figure}

Let $\gimel_N$ denote the set of all alternating sign matrices of size
$N$ or, equivalently, all
configurations of six-vertex model with domain wall boundary condition.
Equip $\gimel_N$ with
\emph{uniform} probability measure and let $\omega_N$ be a random
element of $\gimel_N$. We are
going study the asymptotic properties of $\omega_N$ as $N\to\infty$.

For $\vartheta\in\gimel_N$ let $a_i(\vartheta)$, $b_i(\vartheta)$,
$c_i(\vartheta)$ denote the
number of vertices in horizontal line $i$ of types $a$, $b$ and $c$,
respectively (the types are
shown in Figure~\ref{Figure_six}).\vspace*{1pt} Likewise, let $\widehat
a_j(\vartheta
)$, $\widehat
b_j(\vartheta)$ and $\widehat c_j(\vartheta)$ be the same quantities in
vertical line $j$. Also
let $a^{ij}(\vartheta)$, $b^{ij}(\vartheta)$ and $c^{ij}(\vartheta)$ be
0--1 functions equal to
the number of vertices of types $a$, $b$ and $c$, respectively, at the
intersection of vertical
line $j$ and horizontal line $i$. To simplify the notation we view
$a_i$, $b_i$ and $c_i$ as
random variables and omit their dependence on $\vartheta$.

\begin{theorem}
\label{Theorem_ASM}
For any fixed $j$ the random variable ${(a_j-N/2)}/{\sqrt{N}}$
weakly converges to the normal
random variable $N(0, \sqrt{3/8})$. The same is true for $a_{N-j}$,
$\widehat a_j$ and $\widehat
a_{N-j}$. Moreover, the joint distribution of any collection of such
variables converges to the
distribution of independent normal random variables $N(0, \sqrt{3/8})$.
\end{theorem}

Inspecting the bijection between ASMs and the configurations of the
six-vertex model one readily
sees that Theorem~\ref{Theorem_ASM} implies Theorem~\ref
{Theorem_ASM_Intro}. The rest of this
section is devoted to the proof of Theorem~\ref{Theorem_ASM}.\vadjust{\goodbreak}

The 6 types of vertices in a six-vertex model are divided into 3
groups, as shown in Figure~\ref{Figure_six}. Define a weight depending
on the position $(i,j)$
($i$ is the vertical
coordinate) of the vertex and its type as follows:
\[
a: q^{-1} u_i^2-q v_j^2,\qquad
b: q^{-1} v_j^2-q u_i^2,\qquad
c:\bigl(q^{-1}-q\bigr) u_i v_j,
\]
where $v_1,\ldots,v_N$, $u_1,\ldots,u_N$ are parameters, and from now and
until the end of this section,
we set $q=\exp(\pi\ii/3)$. (Notice that this implies $q^{-1}+q=1;
q-q^{-1}=\ii\sqrt{3}$.)

Let the weight $W$ of a configuration be equal to the product of
weights of vertices. The
partition function of the model can be explicitly evaluated in terms of
Schur polynomials.
%
\begin{proposition} We have \label{Prop_partition_function_of_six_vertex}
\begin{eqnarray*}
&&\sum_{\vartheta\in\gimel_N} W(\vartheta)\\
&&\qquad =(-1)^{N(N-1)/2}
\bigl(q^{-1}-q\bigr)^N \prod_{i=1}^N
(v_i u_i)^{-1} s_{\lambda(N)}
\bigl(u_1^2,\ldots,u_N^2,
v_1^2,\ldots, v_N^2\bigr),
\end{eqnarray*}
where $\lambda(N)=(N-1,N-1,N-2,N-2,\ldots,1,1,0,0)\in\GT_{2N}$.
\end{proposition}
\begin{pf}
See \cite{Oka,St,FZ}.
\end{pf}

The following proposition is a straightforward corollary of Proposition~\ref{Prop_partition_function_of_six_vertex}.

\begin{proposition}
\label{prop_observable_6_vertex}
Fix any $n$ distinct vertical lines $i_1,\ldots,i_n$ and $m$ distinct
horizontal lines
$j_1,\ldots,j_m$ and any set of complex numbers $u_1,\ldots,u_n$,
$v_1,\ldots,v_m$. We have
%
\begin{eqnarray}
\label{eq_obs_6v_formula} &&\mathbb E_N \prod_{k=1}^n
\biggl[ \biggl(\frac{q^{-1} u_{k}^2-q
}{q^{-1}-q} \biggr)^{a_{i_k}} \biggl(
\frac{q^{-1} -q
u_{k}^2}{q^{-1}-q} \biggr)^{b_{i_k}} (u_{k} )^{c_{i_k}}
\biggr]
\nonumber
\\[-8pt]
\\[-8pt]
\nonumber
&&\qquad = \Biggl(\prod_{k=1}^n
u_{k}^{-1} \Biggr) \frac{s_{\lambda(N)}(u_{1},\ldots
,u_{n},1^{2N-n})}{s_{\lambda(N)}(1^{2N})},
\\
&&\mathbb E_N \prod_{\ell=1}^m
\biggl[ \biggl(\frac{q^{-1}-q v_{\ell}^2
}{q^{-1}-q} \biggr)^{\widehat a_{j_\ell}} \biggl(
\frac{q^{-1} v_{\ell}^2 -q
}{q^{-1}-q} \biggr)^{\widehat b_{j_\ell}} (v_{\ell} )^{\widehat c_{j_\ell}}
\biggr]
\nonumber
\\[-8pt]
\\[-8pt]
\nonumber
&&\qquad= \Biggl(\prod_{\ell=1}^n
v_{\ell}^{-1} \Biggr) \frac{s_{\lambda(N)}(v_{1},\ldots
,v_{m},1^{2N-m})}{s_{\lambda(N)}(1^{2N})}
\end{eqnarray}
and, more generally
%
\begin{eqnarray}
\label{eq_obs_6v_long}&& \mathbb E_N \Biggl( \prod
_{k=1}^n \biggl[ \biggl(\frac{q^{-1} u_{k}^2-q
}{q^{-1}-q}
\biggr)^{a_{i_k}} \biggl(\frac{q^{-1} -q
u_{k}^2}{q^{-1}-q} \biggr)^{b_{i_k}}
(u_{k} )^{c_{i_k}} \biggr]\nonumber
\\
&&\qquad{}\times \prod_{\ell=1}^m \biggl[ \biggl(
\frac{q^{-1}-q v_{\ell}^2
}{q^{-1}-q} \biggr)^{\widehat
a_{j_\ell}} \biggl(\frac{q^{-1} v_{\ell}^2 -q
}{q^{-1}-q}
\biggr)^{\widehat b_{j_\ell}} (v_{\ell} )^{\widehat c_{j_\ell}} \biggr]\nonumber
\\
&&\qquad{}\times \prod_{k=1}^n\prod
_{\ell=1}^m \biggl[ \biggl(\frac{(q^{-1} u_{k}^2-q v_\ell^2) (q^{-1}-q)} {(q^{-1} u_{k}^2-q
)(q^{-1}-q v_{\ell}^2)}
\biggr)^{a^{i_k,j_\ell}} \\
&&\hspace*{51pt}\qquad{}\times \biggl(\frac{(q^{-1} v_{\ell}^2 -q u_k^2) (q^{-1}-q)} {(q^{-1} -q
u_{k}^2)(q^{-1} v_{\ell}^2 -q)} \biggr)^{b^{i_k,j_\ell}} \biggr]
\Biggr)\nonumber
\\
&&\qquad{}= \Biggl(\prod_{\ell=1}^m
v_{\ell}^{-1} \prod_{k=1}^n
u_{k}^{-1} \Biggr) \frac{s_{\lambda(N)}(u_{1},\ldots,u_{n},v_{1},\ldots
,v_{m},1^{2N-n-m})}{s_{\lambda(N)}(1^{2N})},
\nonumber
\end{eqnarray}
where all the above expectations $\mathbb E_N$ are taken with respect
to the uniform measure on $\gimel_N$.
\end{proposition}

We want to study $N\to\infty$ limits of observables of Proposition~\ref
{prop_observable_6_vertex}.
Suppose that $n=1$, $m=0$. Then we have two parameters $u_1=u$ and
$i_1=i$. Suppose that as
$N\to\infty$, we have
%
\begin{equation}
\label{eq_u_y} u=u(N)=\exp(y/\sqrt{N}),
\end{equation}
and $i$ remains fixed. Then we can use Proposition~\ref{Prop_convergence_GUE_case} to understand the asymptotics of the
right-hand side of
\eqref{eq_obs_6v_formula}.

As for the left-hand side of \eqref{eq_obs_6v_formula}, note that $c_i$
is uniformly bounded, in
fact $c_i<2i$ because of the combinatorics of the model. Therefore, the
factors involving $c_i$ in
the observable become negligible as $N\to\infty$. Also note that
$a_i+b_i+c_i=N$. Therefore the
observable can be rewritten as
\[
\biggl(\frac{q^{-1}-q e^{2y/\sqrt{N}} }{q^{-1}-q} \biggr)^N \biggl(\frac
{q^{-1} e^{2y/\sqrt{N}} -q
}{q^{-1}-q e^{2y/\sqrt{N}}}
\biggr)^{a_i} G(y),
\]
with $G$ satisfying the estimate $|\ln G(y)|<Cy/\sqrt{N}$ with some
constant $C$ (independent of
all other parameters).

Now let $z$ be an auxiliary variable, and choose $y=y(z,N)$ such that
%
\begin{equation}
\label{eq_z_through_y} \exp(z/\sqrt{N})=\frac{q^{-1} e^{2y/\sqrt
{N}} -q
}{q^{-1}-q e^{2y/\sqrt{N}}}.
\end{equation}
Now the observable (as a function of $z$) turns into $ (\frac
{q^{-1}-q e^{2y/\sqrt{N}}
}{q^{-1}-q} )^N$ times $\exp(za_i/\sqrt{N})$. Therefore, the
expectation in
\eqref{eq_obs_6v_formula} is identified with the exponential moment
generating function for
$a_i/\sqrt{N}$.

In order to obtain the asymptotics we should better understand the
function $y(z,N)$. Rewrite
\eqref{eq_z_through_y} as
\[
e^{2y/\sqrt{N}} =\frac{ \exp(z/\sqrt{N}) q^{-1} +q
}{q^{-1} +q \exp(z/\sqrt{N})}=\frac{1+(\exp(z/\sqrt{N})-1)
({q^{-1}}/{(q^{-1}+q)})}{1+(\exp(z/\sqrt{N})-1)({q}/{(q^{-1}+q)})}.
\]
Recall that $q^{-1}+q=1$, and therefore
\begin{eqnarray*}
2y&=&\sqrt{N} \bigl(\ln\bigl(1+q^{-1}\bigl(\exp(z/\sqrt{N})-1\bigr)
\bigr)-\ln\bigl(1+q\bigl(\exp (z/\sqrt {N})-1\bigr)\bigr) \bigr)
\\
&=&-\bigl(q-q^{-1}\bigr)z-\frac{q-q^{-1}}{2} z^2/\sqrt{N}+
\frac{q^2-q^{-2}}{2} z^2/\sqrt{N} +O\bigl(z^3/N\bigr).
\end{eqnarray*}
Note that the last two terms cancel out, and we get
%
\begin{equation}
\label{eq_y_through_z} y=-z\ii\frac{\sqrt{3}}{2}+O\bigl(z^3/N\bigr).
\end{equation}
Now we compute
\begin{eqnarray*}
\biggl(\frac{q^{-1}-q e^{2y/\sqrt{N}} }{q^{-1}-q} \biggr)^N&=&\exp \biggl[N\ln \biggl(1-
\frac{q}{q^{-1}-q} \bigl(e^{-\ii\sqrt{3} z/\sqrt{N}+O(z^3 N^{-3/2})}-1\bigr) \biggr) \biggr]
\\
&= &\exp \bigl[- \sqrt{N} qz + q \ii\sqrt{3} z^2/2 -q^2
z^2/2 +o(1) \bigr] \\
&=& \exp \bigl[ -\sqrt{N} qz - z^2/2+
o(1) \bigr].
\end{eqnarray*}

Summing up, the observable of \eqref{eq_obs_6v_formula} is now
rewritten as
%
\begin{equation}
\label{eq_obs_6_vertex_final} \exp \biggl[-\sqrt{N} z \ii\frac
{\sqrt
{3}}{2} -
z^2/2+o(1) \biggr] \exp \biggl[ \frac{a_i-N/2}{\sqrt{N}} z \biggr].
\end{equation}

Now combining \eqref{eq_obs_6v_formula} with Propositions \ref
{Prop_convergence_GUE_case},
\ref{Prop_convergence_GUE_extended} [note that parameter $N$ in these
two propositions differs by
the factor $2$ from that of \eqref{eq_obs_6v_formula}], we conclude
that (for any complex $z$) the
expectation of \eqref{eq_obs_6_vertex_final} is asymptotically
\[
\exp \bigl[ 4\sqrt{N} y E(f) + 4 S(f) y^2 + o(1) \bigr],
\]
where $f$ is the function $\frac{1-x}{2}$. Using \eqref{eq_y_through_z}
and computing
\[
E(f)=1/4,\qquad S(f)=5/48
\]
we get
%
\begin{equation}
\label{eq_x24} \exp \biggl[ -\sqrt{N} z \ii\frac{\sqrt{3}}{2} -
\frac{5}{16} z^2 + o(1) \biggr].
\end{equation}
Now we are ready to prove Theorem~\ref{Theorem_ASM}.

\begin{pf*}{Proof of Theorem~\ref{Theorem_ASM}}
Choose $z_k$ and $z'_\ell$ to be related to $u_{i_k}$ and $v_{j_\ell}$,
respectively, in the same
way as $z$ was related to $u$ [through \eqref{eq_u_y} and \eqref
{eq_z_through_y}]. Then, combining
the asymptotics \eqref{eq_x24} with Corollary~\ref
{Corollary_multiplicativity_for_GUE}, we conclude
that the right-hand side of \eqref{eq_obs_6v_long} as $N\to\infty$ is
%
\begin{eqnarray}
\label{eq_x25} &&\prod_{k=1}^n \exp
\biggl[ -\sqrt{N} z_k \ii\frac
{\sqrt
{3}}{2} -\frac{5}{16}
z_k^2 + o(1) \biggr]
\nonumber
\\[-8pt]
\\[-8pt]
\nonumber
&&\qquad{}\times\prod_{\ell=1}^m
\exp \biggl[ -\sqrt{N} z'_k \ii \frac{\sqrt
{3}}{2} -
\frac{5}{16} \bigl(z'_k\bigr)^2 +
o(1) \biggr].
\end{eqnarray}

Now it is convenient to choose $z_i$ ($z_i'$) to be purely imaginary
$z_i=s_i\ii$ ($z'_i=s'_i\ii$).

Summing up the above discussion, observing that the case $n=0$, $m=1$
is almost the same as $n=1$,
$m=0$ (only the sign of $a_i$ changes) and that the observable \eqref
{eq_obs_6v_long} has a
multiplicative structure and the third (double) product in \eqref
{eq_obs_6v_long} is negligible
as $N\to\infty$, we conclude that as $N\to\infty$ for all real
$s_i$, $s_i'$
%
\begin{eqnarray}
\label{eq_6v_outcome}&& \lim_{N\to\infty}\mathbb E_N \exp
\Biggl[ \sum_{k=1}^{n} \frac{a_{i_k}-N/2}{\sqrt{N}}
s_k \ii+ \sum_{\ell=1}^{m}
\frac
{\widehat
a_{j_\ell}-N/2}{\sqrt{N}} s'_\ell\ii+o(1) \Biggr]
\nonumber
\\[-8pt]
\\[-8pt]
\nonumber
&&\qquad= \exp \Biggl[ -\frac{3}{16} \Biggl(\sum_{k=1}^n
s_k^2 + \sum_{\ell=1}^n
\bigl(s'_\ell\bigr)^2 \Biggr) \Biggr].
\nonumber
\end{eqnarray}

The remainder $o(1)$ on the left-hand side of \eqref{eq_6v_outcome} is
uniform in $a_{i_k}$, $\widehat
a_{i_\ell}$, and therefore, it can be omitted. Indeed, this follows from
\begin{eqnarray*}
&&\biggl\llvert \mathbb E_N\exp \biggl[ \frac{a_i-N/2}{\sqrt{N}} s\ii+o(1)
\biggr]-\mathbb E_N\exp \biggl[ \frac{a_i-N/2}{\sqrt{N}} s \ii \biggr]
\biggr\rrvert \\
&&\qquad\le\mathbb E_N \biggl\llvert \exp \biggl[
\frac{a_i-N/2}{\sqrt{N}} s\ii \biggr]\biggr\rrvert o(1)= o(1).
\end{eqnarray*}
Hence, \eqref{eq_6v_outcome} yields that the characteristic function of
the random vector
\[
\biggl( \frac{a_{i_1}-N/2}{\sqrt{N}},\ldots, \frac
{a_{i_n}-N/2}{\sqrt
{N}}, \frac{\widehat
a_{j_1}-N/2}{\sqrt{N}},
\ldots, \frac{\widehat a_{j_m}-N/2}{\sqrt
{N}} \biggr)
\]
converges as $N\to\infty$ to
\[
\exp \Biggl[ -\frac{3}{16} \Biggl(\sum_{k=1}^n
s_k^2 + \sum_{\ell=1}^n
\bigl(s'_\ell\bigr)^2 \Biggr) \Biggr].
\]
Since convergence of characteristic functions implies weak convergence
of distributions (see, e.g., \cite{Bi}, Section~26), the proof of
Theorem~\ref{Theorem_ASM} is complete.
\end{pf*}

\subsection{Toward dense loop model}\label{section:dense_loop_model}

In \cite{GNP} de Gier, Nienhuis and Ponsaing study the completely
packed $O(n=1)$ dense loop model
and introduce the following quantities related to the symplectic characters.

Following the notation from \cite{GNP} we set
\[
\tau_L(z_1,\ldots,z_L) =
\chi_{\lambda^L}\bigl(z_1^2,\ldots,z_L^2
\bigr),
\]
where $\lambda^L\in\GT^+_L$ is given by $\lambda_i^L = \lfloor
\frac
{L-i}{2}\rfloor$ for
$i=1,\ldots,L$. Further, set
%
\begin{eqnarray}
\label{def_uL}&& u_L(\zeta_1,\zeta_2;z_1,
\ldots,z_L)
\nonumber
\\[-8pt]
\\[-8pt]
\nonumber
&&\qquad = (-1)^L\ii\frac{\sqrt{3}}{2} \ln \biggl[
\frac{\tau_{L+1}(\zeta_1,z_1,\ldots,z_L)\tau
_{L+1}(\zeta
_2,z_1,\ldots,z_L)}{\tau_L(z_1,\ldots,z_L)\tau_{L+2}(\zeta_1,\zeta
_2,z_1,\ldots,z_L)} \biggr].
\end{eqnarray}
Define
\[
X_L^{(j)} = z_j \frac{\partial}{\partial z_j}
u_L(\zeta_1,\zeta _2;z_1,
\ldots,z_L)
\]
and
\[
Y_L = w \frac{\partial}{\partial w} u_{L+2}\bigl(
\zeta_1,\zeta _2;z_1,\ldots
,z_L,vq^{-1},w\bigr)\Big|_{w=v}.
\]
In particular, $X_L^{(j)}$ is a function of $z_1,\ldots,z_L$ and
$\zeta
_1,\zeta_2$, while $Y_L$
also depends on additional parameters $v$ and $q$.

De Gier, Nienhuis and Ponsaing show that $X_L^{(j)}$ and $Y_L$ are
related to the mean total
current in the $O(n=1)$ dense loop model, which is presented in
Section~\ref{Section_intro_loop}.
More precisely, they prove that under certain factorization assumption
and with an appropriate
choice of weights of configurations of the model, $X^{(j)}_L$ is the
mean total current between
two horizontally adjacent
points in the strip of width $L$,
\[
X^{(j)}_L=F^{(i,j),(i,j+1)},
\]
and $Y$ is the mean total current between two vertically adjacent
points in the strip of width $L$,
\[
Y_L = F^{(j,i),(j-1,i)};
\]
see \cite{GNP} for the details.

This connection motivated the question of the limit behavior of
$X^{(j)}_L$ and $Y_L^{(j)}$ as the width
$L$ tends to infinity; this was asked in \cite{G-MSRI,GP}. In
the present paper we compute
the asymptotic behavior of these two quantities in the homogeneous
case, that is, when $z_i=1$,
$i=1,\ldots,L$.

\begin{theorem}\label{theorem_dense_loop}
As $L\to\infty$ we have
\[
X_L^{(j)} |_{z_j=z; z_i=1, i\ne j} = \frac{\ii\sqrt{3}}{4L}
\bigl(z^3-z^{-3}\bigr) +o \biggl(\frac{1}{L}
\biggr)
\]
and
\[
Y_L |_{ z_i=1, i=1,\ldots,L} = \frac{\ii\sqrt{3}}{4L} \bigl(w^{3}-w^{-3}
\bigr) + o \biggl( \frac{1}{L} \biggr).
\]
\end{theorem}
\setcounter{remark}{0}
\begin{remark} When $z=1$, $X_L^{(j)}$ is identically zero, and so is
our asymptotics.
\end{remark}

\begin{remark}
The fully homogeneous case corresponds to
$w=e^{-\pi\ii
/6}$, $q=e^{2\pi\ii/3}$.
In this case,
\[
Y_L=\frac{\sqrt{3}}{2L} +o \biggl(\frac{1}{L} \biggr).
\]
\end{remark}

\begin{remark} The leading asymptotics terms do not depend on the
boundary parameters $\zeta_1$ and $\zeta_2$.
\end{remark}

The rest of this section is devoted to the proof of Theorem~\ref
{theorem_dense_loop}.

\begin{proposition}\label{prop:univariate_dense_loop}
The normalized symplectic character for $\lambda^L = \break ( \lfloor
\frac{L-1}{2}\rfloor, \lfloor\frac{L-2}{2}\rfloor,\ldots,1,0,0)$ is
asymptotically given for even
$L$ by
\[
\X_{\lambda^L}\bigl(e^y;L\bigr) = \frac{ 3e^{-({9}/4)y}(e^y-1)}{(e^{3/2y}-1)(e^y+1)} \biggl(
\frac{4}9 \frac{(e^{3/2y}-1)^2}{e^{y/2}(e^y-1)^2} \biggr)^L \biggl(1+
\frac
{t_1(y)}{L^{1/2}}+ \frac{t_2(y)}{L^{2/2}}+\cdots \biggr),
\]
and for odd $L$ by
\[
\X_{\lambda^L}\bigl(e^y;L\bigr) =
 \frac{ 3e^{-({9}/4)y}(e^y-1)}{(e^{3/2y}-1)(e^y+1)} \biggl(
\frac{4}9 \frac{(e^{3/2y}-1)^2}{e^{y/2}(e^y-1)^2} \biggr)^L \biggl(1+
\frac
{t'_1(y)}{L^{1/2}}+ \frac{t'_2(y)}{L^{2/2}}+\cdots \biggr),
\]
for some analytic functions $t_1, t_2, \ldots$ and $t_1',t_2',\ldots$
such that $t_1=t_1'$ and
\[
t'_2=t_2+\tfrac{1}{12}
\bigl(e^{3/2 y}-1\bigr)^2 e^{-3/2 y}.
\]
\end{proposition}
\begin{pf}
We will apply the formula from Proposition~\ref
{Proposition_Schur_Simplectic_1} to express the
normalized symplectic character as a normalized Schur function. The
corresponding $\nu$ is given
by $\nu_i = \lfloor\frac{L-i}{2}\rfloor+ 1$ for $i=1,\ldots,L$ and
$\nu_i = -\lfloor
\frac{i-L-1}{2}\rfloor$ for $i=L+1,\ldots,2L$, which is equivalent to
$\nu_i=\lfloor\frac{L-i}{2}
\rfloor+1$ for all $i=1,\ldots,2L$. We will apply Proposition~\ref{proposition_convergence_strongest} to directly derive the
asymptotics for
$S_{\nu}(e^y;\break 2L,1)$. For the specific signature we find that $f(t) =
\frac{1}4 -\frac{1}2 t$ and
\begin{eqnarray*}
\F(w;f) &=& \int_0^1 \ln\bigl(w- f(t)-1+t
\bigr)\,dt
\\
&= &\frac{1}6 \biggl(-6 + (5 - 4 w) \ln \biggl[-\frac{5}4 + w
\biggr] + (1 + 4 w) \ln \biggl[\frac{1}4 + w \biggr] \biggr).
\end{eqnarray*}
In particular, we have
\begin{eqnarray*}
\F'(w;f) &=& -\frac{2}3 \biggl(\ln \biggl[-
\frac{5}4 + w \biggr] - \ln \biggl[\frac{1}4 + w \biggr]
\biggr),
\\
\F''(w;f) &=& -\frac{1}{(w+{1}/4)(w-{5}/4)}.
\end{eqnarray*}
The root of $\F'(w;f)=y$, referred to as the critical point, is given
by
\[
w_0 =w_0(y) = \frac{1 + 5 e^{3 /2y}}{4 (-1 + e^{3/2 y})}.
\]
Example~\ref{ex2} of Section~\ref{subsection_complex_points}
shows that a steepest descent contour exists for any complex values of
$y$ for which $w_0 \notin[-1/4,5/4]$, that is, if
$e^{3/2y}$ is not a negative real number.
The values at $w_0$ are
\[
yw_0- \F(w_0;f) 
=
-\tfrac{1}4 y + \ln\bigl(e^{3/2y}-1\bigr) +1-\ln
\tfrac{3}2
\]
and
\[
\F''(w_0;f) = -\frac{4}9
\frac{(e^{3/2y}-1)^2}{e^{3/2y}}.
\]
In order to apply Proposition~\ref{proposition_convergence_strongest}
we need to ensure the convergence of $\Q(w;\nu,f)$, defined as in
Section~\ref{section:steepest_descent}
via
%
\begin{eqnarray}
\label{eq_Q_Q0} \ln\Q(w;\nu,f) &=& (2L)\F(w;f) - \sum_{j=1}^{2L}
\ln \biggl(w - \frac{\nu
_j+2L-j}{2L} \biggr)\nonumber
\\
&=&\underbrace{ \Biggl( 2L\F(w;f) -\sum_{j=1}^{2L}
\ln \biggl(w - \widehat {f} \biggl(\frac{j}{2L} \biggr) \biggr)
\Biggr)}_{P_1(w;\nu,f)}\\
&&{} - \underbrace{ \sum_{j=1}^{2L}
\ln \biggl(1 + \frac{f ({j}/{(2L)}
)-{\nu_j}/{(2L)}}{w-f ({j}/{(2L)} )-1+
{j}/{(2L)}} \biggr) }_{P_2(w;\nu,f)} .
\nonumber
\end{eqnarray}
As in \eqref{eq_x32}, we can write
\[
P_1(w;\nu,f)= \frac{\ln(w-\widehat f(0))-\ln(w-\widehat f(1))}{2}+\frac
{b(w)}{L}+o(1/L),
\]
where the exact value of $b(w)$ does not depend on the parity of $L$
and thus will not affect
the differences $t_1-t_1'$ and $t_2-t_2'$ in the statement.

We now estimate $P_2(w;\nu,f)$. We substitute the values for $\nu$ and
expand the logarithms as
$\ln(1+x)\approx x-x^2/2$. Let
%
\begin{equation}
\label{eq_x36} A(w;L):= -L\sum_{i=1}^{2L}
\biggl(\frac{ -{\nu_i}/{(2L)} +
f({i}/{(2L)})}{w-f({i}/{(2L)})
-1+{i}/{(2L)}} \biggr)^2
\end{equation}
be the second order term in this expansion, so that
\[
P_2(w;\nu,f) =\sum_{i=1}^{2L}
\frac{ -{\nu_i}/{(2L)} + f({i}/{(2L)})}
{w-f({i}/{(2L)}) -1+{i}/{(2L)}} +\frac{A(w;L)}{2L} +O\bigl(1/L^2\bigr).
\]
Approximating the last sum by integrals we have
%
\begin{eqnarray}
\label{eq_x30} &&\sum_{i=1, i\equiv L(\mathrm{mod}\ 2)}^{2L}
\frac{ -
{((L-i)/2+1)}/{(2L)} + {1}/4-{1}/2\cdot{i}/{(2L)}}{w-{1}/4+
{i}/{(4L)}-1+{i}/{(2L)}}\nonumber\\
&&\quad{} + \sum_{i=1, i\equiv L+1 (\mathrm{mod}\ 2)}^{2L}
\frac{ -
{((L-i)/2+1/2)}/{(2L)} + {1}/4-{1}/2\cdot{i}/{(2L)}}{w-{1}/4+
{i}/{(4L)}-1+{i}/{(2L)}}
\nonumber\\
&&\qquad=\sum_{i=1, i\equiv L(\mathrm{mod}\ 2)}^{2L} \frac{-{1}/{(2L)}
}{w-{5}/4+{3i}/{(4L)}}
\nonumber
\\[-8pt]
\\[-8pt]
\nonumber
&&\qquad\quad{}+
\sum_{i=1, i\equiv L+1 (\mathrm{mod}\ 2)}^{2L} \frac{ -{1}/{(4L)}
}{w-{5}/4+{3i}/{(4L)}}
\\
&&\qquad=\int_0^1 \frac{-{1}/2}{w-{5}/4+({3}/2) \eta}\,d\eta+\int
_0^1 \frac
{-{1}/4}{w-{5}/4+({3}/2) \eta}\,d\eta +
\frac{B(w;L)}{L}\nonumber
\\
&&\qquad= \frac{1}2 \ln \biggl( \frac{w-{5}/4}{w+{1}/4} \biggr)+\frac
{B(w;L)}{L}
,
\nonumber
\end{eqnarray}
where $B(w;L)$ is the error term in the approximation of the Riemann
sums by integrals. While both
functions $A(w;L)$ and $B(w;L)$
are bounded in $w$ and $L$, they could depend on the parity of $L$.
The sum in \eqref{eq_x36} can be again approximated by an integral
similarly to \eqref{eq_x30}; therefore for
both odd and even $L$, we have
\[
A(w;L)= \hat A(w) + O(1/L).
\]
Next, $B(w;L)$ appears when we approximate the integrals by their
Riemann sums. Using that the
trapezoid formula for the integral gives $O(1/L^2)$ approximation, and denoting
$v(x)=-\frac{1}{4(w-{5}/{4}+({3}/{2})x)}$, we have for even $L$
\[
B(w;L)=-v(0)+v \biggl(\frac{2L}{2L} \biggr) +O(1/L)= v(1)-v(0)+O(1/L)
\]
and for odd $L$,
\[
B(w;L)=-v(0)/2+v \biggl(\frac{2L}{2L} \biggr)\Big/2 +O(1/L)=
v(1)/2-v(0)/2+O(1/L).
\]
Therefore, we have
\[
A(w,L)+B(w,L) = \hat C(w) + (-1)^{L+1} \frac{1}{16} \biggl(
\frac
{1}{w-{5}/{4}}-\frac{1}{w+{1}/{4}} \biggr) +O(1/L),
\]
and hence we obtain as $L\to\infty$,
\begin{eqnarray*}
&&\exp\bigl( \Q(w;\nu,f) \bigr) \\
&&\qquad= \biggl( \frac{w-{5}/4}{w+{1}/4}
\biggr)^{{1}/2} \biggl(1+(-1)^{L+1} \frac{1}{16L} \biggl(
\frac{1}{w-{5}/{4}}-\frac
{1}{w+{1}/{4}} \biggr)+O\bigl(1/L^2\bigr)
\biggr)
\end{eqnarray*}
and
\begin{eqnarray*}
&&\exp \bigl( \Q(w_0;\nu,f) \bigr) \\
&&\qquad= \exp \biggl(-
\frac{3}4 y \biggr) \biggl(1+(-1)^{L+1} \frac{1}{24L}
\bigl(\bigl(e^{3/2 y}-1\bigr)^2e^{-3/2
y} \bigr)+O
\bigl(1/L^2\bigr) \biggr).
\end{eqnarray*}

Now combining Proposition~\ref{proposition_convergence_strongest} and
remark after it with the
expansion of $\Q$ and explicit values found above, we obtain
%
\begin{eqnarray}
&&S_{\nu}\bigl(e^y;2L,1\bigr) \nonumber\\
&&\qquad= \sqrt{ -
\frac{w_0 - f(0)-1}{\F''(w_0)(w_0-f(1))} }
 \biggl( \frac{w_0-{5}/4}{w_0+{1}/4} \biggr)^{{1}/2}
\frac{\exp{2L(yw_0-\F(w_0))}}{e^{2L}(e^y-1)^{2L-1}}
\nonumber\\
&&\qquad\quad{}\times \biggl(1+(-1)^{L+1} \frac{1}{16L} \biggl(
\frac{1}{w_0-{5}/{4}}-\frac{1}{w_0+{1}/{4}} \biggr) +\cdots \biggr) (1+\cdots)
\\
&&\qquad= \frac{ 3e^{-({9}/4)y}(e^y-1)}{2(e^{3/2y}-1)} \biggl( \frac{4}9 \frac{(e^{3/2y}-1)^2}{e^{y/2}(e^y-1)^2}
\biggr)^L\nonumber
\\
&&\qquad\quad{}\times \biggl(1+\hat t_1 L^{-1/2} + \biggl(\hat
t_2 +(-1)^{L+1} \frac{(e^{3/2 y}-1)^2e^{-3/2 y}}{12} \biggr)
L^{-1}+\cdots \biggr).
\nonumber
\end{eqnarray}
Proposition~\ref{Proposition_Schur_Simplectic_1} then immediately gives
$\X_{\lambda^L}(e^y;L,1)$
as $\frac{2}{e^y+1} S_{\nu}(e^y;2L,1)$.
\end{pf}
We will now proceed to derive the multivariate formulas needed to
compute $u_L$. First of all, set
$h(x)=\frac{4}9 x^{-3/2}(x^{3/2}-1)^{2}$, and define $\alpha_L(x)$
through
\[
\X_{\lambda^L}(x;L )= \alpha_L(x)\frac{x-1}{x+1}
h(x)^L \bigl(2-x-x^{-1} \bigr)^{-L},
\]
with $\lambda^L$ as in
Proposition~\ref{prop:univariate_dense_loop}.

Define
\begin{eqnarray*}
\widetilde\tau_L(z_1,\ldots,z_k)& =&
\frac{\chi_{\lambda^L}(z_1^2,\ldots,z_k^2,1^{L-k})}{\chi_{\lambda
^L}(1^L)}=\X_{\lambda^L}\bigl(z_1^2,
\ldots,z_k^2;L,1\bigr),
\\
\widetilde{u}_L(\zeta_1,\zeta_2;z_1,
\ldots,z_k) &=&(-1)^L\ii\frac
{\sqrt
{3}}{2} \ln \biggl[
\frac{ \widetilde\tau_{L+1}(\zeta_1,z_1,\ldots
,z_k)\widetilde\tau_{L+1}(\zeta_2,z_1,\ldots,z_k)}{\widetilde\tau
_L(z_1,\ldots,z_k)\widetilde\tau_{L+2}(\zeta_1,\zeta_2,z_1,\ldots
,z_k)} \biggr].
\end{eqnarray*}
Then $\widetilde{u}_L(\zeta_1,\zeta_2;z_1,\ldots,z_k)-u_L(\zeta
_1,\zeta
_2,z_1,\ldots,z_k)$ is a
constant, and thus we have
\begin{eqnarray*}
z_j\frac{\partial}{\partial z_j} \widetilde{u}(\zeta_1,\zeta
_2;z_1) &=& X_L^{(j)},
\\
w\frac{\partial}{\partial w} \widetilde{u}_{L+2}\bigl(\zeta_1,
\zeta _2;vq^{-1},w\bigr) \Big|_{v=w} &=&
Y_L.
\end{eqnarray*}
Therefore, we can work with $\X_{\lambda^L}$ instead of $\chi
_{\lambda
^L}$ and with
$\widetilde{u}$ instead of $u$.

For any function
$\xi$ and variables $v_1,\ldots,v_m$ we define
\[
B(v_1,\ldots,v_m;\xi):= \frac{\sum_{i=1}^m \xi(v_i) v_i
({\partial}/{(\partial v_i)})
\Delta(\xi(v_1)^2,\ldots,\xi(v_{m})^2)}{ \Delta(\xi(v_1)^2,\ldots
,\xi
(v_{m})^2) }.
\]


\begin{proposition}\label{Proposition_uL_ratio_general} Suppose that
signature $\lambda$ depends on a large parameter $L$
in such a way that
\[
\X_{\lambda}(x;L,1) = \alpha_L(x) h(x)^L
\frac{x-1}{x+1}\bigl(x+x^{-1}-2\bigr)^{-L},
\]
where
\begin{eqnarray*}
\alpha_L(x)&=&a(x) \bigl(1+b_1(x)L^{-1/2}+b_2(x)L^{-1}+
\cdots\bigr) \qquad\mbox{for even $L$},
\\
\alpha_L(x) &=& a(x) \bigl(1+b_1(x)L^{-1/2}+
\widehat{b}_2(x)L^{-1}+\cdots\bigr) \qquad\mbox{for odd $L$}
\end{eqnarray*}
and $a(x),b_1(x),b_2(x),\widehat{b}_2(x),h(x)$ are some analytic
functions of $x$. Let $\xi(x) =
x\frac{\partial}{\partial x}\ln( h(x) )$. Then for any $k$ we have
\begin{eqnarray*}
&&\ln \biggl[ \frac{ \X_{\lambda}(x_0,\ldots,x_k;L+1)\X_{\lambda
}(x_1,\ldots
,x_{k+1};L+1)}{\X_{\lambda}(x_1,\ldots,x_k;L)\X_{\lambda
}(x_0,\ldots
,x_{k+1};L+2)} \biggr]
\\
&&\qquad= c_1(x_0,x_{k+1};L) + \sum
_{i=1}^k 2 \bigl(\widehat {b}_2(x_i)-b_2(x_i)
\bigr)\frac{(-1)^L}{L}  
\\
&&\qquad\quad{}+\ln \biggl[ \bigl(\xi(x_{k+1})^2-\xi(x_0)^2
\bigr)+\frac{2}{L} \bigl(B(x_0,\ldots,x_k;\xi) \\
&&\hspace*{82pt}{}-
B(x_1,\ldots,x_{k+1};\xi) +c_2(x_0,x_{k+1})
\bigr) \biggr]
\\
&&\qquad\quad{}+ o\bigl(L^{-1}\bigr),
\end{eqnarray*}
where $c_0$ and $c_1$ are analytic functions not depending on
$x_1,\ldots,x_k$.
\end{proposition}

\begin{pf}
We use Theorem~\ref{Theorem_symp_multivar_1} to express the
multivariate normalized character in
terms of $\alpha_L(x_i)$ and $h(x_i)$ as follows:
%
\begin{eqnarray}
&&\frac{\X_{\lambda}(x_1,\ldots,x_m;N)}{\prod\X_{\lambda}(x_i;N)}
\nonumber
\\
&&\qquad= \prod_{j=0}^{m-1}\frac{(2N-2j-1)! N^{2j}}{(2N-1)!}
\frac{\prod_{i=1}^m (x_i-1)^{2m-1}(x_i+1) x_i^{-m}}{\Delta_s(x_1,\ldots,x_m)}\\
&&\hspace*{16pt}\qquad\quad{}\times M_N(x_1,\ldots,x_m),\nonumber
\end{eqnarray}
which is applied with $N=L,L+1,L+2$, $m=k,k+1,k+2$ and define for any
$N$ and $m$,
%
\begin{eqnarray}
\label{def:M_N}M_N(x_1,\ldots,x_m)&:=& \det
\biggl[ \frac{D_i^{2j-2} [ \alpha_N(x_i)h(x_i)^N
]}{N^{2j-2}\alpha_N(x_i)h(x_i)^N} \biggr]_{i,j=1}^m
\nonumber
\\[-8pt]
\\[-8pt]
\nonumber
&= &\frac{\Delta ({D_1^2}/{N^2},\ldots,
{D_m^2}/{N^2} )\prod_{i=1}^{m}
\alpha_N(x_i)h(x_i)^N}{\prod_{i=1}^{m} \alpha_N(x_i)h(x_i)^N},
\nonumber
\end{eqnarray}
where, as above, $D_i
=x_i\frac{\partial}{\partial x_i}$. The second form in \eqref{def:M_N}
will be useful later.

We can then rewrite the expression of interest as
\old{}{changed $\log$ to $\ln$ below}
%
\begin{eqnarray}
\label{u_L_2}&& \ln \biggl[ \frac{ \X_{\lambda}(x_0,\ldots,x_k;L+1)\X_{\lambda
}(x_1,\ldots,x_{k+1};L+1)}{
\X_{\lambda}(x_1,\ldots,x_k;L)\X_{\lambda}(x_0,\ldots,x_{k+1};L+2)} \biggr]
\nonumber\\
&&\qquad= \operatorname{const}_1(L) +\ln \biggl[ \frac{\X_{\lambda}(x_0;L+1)\X_{\lambda}(x_{k+1};L+1)}{\X_{\lambda
}(x_0;L+2)\X_{\lambda}(x_{k+1};L+2)} \biggr]
\nonumber
\\[-8pt]
\\[-8pt]
\nonumber
&&\qquad\quad{}+\ln \Biggl[\prod_{i=1}^k
\frac{\X_{\lambda}(x_i;L+1)^2}{\X_{\lambda}(x_i;L)\X_{\lambda
}(x_i;L+2)} \Biggr] - \ln \biggl[\frac
{(x_0-1)^2x_0^{-1}(x_{k+1}-1)^2x_{k+1}^{-1}}{x_0+x_0^{-1}-(x_{k+1}+x_{k+1}^{-1})} \biggr]
\\
&&\qquad\quad{}+\ln \frac{M_{L+1}(x_0,x_1,\ldots,x_k)M_{L+1}(x_1,\ldots
,x_{k+1})}{M_L(x_1,\ldots,x_k)M_{L+2}(x_0,\ldots,x_{k+1})},
\nonumber
\end{eqnarray}
where $\operatorname{const}_1(L)$ will be part of $c_1(x_0,x_{k+1};L)$.
We investigate each of the other terms separately.
First, we have that
\begin{eqnarray*}
&&\ln \biggl[ \frac{\X_{\lambda}(x_0;L+1)\X_{\lambda}(x_{k+1};L+1)}{\X_{\lambda
}(x_0;L+2)\X_{\lambda}(x_{k+1};L+2)} \biggr] +\ln \Biggl[\prod
_{i=1}^k \frac{\X_{\lambda}(x_i;L+1)^2}{\X_{\lambda}(x_i;L)\X_{\lambda
}(x_i;L+2)} \Biggr]
\\
&&\qquad= \sum_{i=1}^k
\ln \biggl( \frac{\alpha_{L+1}(x_i)^2}{\alpha
_L(x_i)\alpha
_{L+2}(x_i)} \biggr) + \ln \biggl( \frac{\alpha_{L+1}(x_0)\alpha
_{L+1}(x_{k+1})}{\alpha_{L+2}(x_0)\alpha_{L+2}(x_{k+1})} \biggr)
\\
&&\qquad\quad{}+\ln \biggl[\frac{
(x_0+x_0^{-1}-2)(x_{k+1}+x_{k+1}^{-1}-2)}{h(x_0)h(x_{k+1})} \biggr], 
\end{eqnarray*}
where the terms involving $x_0$ and $x_{k+1}$ are absorbed in $c_1$,
and we notice that
\[
\ln \biggl( \frac{ \alpha_{L+1}(x)^2}{\alpha_L(x)\alpha_{l+2}(x)} \biggr)= 2 \bigl(\widehat{b}_2(x)-b_2(x)
\bigr)\frac{(-1)^L}{L} + O \biggl(\frac
{1}{L^2} \biggr).
\]
%
Next we observe that for any $\ell$ and $N$,
\begin{eqnarray}
\label{derivative_expansion}&& \frac{ (x({\partial}/{(\partial x)}) )^\ell[\alpha
_N(x)h(x)^N]}{N^\ell\alpha_N(x)h(x)^N}
\nonumber
\\[-8pt]
\\[-8pt]
\nonumber
&&\qquad =\xi(x)^\ell+ \left(\pmatrix{\ell
\cr
2}q_1-\pmatrix{\ell
\cr
2} \xi(x)^\ell+ \ell
r_1\xi(x)^{\ell-1} \right)\frac{1}N + O
\bigl(N^{-3/2} \bigr),
\end{eqnarray}
where $q_1=\xi(x)(x\frac{\partial}{\partial x}\xi(x)+\xi(x)^2)$ and
$r_1(x)=x\frac{\partial}{\partial x}\log(a(x))$. In particular, since
$M_N$ is a polynomial in the
left-hand side of \eqref{derivative_expansion}, it is of the form
%
\begin{eqnarray}
\label{M_L_expansion}&& M_N(x_1,\ldots,x_m)
\nonumber
\\[-8pt]
\\[-8pt]
\nonumber
&&\qquad = \Delta
\bigl(\xi^2(x_1),\ldots,\xi ^2(x_m)
\bigr)+p_1(x_1,\ldots,x_m)\frac{1}N
+ O \bigl(N^{-3/2} \bigr)
\end{eqnarray}
for some function $p_1$ which depends only on $\xi$ and $a$. That is,
the second order asymptotics
of $M_N$ does not depend on the second order asymptotics of $\alpha_L$.
Further, we have
\[
\frac{M_N}{M_{N+1}} =1 +O \bigl(N^{-3/2} \bigr)
\]
for any $N$, so in formula \eqref{u_L_2} we can replace $M_{L+1}$ and
$M_{L+2}$ by $M_L$ without affecting the second order asymptotics.
Evaluation of $M$ directly will not lead to an easily analyzable
formula. Therefore we will do some simplifications and approximations
beforehand.

We will use Lewis Carroll's identity (Dodgson condensation), which
states that for any square matrix $A$ we have
\[
(\det A) (\det A_{1,2;1,2})=(\det A_{1;1}) (\det
A_{2;2}) - (\det A_{1;2}) (\det A_{2;1}),
\]
where $A_{I;J}$ denotes the submatrix of $A$ obtained by removing the
rows whose indices are in
$I$ and columns whose indices are in $J$. Applying this identity to the matrix
\[
A = \biggl[ \frac{D_i^{2j-2} [ \alpha_L(x_i)h(x_i)^L
]}{L^{2j}\alpha_L(x_i)h(x_i)^L} \biggr]_{i,j=0}^{k+1},
\]
we obtain
%
\begin{eqnarray}
\label{LewisCarrol_M_L} &&M_L(x_1,\ldots,x_k)M_L(x_0,x_1,
\ldots,x_k,x_{k+1})\nonumber
\\
&&\qquad= \det \biggl[ \frac{D_i^{2j}(\alpha_L(x_i)h^L(x_i))}{L^{2j}\alpha
_L(x_i)h(x_i)^L} \biggr]_{i=[1\dvtx k+1]}^{j=[0\dvtx k-1,k+1]} \det
\biggl[ \frac
{D_i^{2j}(\alpha_L(x_i)h^L(x_i))}{L^{2j}\alpha_L(x_i)h(x_i)^L} \biggr]_{i,j=0}^k
\nonumber
\\[-8pt]
\\[-8pt]
\nonumber
&&\qquad\quad{}- \det \biggl[ \frac{D_i^{2j}(\alpha_L(x_i)h^L(x_i))}{L^{2j}\alpha
_L(x_i)h(x_i)^N} \biggr]_{i=[0\dvtx k]}^{j=[0\dvtx k-1,k+1]}\\
&&\qquad\quad{}\times\det
\biggl[ \frac
{D_i^{2j}(\alpha_L(x_i)h^L(x_i))}{L^{2j}\alpha_L(x_i)h(x_i)^L} \biggr]_{i,j+1=1}^{k+1},
\nonumber
\end{eqnarray}
where $[0\dvtx k-1,k+1] =\{0,1,\ldots,k-1,k+1\}$.
The second factors in the two products on the right-hand side above
are just $M_L$ evaluated at the corresponding sets of variables. For
the first factors, applying
the alternate formula for $M_L$ from~\eqref{def:M_N} and using the
fact \new{that}
\[
\Delta(v_1,\ldots,v_m) \sum
_{i=1}^m v_i= \det \bigl[
v_i^j \bigr]^{j=[0\dvtx m-2,m]}_{i=[1\dvtx m]},
\]
we obtain
\begin{eqnarray*}
&&\det \biggl[ \frac{D_i^{2j}(\alpha_L(x_i)h(x_i)^L)}{L^{2j}\alpha
_L(x_i)h(x_i)^L} \biggr]_{i=[1\dvtx k+1]}^{j=[0\dvtx k-1,k+1]}
\\
&&\qquad= \frac{1}{\prod_{i=1}^{k+1} \alpha_L(x_i)h(x_i)^L} \det \bigl[ \bigl( D_i^2/L^2
\bigr)^j \bigr]_{i=[1\dvtx k+1]}^{j=[0\dvtx k-1,k+1]} \prod
_{i=1}^{k+1} \alpha_L(x_i)h(x_i)^L
\\
&&\qquad=\frac{1}{\prod_{i=1}^{k+1} \alpha_L(x_i)h(x_i)^L}\\
&&\qquad\quad{}\times \Biggl(\sum_{i=1}^{k+1}
D_i^2/L^2 \Biggr) \Delta
\bigl(D_1^2/L^2,\ldots,D_{k+1}^2/L^2
\bigr) \prod_{i=1}^{k+1}
\alpha_L(x_i)h(x_i)^L
\\
&&\qquad=\frac{1}{\prod_{i=1}^{k+1} \alpha_L(x_i)h(x_i)^L} \\
&&\qquad\quad{}\times\Biggl(\sum_{i=1}^{k+1}
D_i^2/L^2 \Biggr) \Biggl[ \Biggl(\prod
_{i=1}^{k+1} \alpha _L(x_i)h(x_i)^L
\Biggr) M_L(x_1,\ldots,x_{k+1}) \Biggr].
\end{eqnarray*}

Substituting these computations into \eqref{LewisCarrol_M_L} we get
%
\begin{eqnarray}
\label{TN_ratio} &&\frac{M_L(x_0,x_1,\ldots,x_k)M_L(x_1,\ldots
,x_k,x_{k+1})}{M_L(x_1,\ldots
,x_k)M_L(x_0,x_1,\ldots,x_k,x_{k+1})}\nonumber
\\
&&\qquad=\frac{(\sum_{i=1}^{k+1} {D_i^2}/{L^2}) [(\prod_{i=1}^{k+1}
\alpha_L(x_i)h(x_i)^L) M_L(x_1,\ldots,x_{k+1}) ]}{\prod_{i=1}^{k+1} \alpha_L(x_i)h(x_i)^L M_L(x_1,\ldots,x_{k+1})}
\nonumber\\
&&\qquad\quad{}- \frac{ (\sum_{i=0}^k {D_i^2}/{L^2}) [(\prod_{i=0}^k
\alpha
_L(x_i)h(x_i)^L) M_L(x_0,\ldots,x_k) ]}{\prod_{i=0}^k \alpha
_L(x_i)h(x_i)^LM_L(x_0,\ldots,x_k)}
\nonumber\\
&&\qquad=\frac{D_{k+1}^2\alpha_L(x_{k+1})h(x_{k+1})^L}{L^2\alpha
_L(x_{k+1})h(x_{k+1})^L} - \frac{D_{0}^2\alpha
_L(x_{0})h(x_{0})^L}{L^2\alpha_L(x_{0})h(x_{0})^L}
\\
&&\qquad\quad{}+\frac{(\sum_{i=1}^{k+1} D_i^2)[M_L(x_1,\ldots,x_{k+1})]}{L^2
M_L(x_1,\ldots,x_{k+1})}-\frac{(\sum_{i=0}^k D_i^2)[M_L(x_0,\ldots
,x_{k})]}{ L^2 M_L(x_0,\ldots,x_{k})}
\nonumber\\
&&\qquad\quad{}+2 \Biggl( \sum_{i=1}^{k+1}
\frac{D_i[\alpha_L(x_i)h(x_i)^L]}{L
\alpha_L(x_i)h(x_i)^L}\frac{D_iM_L(x_1,\ldots,x_{k+1})}{L
M_L(x_1,\ldots
,x_{k+1})} \nonumber\\
&&\hspace*{36pt}\qquad{}- \sum_{i=0}^{k}
\frac{D_i[\alpha_L(x_i)h(x_i)^L]}{L \alpha
_L(x_i)h(x_i)^L}\frac{
D_iM_L(x_0,\ldots,x_{k})}{L M_L(x_0,\ldots,x_{k})} \Biggr).
\nonumber
\end{eqnarray}
Using the expansion for $M_L$ from equation \eqref{M_L_expansion} and
the expansion from
\eqref{derivative_expansion}, we see that the only terms contributing
to the first two orders of
approximation in~\eqref{TN_ratio} above are
%
\begin{eqnarray}
&&\xi(x_{k+1})^2- \xi(x_0)^2 +
\frac{1}L \bigl(c_3(x_{k+1})-c_3(x_0)
\bigr)\nonumber\\
&&\qquad{} +
\frac{2}L \Biggl( \sum_{i=1}^{k+1}
\xi(x_i)\frac{D_i \Delta(\xi
(x_1)^2,\ldots,\xi(x_{k+1})^2)}{ \Delta(\xi(x_1)^2,\ldots,\xi
(x_{k+1})^2) }\\
&&\hspace*{29pt}\qquad{} - \sum_{i=0}^{k}
\xi(x_i)\frac{D_i \Delta(\xi
(x_0)^2,\ldots,\xi(x_{k})^2)}{ \Delta(\xi(x_0)^2,\ldots,\xi
(x_{k})^2) } \Biggr)
+ o \bigl(L^{-1} \bigr)
\nonumber
\end{eqnarray}
for some function $c_3$ not depending on $L$, so
$c_2(x_0,x_{k+1})=c_3(x_{k+1})-c_3(x_0)$. Substituting this result into
\eqref{u_L_2} we arrive at the desired formula.
\end{pf}
\begin{pf*}{Proof of Theorem~\ref{theorem_dense_loop}}
Proposition~\ref{Proposition_uL_ratio_general} with $x_0=\zeta_2^2$,
$x_{k+1}=\zeta_1^2$ and
$x_i=z_i^2$ shows that
%
\begin{eqnarray}
\label{eq_convergence_Gier}&& L \Biggl(\widetilde{u}_L(\zeta_1,\zeta
_2,z_1,\ldots,z_k) - c_1\bigl(
\zeta_1^2,\zeta_2^2;L\bigr) -\sum
_{i=1}^k 2 \bigl(\widehat
{b}_2(x_i)-b_2(x_i) \bigr)
\frac{(-1)^L}{L}\nonumber
\\
&&\qquad{}- \ln \biggl[ \bigl(\xi\bigl(\zeta_1^2
\bigr)^2-\xi\bigl(\zeta_2^2
\bigr)^2 \bigr)\\
&&\hspace*{52pt}{}+2\bigl(B(x_0,\ldots,x_k;\xi ) -
B(x_1,\ldots,x_{k+1};\xi) +c_2\bigl(
\zeta_1^2,\zeta_2^2\bigr) \bigr)
\frac{1}{L} \biggr] \Biggr)
\nonumber
\end{eqnarray}
converges uniformly to $0$, and so its derivatives also converge to $0$.
Proposition~\ref{prop:univariate_dense_loop} shows that in
our case,
\[
h(x) = \tfrac{4}9 x^{-3/2}\bigl(x^{3/2}-1
\bigr)^2
\]
and thus $\xi(x) = \frac{3}{2}\cdot\frac{x^{3/2}+1}{x^{3/2}-1}$.
Moreover, the function $\xi$ satisfies the following equation:
\[
x\frac{\partial}{\partial x}\xi(x) = -\frac{9}2\frac{
x^{3/2}}{(x^{3/2}-1)^2} = -
\frac{9}8\bigl(\xi(x)^2-1\bigr),
\]
and so we can simplify the function $B$ as a sum as follows:
%
\begin{eqnarray}
B(v_1,\ldots,v_m;\xi) &=& \frac{ \sum_i \xi(v_i) v_i( {\partial}/{(\partial v_i)}) \Delta(\xi(v_1)^2,\ldots,\xi(v_m)^2) }{\Delta(\xi
(v_1)^2,\ldots,\xi(v_m)^2)}\nonumber
\\
&=& \sum_i \sum_{j\neq i}
\xi(v_i)v_i \frac{ ({\partial}/{(\partial v_i)}) (\xi(v_i)^2- \xi(v_j)^2)}{ \xi(v_i)^2 - \xi(v_j)^2}\nonumber\\
& =& \sum
_i \sum_{j \neq i}
\frac{ 2\xi(v_i)^2 v_i ({\partial}/{(\partial v_i)})\xi(v_i)}{\xi(v_i)^2 -
\xi(v_j)^2}
\\
&=& \sum_i \sum_{j \neq i}
\frac{ -({9}/4) ( \xi(v_i)^4 -\xi
(v_i)^2) }{\xi
(v_i)^2 - \xi(v_j)^2 } \nonumber\\
&= &\sum_{i < j} \frac{ -({9}/4) ( \xi(v_i)^4 -\xi(v_i)^2 -\xi
(v_j)^4 +\xi
(v_j)^2 ) }{\xi(v_i)^2 - \xi(v_j)^2 }
\nonumber\\
&=& \sum_{i<j} -\frac{9}4 \bigl(
\xi(v_i)^2 + \xi(v_j)^2 -1\bigr)
\nonumber\\
&=& -\frac{9}4 (m-1) \Bigl(\sum\xi(v_i)^2
\Bigr) +\frac{9}4 \pmatrix{m
\cr
2}.
\nonumber
\end{eqnarray}
We thus have that
\[
B(x_0,\ldots,x_k;\xi) - B(x_1,
\ldots,x_{k+1};\xi) = -\tfrac{9}4k\bigl(\xi
(x_{k+1})^2 -\xi(x_0)^2\bigr),
\]
which does not depend on $x_1,\ldots,x_k$.

Differentiating \eqref{eq_convergence_Gier} we obtain the asymptotics
of ${X}_L^{(j)}$ as
\begin{eqnarray*}
{X}_L^{(j)}&=& \ii\frac{\sqrt{3}}{2}(-1)^L z
\frac{\partial
}{\partial z} 2 \bigl(\widehat{b}_2\bigl(z^2
\bigr)-b_2\bigl(z^2\bigr) \bigr)\frac{(-1)^L}{L}=\ii
\frac
{\sqrt
{3}}{2} z \frac{\partial}{\partial z} \biggl[\frac{1}{6}
\bigl(z^3-1\bigr)^2z^{-3} \biggr]
\\
&=&\ii\frac{\sqrt{3}}{4}\bigl(z^3-z^{-3}\bigr).
\end{eqnarray*}
For ${Y}_L^{(j)}$ the computations is the same.
\end{pf*}

\section{Representation-theoretic applications}\label{s:rep_theory}

\subsection{Approximation of characters of \texorpdfstring{$U(\infty)$}{U(infty)}}
\label{Section_U_infty}
In this section we give a new proof of Theorem~\ref{Theorem_U_VK}
presented in the \hyperref[sec1]{Introduction}.

Recall that a character of $U(\infty)$ is given by the function $\chi
(u_1,u_2,\ldots)$, which is
defined on sequences $u_i$ such that $u_i=1$ for all large enough $i$.
Also \mbox{$\chi(1,1,\ldots)=1$}.
By Theorem~\ref{Theorem_Voiculescu} extreme characters of $U(\infty)$
are parameterized by the points $\omega$ of the
infinite-dimensional domain
\[
\Omega\subset{\mathbb R}^{4\infty+2}={\mathbb R}^\infty\times {
\mathbb R}^\infty\times{\mathbb R}^\infty\times{\mathbb
R}^\infty\times{\mathbb R} \times{\mathbb R},
\]
where $\Omega$ is the set of sextuples
\[
\omega=\bigl(\alpha^+,\alpha^-,\beta^+,\beta^-;\delta^+,\delta^-\bigr)
\]
such that
\begin{eqnarray*}
&\alpha^\pm=\bigl(\alpha_1^\pm\ge
\alpha_2^\pm\ge\cdots\ge0\bigr)\in {\mathbb
R}^\infty,\qquad \beta^\pm=\bigl(\beta_1^\pm
\ge\beta_2^\pm\ge\cdots\ge0\bigr)\in {\mathbb
R}^\infty,&
\\
&\displaystyle\sum_{i=1}^\infty\bigl(
\alpha_i^\pm+\beta_i^\pm\bigr)
\le\delta^\pm,\qquad \beta_1^+ +\beta_1^-\le1.&
\end{eqnarray*}

Let $\mu$ be a Young diagram with the length of main diagonal $d$.
Recall that \emph{modified
Frobenius coordinates} are defined via
\[
p_i=\mu_i-i+1/2,\qquad q_i=
\mu'_i-i+1/2,\qquad i=1,\ldots,d.
\]
Note that $\sum_{i=1}^d p_i +q_i = |\mu|$.

Now let $\lambda\in\mathbb{GT}_N$ be a signature, and we associate two
Young diagrams $\lambda^+$ and
$\lambda^-$ to it, corresponding to the positive and negative entries
of $\lambda$, respectively: let $r =\max(i \dvtx\lambda_i\geq0)$, then
\[
\lambda^+ = (\lambda_1,\ldots,\lambda_r) \quad\mbox{and}\quad
\lambda^-=(-\lambda_N, -\lambda_{N-1},\ldots,-
\lambda_{r+1}).
\]
In this way we get two sets of modified Frobenius
coordinates, $p_i^+,q_i^+$, $i=1,\ldots,d^+$ and $p_i^-,q_i^-$,
$i=1,\ldots,d^-$.

\begin{proposition}
\label{Proposition_1d_VK_for_U} Suppose that $\lambda(N)\in\mathbb
{GT}_N$ is such a way that
\begin{eqnarray*}
&\displaystyle \frac{p_i^+}{N}\to\alpha_i^+, \qquad\frac{q_i^+}{N}\to\beta
_i^+, \qquad\frac{p_i^-}{N}\to \alpha_i^-,\qquad
\frac{q_i^-}{N}\to\beta_i^-,&
\\
&\displaystyle \frac{\sum_{i=1}^{d^+}{p_i^+ +q_i^+}}N\to\sum_{i=1}^\infty
\bigl(\alpha _i^++\beta_i^+\bigr) + \gamma^+,&\\
&\displaystyle\frac{\sum_{i=1}^{d^-}{p_i^- +q_i^-}}N\to\sum_{i=1}^\infty\bigl(
\alpha _i^-+\beta_i^-\bigr) + \gamma^-&
\end{eqnarray*}
then
\[
\lim_{N\to\infty} S_{\lambda(N)}(x;N,1)= \Phi_{\infty}
\biggl(\alpha,\beta ,\gamma; \frac{x}{x-1} \biggr),
\]
where
\begin{eqnarray*}
\Phi_{\infty} \biggl(\alpha,\beta,\gamma; \frac{x}{x-1} \biggr)&=&\exp
\bigl(\gamma^+(x-1)+\gamma^-\bigl(x^{-1}-1\bigr)\bigr)
\\
&&{}\times \prod_{i=1}^{\infty}
\frac{1+\beta_i^+(x-1)}{1-\alpha_i^+(x-1)} \cdot\frac{1+(1-\beta_i^-)(x-1)}{1+(1+\alpha_i^-)(x-1)}.
\end{eqnarray*}
The convergence is uniform over $1-\varepsilon<|x|<1+\varepsilon$ for
certain $\varepsilon>0$.
\end{proposition}
\setcounter{remark}{0}
\begin{remark} Note that
\[
\frac{1+(1-\beta_i^-)(x-1)}{1+(1+\alpha_i^-)(x-1)}=\frac{1+\beta
^-_i(x^{-1}-1)}{1-\alpha^-_i(x^{-1}-1)},
\]
which brings the function $\Phi_\infty$ into a more traditional form
of Theorems
\ref{Theorem_Voiculescu}, \ref{Theorem_U_VK}.
\end{remark}

\begin{remark} Our methods, in principle, allow us to give a full
description of the set on which
the convergence holds.
\end{remark}

\begin{pf*}{Proof of Proposition \ref{Proposition_1d_VK_for_U}}
The following combinatorial
identity is known (see, e.g., \cite{BO-newA}, (5.15), and references therein):
%
\begin{equation}
\label{eq_identity_BO} \prod_{i=1}^N
\frac{s+i -\lambda_i}{ s+i }= \prod_{i=1}^{d^+}
\frac{s+1/2 -
p_i^+}{s+1/2 + q_i^+} \prod_{i=1}^{d^-}
\frac{s+1/2+N + p_i^+}{ s+1/2+
N -q_i^-}.
\end{equation}
Introduce the following notation:
\[
\Phi_N\bigl(\lambda(N); w\bigr)= \prod_{i=1}^{d^+}
\frac{w-1 + {(1/2 + q_i^+)}/ N}{w-1 +
{(1/2 - p_i^+)}/N} \prod_{i=1}^{d^-}
\frac{ w-1+{(1/2+ N -q_i^-)}/N}{w-1 + {(1/2+N + p_i^+)}/N},
\]
and observe that \eqref{eq_identity_BO} implies that in the notation of
Section~\ref{section:steepest_descent} we have
%
\begin{equation}
\label{eq_x13} \prod_j \frac{1}{(w - \mu_j(N)/N)}=
\Phi_N\bigl(\lambda (N);w\bigr)\prod_i
\frac{1}{ w -
{(N-i)}/{N}}. 
\end{equation}
Then the integral formula for the Schur function (Theorem~\ref{Theorem_Integral_representation_Schur_1}) gives
\[
S_{\lambda(N);N,1}(x) = \frac{(N-1)!}{(x-1)^{N-1}} \frac{1}{2\pi\ii}\oint
\frac{x^z}{\prod_{i=1}^N(z - (N-i))} \Phi_N\bigl(\lambda(N);z/N\bigr) \,dz.
\]
We recognize in the integrand the setting of Proposition~\ref{proposition_convergence_strongest} with $f(t)=0$ for $t\in[0,1]$.
Thus, following the
notation of Proposition~\ref{proposition_convergence_strongest}, we denote
\[
\Q\bigl(w;\lambda(N),f\bigr) =\frac{ \exp( N\F(w;f))}{\prod_{i=1}^N(w -
{(N-i)}/{N})} \Phi_N\bigl(
\lambda(N);w\bigr).
\]
As $N\to\infty$ we have
%
\begin{equation}
\label{eq_x12} \Phi_N\bigl(\lambda(N);w\bigr)\to
\Phi_{\infty}(\alpha,\beta,\gamma;w).
\end{equation}
Further, we have that for $f\equiv0$, $\F(w;0)=w\ln(w)-(w-1)\ln
(w-1)-1$ and as $N\to\infty$
%
\begin{equation}
\label{eq_x11} \frac{ \exp( N\F(w;f))}{\prod_{i=1}^N(w - {(N-i)}/{N})} \to1.
\end{equation}
Combining \eqref{eq_x12} and \eqref{eq_x11} we conclude that
$\Q(w;\lambda(N),f) \to\Phi_{\infty}(\alpha,\beta,\gamma;w)$ as
$N\to
\infty$.
Now we can use
Propositions \ref{proposition_convergence_strongest} and \ref
{proposition_convergence_extended}
with the steepest descent contours of Example~\ref{ex1} of Section~\ref
{subsection_complex_points}. Recall
that here $f(t)=0$, $x=e^y$, $\F(w;0)=w\ln(w)-(w-1)\ln(w-1)-1$ and
$w_0=1/(1-e^{-y})$.

We conclude
that as $N\to\infty$,
\[
S_{\lambda(N)}\bigl(e^{y};N,1\bigr) = \frac{g(w_0)}{\sqrt{-\F''(w_0;f)}} \cdot
\frac{\exp (N(yw_0-\F(w_0;f)) )} {e^N
(e^{y}-1
)^{N-1}}\cdot \bigl(1+o(1) \bigr).
\]
Substituting $\F$, $w_0$, $g(w_0)=\Phi_{\infty}(\alpha,\beta
,\gamma
;w_0)$ and simplifying, we
arrive at
%
\begin{equation}
\label{eq_x35} S_{\lambda(N);N,1}(x) \to\Phi_{\infty} \biggl(\alpha ,\beta
,\gamma; \frac{x}{x-1} \biggr).
\end{equation}
Note that the convergence in \eqref{eq_x12} is uniform (on compact
subsets) outside the poles of
$\Phi_{\infty} (\alpha,\beta,\gamma; w )$, while the
convergence in \eqref{eq_x11} is
uniform over outside the interval $[0,1]$. Therefore, the convergence
in \eqref{eq_x35} is uniform
over compact subsets of
\[
\mathcal D=\bigl\{x=e^y\in\mathbb C \mid-\!\pi<\operatorname{Im}(y)<\pi, -
\varepsilon _2<\operatorname{Re}(y)<\varepsilon_2, y\ne0\bigr\}.
\]
(Here the small parameter $\varepsilon_2$ shrinks to zero as $\alpha
_1^\pm$ goes to
infinity.)

It remains to prove that this implies uniform convergence over
$1-\varepsilon<|x|<1+\varepsilon$.

Decompose
\[
S_{\lambda(N)}(x;N,1)=\sum_{k=-\infty}^{\infty}
c_k(N) x^k.
\]
Since $S_{\lambda(N)}$ is a polynomial, only finitely many coefficients
$c_k(N)$ are nonzero. The
coefficients $c_k(N)$ are \emph{nonnegative} (see, e.g.,
\cite{M}, Chapter I, Section~5), also
$\sum_k c_k(N)= S_{\lambda(N)}(1;N,1)=1$.

Since $\Phi_{\infty}(\alpha,\beta,\gamma;\frac{x}{x-1})$ is
analytic in
the neighborhood of the
unit circle, we can similarly decompose
\[
\Phi_{\infty} \biggl(\alpha,\beta,\gamma;\frac{x}{x-1} \biggr)=\sum
_{k=-\infty}^{\infty} c_k(\infty)
x^k.
\]

We claim that $\lim_{N\to\infty} c_k(N)=c_k(\infty)$. Indeed this
follows from the integral
representations
%
\begin{equation}
\label{eq_Fourier} c_k(N)=\frac{1}{2\pi\ii}\oint_{|z|=1}
S_{\lambda(N)}(z;N,1) z^{-k-1} \,dz,
\end{equation}
and similarly for $\Phi_\infty$. Pointwise convergence for all but
finitely many points of the
unit circle and the fact that $|S_{\lambda(N)}(z;N,1)|\le1$ for
$|z|=1$ implies that we can send
$N\to\infty$ in \eqref{eq_Fourier}.

Now take two positive real numbers $a$ and $b$, with
$\exp(-\varepsilon_2)<a<1<b<\exp(\varepsilon_2)$ such that
%
\begin{eqnarray}
\label{eq_x9} \lim_{N\to\infty}S_{\lambda(N)}(a;N,1)&=&
\Phi_{\infty} \biggl(\alpha,\beta ,\gamma;\frac{a}{a-1} \biggr),
\\
\label{eq_x29} \lim_{N\to\infty}S_{\lambda(N)}(b;N,1)&=&
\Phi_{\infty} \biggl(\alpha,\beta ,\gamma;\frac{b}{b-1} \biggr).
\end{eqnarray}
For $x$ satisfying $a\le|x|\le b$ and some positive integer $M$, write
%
\begin{eqnarray}
\label{eq_x10}&& \biggl\llvert S_{\lambda(N)}(x)-\Phi_{\infty} \biggl(
\alpha,\beta,\gamma ;\frac
{x}{x-1} \biggr)\biggr\rrvert\nonumber
\\
&&\qquad=\biggl\llvert \sum_k \bigl(c_k(N)-c_k(
\infty)\bigr)x^k\biggr\rrvert \le\sum_{k}\bigl|c_k(N)-c_k(
\infty )\bigr|\bigl(a^k+b^k\bigr)
\\
&&\qquad\le \sum_{k=-M}^M\bigl|c_k(N)-c_k(
\infty)\bigr|\bigl(a^k+b^k\bigr)+\sum
_{|k|>M} c_k(N) \bigl(a^k+b^k
\bigr)\nonumber\\
&&\qquad\quad{} +\sum_{|k|>M} c_k(\infty)
\bigl(a^k+b^k\bigr).
\nonumber
\end{eqnarray}
The third
term goes zero as $M\to\infty$ because the series $\sum_k c_k(\infty)
z^k$ converges for $z=a$ and
$z=b$. The second term goes to zero as $M\to\infty$ because of \eqref
{eq_x9}, \eqref{eq_x29} and
$c_k(N)\to c_k(\infty)$. Now for any $\delta$ we can choose $M$ such
that each of the last two
terms in \eqref{eq_x10} are less than $\delta/3$. Since $c_k(N)\to
c_k(\infty)$, the first term is
a less than $\delta/3$ for large enough $N$. Therefore, expression
\eqref{eq_x10} is less
than $\delta$, and the proof is complete.
\end{pf*}

Now applying Corollary~\ref{Corollary_multiplicativity_for_U} we arrive
at the following theorem.

\begin{theorem}[(cf. Theorem~\ref{Theorem_U_VK})]
In the settings of Proposition~\ref{Proposition_1d_VK_for_U} for any
$k$, we have
\[
\lim_{N\to\infty} S_{\lambda(N)}(x_1,
\ldots,x_k;N,1)= \prod_{\ell=1}^k
\Phi_{\infty} \biggl(\alpha,\beta,\gamma; \frac{x_\ell}{x_\ell-1} \biggr).
\]
The convergence is uniform over the set $1-\varepsilon<|x_\ell
|<1+\varepsilon$, $\ell=1,\ldots,k$
for certain $\varepsilon>0$.
\end{theorem}

Note that we can prove analogues of Theorem~\ref{Theorem_U_VK} for
infinite-dimensional symplectic
group $\operatorname{Sp}(\infty)$ and orthogonal group $O(\infty)$ in exactly the same
way as for $U(\infty)$.
Even the computations remain almost the same. This should be compared
to the analogy between the
argument based on binomial formulas of \cite{OkOlsh} for characters of
$U(\infty)$ (and their
Jack-deformation) and that of \cite{OkOlsh_BC} for characters
corresponding to other root series.

\subsection{Approximation of $q$-deformed characters of \texorpdfstring{$U(\infty)$}{U(infty)}}
\label{Section_U_q_infty}

In \cite{G-A} a $q$-\break deformation for the characters of $U(\infty)$
related to the notion of
quantum trace for quantum groups was proposed. One point of view on
this deformation is that we
\emph{define} characters of $U(\infty)$ through Theorem~\ref
{Theorem_U_VK}, that is, as all possible
limits of functions $S_{\lambda}$, and then \emph{deform} the function
$S_{\lambda(N)}$ keeping
the rest of the formulation the same. A ``good'' $q$-deformation of
turns out to be (see
\cite{G-A} for the details)
\[
\frac{s_{\lambda}(x_1,\ldots,x_k,q^{-k},\ldots,q^{1-N})}{s_\lambda
(1,q^{-1},\ldots,q^{1-N})}.
\]

Throughout this section we assume that $q$ is a real number satisfying
$0<q<1$. The next
proposition should be viewed as $q$-analogue of Proposition~\ref
{Proposition_1d_VK_for_U}.
%
\begin{proposition}
\label{prop_q_limit} Suppose that $\lambda(N)$ is such that $\lambda
_{N-j+1}\to\nu_j$ for every
$j$. Then
\[
\frac{s_{\lambda}(x,q^{-1},q^{-2},\ldots,q^{1-N})}{s_{\lambda
}(1,q^{-1},\ldots,q^{1-N})}\to F_\nu(x),
\]
%
\begin{equation}
\label{eq_limit_q_expression} F_\nu(x)=\prod_{j=0}^{\infty}
\frac{ (1-q^{j+1})}{(1-q^{j+1}x)} \frac
{\ln
(q)}{2\pi\ii} \int_{\mathcal C'}
\frac{ x^z}{\prod_{j=1}^{\infty}(1-q^{-z}q^{\nu
_j+j-1})} \,dz,
\end{equation}
where the contour of integration $\mathcal C'$ consists of two infinite
segments of $\operatorname{Im}(z)=\pm
\frac{\pi\ii}{\ln(q)}$ going to the right and vertical segment
$[-M(\mathcal C') -\frac{\pi
\ii}{\ln(q)},\break  -M(\mathcal C') +\frac{\pi\ii}{\ln(q)}]$ with arbitrary
$M(\mathcal C')<\nu_1$.
Convergence is uniform over $x$ belonging to compact subsets of
$\mathbb C\setminus\{0\}$.
\end{proposition}
\begin{remark*} Note that we can evaluate the integral in the definition
of $F_\nu(x)$ as the sum of
the residues
%
\begin{equation}
\label{eq_sum_expression} F_\nu(x)=\prod_{j=0}^{\infty}
\frac{ (1-q^{j+1})}{(1-q^{j+1}x)} \sum_{k=1}^{\infty}
\frac{ x^{\nu_k+k-1}}{\prod_{j\ne k}(1-q^{-\nu_k-k+1}q^{\nu
_j+j-1})} \,dz.
\end{equation}
The sum in \eqref{eq_sum_expression} is convergent for any $x$. Indeed,
the product over $j>k$ can
be bounded from above by $1/(q;q)_\infty$. The product over $j<k$
is [up to the factor
bounded by $(q;q)_\infty$]
\[
\prod_{j=1}^{k-1}q^{\nu_k+k-\nu_j-j}.
\]
Note that for any fixed $m$, if $k>k_0(m)$, then the last product is
less than $ q^{m(\nu_k+k-1)}
$. We conclude that the absolute value of $k$th term in \eqref
{eq_sum_expression} is bounded by
\[
|x|^{\nu_k+k-1} q^{m(\nu_k+k-1)} \frac{1}{ ( (q;q)_\infty
 )^2}.
\]
Choosing large enough $m$ and $k>k_0(m)$ we conclude that \eqref
{eq_sum_expression} converges.
\end{remark*}

\begin{pf*}{Proof of Proposition~\ref{prop_q_limit}}
We start from the formula of Theorem~\ref
{Theorem_Integral_representation_Schur_q},
%
\begin{eqnarray}
\label{eq_Schur_q_start} &&\frac{s_{\lambda}(x,1,q^{-1},\ldots,q^{2-N})}{s_{\lambda
}(1,q^{-1},\ldots,q^{1-N})}
\nonumber
\\[-8pt]
\\[-8pt]
\nonumber
&&\qquad=\frac{-\ln(q)}{2\pi\ii} \prod
_{i=0}^{N-2}\frac{
(q^{1-N}-q^{-i})}{(x-q^{-i})} \int
_{C} \frac{ (x/q)^z}{\prod_{j=1}^{N}(q^{-z} - q^{-\lambda
_j-N+j})} \,dz,
\end{eqnarray}
where the contour contains only the real poles $z=\lambda_j+N-j$; for
example, $C$ is the rectangle
through $M+\frac{\pi\ii}{\ln(q)}, M-\frac{\pi\ii}{\ln(q)}, -M
-\frac
{\pi\ii}{\ln(q)}, -M
+\frac{\pi\ii}{\ln(q)}$ for a sufficiently large $M$.

Since
\[
s_{\lambda}\bigl(x,1,q^{-1},\ldots,q^{2-N}
\bigr)=q^{|\lambda|} s_{\lambda
}\bigl(q^{-1}x,q^{-1},q^{-2},
\ldots,q^{1-N}\bigr),
\]
we may also write
%
\begin{eqnarray}
\label{eq_Schur_q_good} &&\frac{s_{\lambda}(x,q^{-1},q^{-2},\ldots,q^{1-N})}{s_{\lambda
}(1,q^{-1},\ldots,q^{1-N})} \nonumber\\
&&\qquad= -\prod_{i=0}^{N-2}
\frac{ (q^{1-N}-q^{-i})}{(qx-q^{-i})} \frac{\ln
(q)}{2\pi\ii} q^{-|\lambda|}
\nonumber
\\[-8pt]
\\[-8pt]
\nonumber
&&\qquad\quad{}\times\int
_{C} \frac{ x^z}{\prod_{j=1}^{N}(q^{-z} - q^{-\lambda_j-N+j})} \,dz
\\
&&\qquad=\prod_{i=0}^{N-2}\frac{ (1-q^{i+1})}{(1-q^{i+1}x)}
\frac{\ln
(q)}{2\pi
\ii} \int_{C} \frac{ x^z}{\prod_{j=1}^{N}(1-q^{-z}q^{\lambda_j+N-j})} \,dz.
\nonumber
\end{eqnarray}
Note that for large enough $N$ (compared to $x$), the integrand rapidly
decays as
$\operatorname{Re}(z)\to+\infty$. Therefore, we can deform the contour of integration
to be $\mathcal C'$ which
consists of two infinite segments of $\operatorname{Im}(z)=\pm\frac{\pi\ii}{\ln(q)}$
going to the right and
vertical segment $[-M(\mathcal C') -\frac{\pi\ii}{\ln(q)},
-M(\mathcal
C') +\frac{\pi
\ii}{\ln(q)}]$ with some $M(\mathcal C')$.

Note that the prefactor in \eqref{eq_Schur_q_good} converges as $N\to
\infty$. Let us study the
convergence of the integral. Clearly, the integrand converges to the
same integrand in $F_\nu(x)$.
Thus it remains only to check the contribution of infinite parts of
contours. But note that for
$z=s\pm\frac{\pi\ii}{\ln(q)}$, $s\in\mathbb R$, we have
\[
\frac{ x^z}{\prod_{j=1}^{N}(1-q^{-z}q^{\lambda_j+N-j})}=\frac{ x^z}{
\prod_{j=1}^{N}(1+q^{-s}q^{\lambda_j+N-j})}.
\]
Now the absolute value of each factor in denominator is greater than
$1$ and each factor rapidly
grows to infinity as $s\to\infty$. We conclude that the integrand in
\eqref{eq_Schur_q_good}
rapidly and uniformly in $N$ decays as $s\to+\infty$.

It remains to deal with the singularities of the prefactors in \eqref
{eq_Schur_q_good} and
\eqref{eq_limit_q_expression} at $x=q^{-i}$. But note that pre-limit
function is analytic in $x$
(indeed it is a polynomial), and for the analytic functions uniform
convergence on a contour
implies the convergence everywhere inside.
\end{pf*}

As a side effect we have proved the following analytic statement:
%
\begin{corollary}
The integral in \eqref{eq_limit_q_expression} and the sum in \eqref
{eq_sum_expression} vanish at
$x=q^{-i}$.
\end{corollary}

\begin{theorem}
\label{theorem_q_limit_multivar} Suppose that $\lambda(N)$ is such that
$\lambda_{N-j+1}\to\nu_j$
for every $j$. Then
\[
\frac{s_{\lambda(N)}(x_1,\ldots,x_k,q^{-k},q^{-k-1},\ldots
,q^{1-N})}{s_{\lambda(N)}(1,q^{-1},\ldots,q^{1-N})}\to F^{(k)}_\nu(x_1,
\ldots,x_k),\vspace*{-9pt}
\]
%
\begin{eqnarray}
\label{eq_limit_q_expression_multi} F^{(k)}_\nu(x_1,
\ldots,x_k)&=& \frac{(-1)^{{k\choose2}} q^{-2{k\choose3}} }{\Delta(x_1,\ldots,x_k)
\prod_i (x_iq^{k-1};q)_{\infty} }
\nonumber
\\[-8pt]
\\[-8pt]
\nonumber
&&{}\times
\det \bigl[ D_{i,q^{-1}}^{j-1} \bigr]_{i,j=1}^k
\prod_{i=1}^k F_{\nu
}
\bigl(x_iq^{k-1}\bigr) \bigl(xq^{k-1};q
\bigr)_{\infty}.
\end{eqnarray}
Convergence is uniform over each $x_i$ belonging to compact subsets of
$\mathbb C\setminus\{0\}$.
\end{theorem}
\begin{remark*} Formula \eqref{eq_limit_q_expression} should be
viewed as
a $q$-analogue of the
multiplicativity in the Voiculescu--Edrei theorem on characters of
$U(\infty)$ (Theorem~\ref{Theorem_Voiculescu}). There exists a natural linear
transformation, which restores the
multiplicitivity for $q$-characters; see \cite{G-A} for the details.
\end{remark*}
\begin{pf*}{Proof of Theorem~\ref{theorem_q_limit_multivar}}
Using Proposition~\ref{prop_q_limit} and Theorem~\ref
{Theorem_multivariate_Schur_q}, we get
\begin{eqnarray*}
F^{(k)}_\nu(x_1,\ldots,x_k) &=& \lim
_{N \to\infty} q^{-k|\lambda
(N)|}S_{\lambda}\bigl(q^kx_1,
\ldots,q^kx_k;N,q^{-1}\bigr)
\\
&=&\frac{q^{-{k+1\choose3} + (N-1){k\choose2}} \prod_{i=1}^k
[N-i]_{q^{-1}} !}{\prod_{i=1}^k\prod_{j=1}^{N-k}
(x_iq^k-q^{-j+1})}\\
&&{}\times
\frac{(-1)^{{k\choose2}}\det [
D_{i,q}^{j-1} ]_{i,j=1}^k}{q^{k {k\choose2}}
\Delta(x_1,\ldots,x_k)} \\
&&{}\times\prod_{i=1}^k
\frac{S_{\lambda}(x_iq^k;N,q^{-1}) \prod_{j=1}^{N-1}(x_iq^k-q^{-j+1})}{[N-1]_{q^{-1}}!}.
\end{eqnarray*}
In order to simplify this expression we observe that
\[
\frac{[N-i]_{q^{-1}}!}{[N-1]_{q^{-1}}!} 
\approx
q^{ N(i-1) -{i-1\choose2}},\qquad  N\to\infty.
\]
Also,
\[
\prod_{j=1}^{m}\bigl(xq^k
-q^{-j+1}\bigr) = 
(-1)^m q^{-{m\choose2}}
\bigl(xq^{k-1};q\bigr)_m.
\]
Last, we have
\[
\lim_{N\to\infty} q^{-|\lambda|}S_\lambda
\bigl(xq^k;N,q^{-1}\bigr) = F_{\nu
}
\bigl(q^{k-1}x\bigr).
\]
Substituting all of these into the formula above, we obtain
\begin{eqnarray*}
F^{(k)}_\nu(x_1,\ldots,x_k) &=& \lim
_{N\to\infty} \frac
{q^{-{k+1\choose3} + (N-1){k\choose2}} \prod_{i=1}^k q^{ N(i-1)
-{i-1\choose2}}
}{\prod_{i=1}^k(-1)^{N-k} q^{-{N-k\choose2}}
(xq^{k-1};q)_{N-k}}
\\
&&{}\times\frac{(-1)^{{k\choose2}}\det [ D_{i,q^{-1}}^{j-1}
]_{i,j=1}^k}{q^{k {k\choose2}} \Delta(x_1,\ldots,x_k)} \\
&&{}\times\prod_{i=1}^k
S_{\lambda}\bigl(x_iq^k;N,q^{-1}\bigr)
(-1)^{N-1} q^{-{N-1\choose
2}} \bigl(xq^{k-1};q
\bigr)_{N-1}
\\
&=& \frac{1}{q^{2{k\choose3}} \prod_i (x_iq^{k-1};q)_{\infty} }\frac
{(-1)^{{k\choose2}}\det [
D_{i,q^{-1}}^{j-1} ]_{i,j=1}^k}{ \Delta(x_1,\ldots,x_k)} \\
&&{}\times\prod_{i=1}^k
F_{\nu}\bigl(x_iq^{k-1}\bigr)
\bigl(xq^{k-1};q\bigr)_{\infty}. 
\end{eqnarray*}
\upqed\end{pf*}

\section*{Acknowledgments}
We would like to thank Jan de Gier for presenting the problem on the
asymptotics of the mean total
current in the $O(n=1)$ dense loop model, which led us to study the
asymptotics of normalized
Schur functions and beyond. We would also like to thank A.~Borodin,
Ph.~Di~Francesco, A.~Okounkov,
A.~Ponsaing, G.~Olshanski and L.~Petrov for many helpful discussions.
We thank the anonymous
referee whose very useful comments helped to improve the manuscript.

This project started while both authors were at MSRI (Berkeley) during
the Random Spatial
Processes program.

%





\printaddresses
\end{document}